\documentclass[11pt,e-only]{amsart}
\usepackage[utf8]{inputenc}
\usepackage[english]{babel}

\usepackage{amsmath, amsfonts, amssymb, amsthm,mathtools}
\newcommand{\defeq}{\coloneqq}
\usepackage{mathtools}
\usepackage{graphicx}
\usepackage{enumitem}
\usepackage{xcolor}
\usepackage{makecell}
\usepackage{stackengine}
\usepackage{longtable}
\usepackage{aliascnt}

\usepackage{bbold}

\usepackage{booktabs}

\usepackage{tikz}
\usetikzlibrary{braids,decorations.pathreplacing}
\usetikzlibrary{matrix,arrows.meta}

\usepackage{multirow}
\usepackage{array}
\usepackage[breaklinks=true]{hyperref}
\usepackage{xurl}
\hypersetup{
    colorlinks,
    linkcolor={blue!80!black},
    citecolor={blue!50!black},
    urlcolor={blue!80!black}
}

\usepackage[capitalize]{cleveref}
\usepackage[margin=3cm]{geometry}

\usepackage{caption}
\usepackage{setspace}
\usepackage{fancyhdr}
\usepackage{subcaption}
\usepackage[normalem]{ulem}
\usepackage[isbn=false,giveninits=true,date=year]{biblatex}
\usepackage{csquotes}
\renewbibmacro*{journal}{
  \iffieldundef{shortjournal}
    {
      \iffieldundef{journaltitle}
        {}
        {
          \printtext[journaltitle]
            {
              \printfield[titlecase]{journaltitle}
              \setunit{\subtitlepunct}
              \printfield[titlecase]{journalsubtitle}
             }
         }
    }
    {\printtext[journaltitle]{\printfield[titlecase]{shortjournal}}}
}

\AtEveryBibitem{
  \iffieldundef{doi}{}{\clearfield{url}}
}
\AtEveryBibitem{
  \iffieldundef{eprint}{}{\clearfield{url}\clearfield{urldate}}
}

\usepackage{mymacros}

\crefname{section}{§}{§§}
\Crefname{section}{§}{§§}

\stackMath

\title{Quantum invariants indexed by fibered faces of the Thurston polytope}

\author[D. Passaro]{Davide Passaro}

\author[L. San Mart\'in Su\'arez]{Lara San Mart\'in Su\'arez}

\thanks{California Institute of Technology, Pasadena, CA 91125, USA}
\thanks{Contact information: \texttt{dpassaro@caltech.edu, lsanmart@caltech.edu.}}

\date{\today}

\date{\today}

\begin{document}

\begin{abstract}
We study the Gukov--Manolescu quantum invariant for oriented links.
While this invariant is a single series in the knot case, we find that links admit multiple such series, each one consistent with the Melvin--Morton--Rozansky expansion of the colored Jones polynomials.
We prove that convergent inverted state sums yield a family of multivariable Gukov--Manolescu series, indexed by monomials of the Alexander polynomial. 
We conjecture that these monomials correspond precisely to the fibered faces of the Thurston norm ball and provide extensive computational evidence for this correspondence. 
Further results concerning the leading term, the corresponding single-variable invariant, and the effect of partial Dehn surgery are also established.
\end{abstract}

\maketitle

\section{Introduction}

The perturbative expansion of the colored Jones polynomial of a knot $K$ is known to satisfy
\begin{equation}
    J_{K,n}(1+\hbar) = \frac{1}{\Delta_K((1+\hbar)^n)}+\hbar \frac{P^{(1)}_K((1+\hbar)^n)}{\Delta_K^3((1+\hbar)^n)} + \hbar^2 \frac{P^{(2)}_K((1+\hbar)^n)}{\Delta_K^5((1+\hbar)^n)}+\cdots~.
\end{equation}
Here, $\Delta_K(x)$ denotes the Alexander polynomial of the knot $K$ and the numerators $P^{(j)}_K(x)\in\Z[x+x^{-1}]$ are polynomial invariants of $K$. The equality holds in $\Z[n][[\hbar]]$ \cite{MM,BNG,RozMMRConj}.

Setting $x=(1+\hbar)^n$, we call the series on the right-hand side
\begin{equation}
    \mathit{MMR}_K(x,\hbar)\defeq \sum_{j=0}^\infty \hbar^j\frac{P^{(j)}_K(x)}{\Delta^{2j+1}_K(x)}\in \Q(x)[[\hbar]]
\end{equation}
the Melvin--Morton--Rozansky (MMR) expansion of the knot $K$.
Gukov and Manolescu \cite[Conj. 1.5]{GM} conjectured that this perturbative series can be resummed as a power series at $x=0$. Concretely, they conjecture the existence of a knot invariant
\begin{equation*}
    F_K(x,q)\in (x^\frac12-x^{-\frac12})\Z((q))[[x]]
\end{equation*}
for which
\begin{equation}\label{eq:FK-MMR}
    F_K(x,e^\hbar)=(x^{\frac12}-x^{-\frac12})\mathit{MMR}_K(x,\hbar)\,,
\end{equation}
where equality holds as elements of $(x^\frac12-x^{-\frac12})\Q[[x]][[\hbar]]$.

Beyond its role as a resummation of the colored Jones polynomials, the $F_K$ series is closely related to the categorification program of quantum invariants of three-manifolds.
To that end, in \cite{GPV,GPPV}, Gukov, Pei, Putrov and Vafa introduced the $\widehat Z$ invariants of negative-definite plumbed manifolds, as a step towards the categorification of the Witten--Reshetikhin--Turaev invariants.
These invariants are expected to admit an extension to arbitrary three-manifolds.
From this perspective, the Gukov--Manolescu series can be understood as the $\widehat Z$ analogue for link complements.

A combinatorial construction for the $F_K(x,q)$ series, satisfying \eqref{eq:FK-MMR}, was developed by Park in \cite{Park21}.
There, he introduces the \emph{inverted state sum}, an infinite sum which depends on a braid representative of the knot $K$ and a choice of a simple multicycle on the braid (cf. \cref{fig:T24stcycles}) called an \emph{inversion datum}.
Whenever such a choice leads to an inverted state sum converging in $\Z[q,q^{-1}][[x]]$, we say that the knot $K$ is \emph{nice}, and that the resulting series is the $F_K$ invariant of the knot $K$ which satisfies \eqref{eq:FK-MMR}.

One of the major obstacles in computing $F_K(x,q)$ through  this construction is that it heavily depends on the choice of a suitable braid and inversion datum that would make the inverted state sum convergent. Given an arbitrary knot $K$, there is currently no algorithm to produce such data.

Nevertheless, despite its computational difficulty, it is known that all homogeneous braid knots \cite[Theorem 1]{Park21}, as well as all fibered knots up to 12 crossings \cite[Theorem 1.4]{OSSS25}, are nice.
In \cite{OSSS25} it was also shown that, for this class of knots,  the $F_K(x,q)$ satisfies
\begin{equation}
    F_K(x,q)=(-1)^{\lambda(K)+1}q^{g(K)-\lambda(K)}x^{g(K)-\frac12} + \text{higher order terms in }x~.
\end{equation}
Here $g(K)$ is the three-genus of the knot $K$ and $\lambda(K)$ is the Hopf invariant of a fibered knot $K$, originally defined as the \emph{enhanced Milnor number} by Rudolph \cite{Rudolph1987}.

Park's formalism only produces convergent series if the knot has monic Alexander polynomial;
therefore the equality of \eqref{eq:FK-MMR} holds in $(x^\frac12-x^{-\frac12})\Z[[x]][[\hbar]]$.
Together with extensive computational support, these observations have led to the following conjecture.
\begin{conjecture}[\cite{OSSS25}]\label{conj:OSSS-nice-fibered}
The class of nice knots and fibered knots is the same.
\end{conjecture}

\subsection{The main result}
While the Gukov--Manolescu series for knots has been explored at some length, results pertaining to the same invariant for links remain scarce.
With our work, we aim to start filling this gap in the literature.
The central focus of our discussion is summarized by the following:
\begin{question}
    Is there a multi-component link analogue of \eqref{eq:FK-MMR}?
\end{question}

To tackle this question, we propose an alternative definition of ``niceness'': Given a braid $\bd$ and an inversion datum $\iota$, we say that the pair $(\bd,\iota)$ is nice if the inverted state sum of Park can be written as
\begin{equation}
    \sum_{\mathbf{k}\in \Z^\ell} g_{\mathbf{k}}(q)\mathbf{x}^\mathbf{k}~,
\end{equation}
where $\mathbf{x}^\mathbf{k}\defeq \prod_{i=1}^{\ell}x_i^{k_i}$ and $g_{\mathbf{k}}(q)\in\Z[q,q^{-1}]$ is obtained from a finite number of contributing states. This alternative definition generalizes the one proposed by \cite{ParkThesis}, in that the exponents $\mathbf{k}$ are not restricted to take values in $\N^\ell$ and may take values in $\Z^{\ell}$.

The multi-component link analogue of the right-hand side of \eqref{eq:FK-MMR} was developed by Rozansky in \cite{Rozansky1}.
For an algebraically connected link, its $\boldsymbol{\alpha}=(\alpha_1,\dots,\alpha_\ell)$-colored Jones polynomial $J_{L;\boldsymbol{\alpha}}(q)$ agrees, perturbatively, with a signed sum of the iterated Laurent series expansions of series of the form
\begin{equation}\label{eq:MMR-link}
    \mathit{MMR}_L(\mathbf{x},\hbar)\defeq \sum_{j=0}^\infty \hbar^j\frac{P_L^{(j)}(\mathbf{x})}{\Delta_L^{2j+1}(\mathbf{x})}~,
\end{equation}
evaluated at $\mathbf{x}=((1+\hbar)^{\alpha_1},\dots,(1+\hbar)^{\alpha_\ell})$ (cf. \cite[Corollary 1.8]{Rozansky1}).

Unlike the knot case, the MMR expansion for multi-component links may have poles at $(x_1,\dots,x_\ell)=(0,\dots,0)$, in which case
\begin{equation}
    \mathit{MMR}_L(\mathbf{x},\hbar)\notin \Z[[x_1^\frac12,\dots,x_\ell^\frac12]][[\hbar]],
\end{equation}
which requires finding an extension of the power series ring in which the coefficients of the multi-component MMR expansion lie.

This extension, however, is not unique.
Following \cite{AK13}, we may define as many different extended rings $R$ as extremal vertices $v$ of the Newton polytope of the Alexander polynomial for which the term $\mathbf{x}^v$ is monic.
We call such vertices \emph{integral}.

Given an integral vertex $v$, let
\begin{equation*}
    \mathcal{C}_v\defeq \mathrm{Cone}\{w-v\mid w\in\supp \Delta_L, w\neq v\}
\end{equation*}
and
\begin{equation*}
    R_{\mathcal{C}_v}[[\mathbf{x}]]\defeq \{f=\sum_{\mathbf{k}\in\Z^\ell} a_\mathbf{k} \mathbf{x}^\mathbf{k} \mid a_\mathbf{k}\in R \text{ and }\supp{f}\subseteq \mathcal{C}_v\}
\end{equation*}
where $R$ is a ring.
Then, the polynomial $\mathbf{x}^{-v} \Delta_L(\mathbf x)$ is a unit in the ring $\Z_{\mathcal{C}_v}[[x_1,\dots,x_\ell]]$. Under identification of the denominators of \eqref{eq:MMR-link} with their inverses, we obtain a formal Laurent series expression for $\mathit{MMR}_L(\mathbf{x},\hbar)$. We call this series the \emph{expansion of} $\mathit{MMR}_L(\mathbf{x},\hbar)$ \emph{near $v$}.

Our main observation is that the inverted state sum naturally makes a choice of such a vertex $v$, and the result perturbatively agrees with the expansion of $\mathit{MMR}_L$ near $v$.

\begin{theorem}
Let $(\bd,\iota)$ be nice and let $F_{(\beta,\iota)}(\mathbf{x},q)$ denote the result of its inverted state sum.
Then, the pair $(\bd,\iota)$ determines a unique integral vertex $v$ such that
\begin{enumerate}
    \item The inverted state sum $F_L^{(2v)}(\mathbf{x},q)\defeq F_{(\beta,\iota)}(\mathbf{x},q)$ is an invariant of the pair $(L,v)$.
    \item The invariant $F_L^{(2v)}$ can be written as
    \begin{equation*}
        F_L^{(2v)}(\mathbf{x},q)=\mathbf{x}^{-v}\sum_{\mathbf{k}\in \mathcal{C}_v} f_\mathbf{k}(q)\mathbf{x}^\mathbf{k}\,,
    \end{equation*}
    where $f_\mathbf{k}(q)\in\Z[q,q^{-1}]$.
    \item $F_L^{(2v)}(\mathbf{x},\hbar+1)$ expanded near $\hbar=0$ recovers the expansion of $\mathit{MMR}_L(\mathbf{x},\hbar)$ near $v$.
\end{enumerate}
\end{theorem}

\subsection{Interpretation}
Surprisingly, the integral vertices associated to the Gukov--Manolescu series seem to have a topological meaning.

In \cite{Thurston}, Thurston showed that the cohomology classes in $H^1(S^3\setminus L,\Z)$ that represent fibrations of the link complement form open cones over top-dimensional faces of the Thurston polytope. Under duality, these distinguished faces (called \emph{fibered faces}) correspond to vertices in the dual Thurston polytope, which determine integral vertices by work of McMullen \cite{McMullen}.

We conjecture that the elements of the support of the Alexander polynomial that make $L$ a nice link are precisely the ones that correspond to fibered faces under duality.
\begin{conjecture} Let $L$ be an oriented link and $v$ an extremal vertex of the dual Thurston polytope. Then, $(L,2v)$ is nice if, and only if, $v$ is dual to the interior of a fibered face.
\end{conjecture}

For a knot, this reduces to \cref{conj:OSSS-nice-fibered}.
For a link, however, this statement is considerably richer.
A link may admit several fibered faces, each of which, we conjecture, gives an $F_L^{(v)}$ series. From this perspective, we regard ``niceness'' as a property of a pair $(L,v)$, where $L$ is an oriented link and $v$ is an  extremal vertex of the dual Thurston polytope. Therefore, the link analogue of the Gukov--Manolescu series is a family of invariants $\{F_L^{(v)}\}_v$, conjecturally indexed by the vertices associated to fibered faces of the Thurston polytope.

We provide computational support for this conjecture: for all links $L$ up to 10 crossings with $2$ or $3$ components, and for all vertices $v$ of the dual Thurston polytope associated to fibered faces, we find that $(L,2v)$ is a nice pair.
The check was made using several tools \cite{OrlandBraidsSoftware,FKCompute2026,lfhcompute} and we collect the associated $F_L^{(v)}$ series in \cite{topologyfyi}.

\subsection{The single-variable analogue} While the Gukov--Manolescu series of a link is naturally a series in the $\mathbf{x}=(x_1,\dots,x_\ell)$-variables, it still makes sense to consider its single-variable equivalent. In particular, whenever we set $\mathbf{x}=(x,\dots,x)$, we recover Park's stricter notion of niceness.
We have the following result for fibered links.
\begin{proposition} Let $L$ be a fibered link and $v$ be an extremal vertex of the multivariable Alexander polynomial associated to the canonical cohomology class $\mathbb{1}^*\in H^1(S^3\setminus L,\Z)$. Then, if $(L,v)$ is nice, the specialization $\mathbf{x}=(x,\dots,x)$ gives a well-defined power series satisfying
\begin{equation*}
    F_L(x,q)\defeq F_L^{(v)}((x,\dots,x),q)=\varepsilon_v q^{\alpha(L)} x^{-\frac{\chi(L)}{2}} + O(x^{-\frac{\chi(L)}{2}+1})\in\Z[q,q^{-1}][[x]]~.
\end{equation*}
If, moreover, $L$ is a homogeneous braid link, $\alpha(L)=\frac{b_1(L)}{2}-\lambda(L)$.
\end{proposition}

In general, we may still define the single-variable invariant $F_L(x,q)$ of a non-fibered link whose link exterior admits a fibration. In this case, we are no longer guaranteed that the result will have finite coefficients in $q$; however, we observe in examples that the specialization $\mathbf{x}=(x,\dots,x)$ still makes sense if $(L,v)$ is a nice pair and $v$ minimizes the pairing $\inner{\mathbb{1}^*}{v}$. The series we obtain is a Laurent power series in two variables, $x$ and $q$.

This observation, together with the analogous result for fibered knots, motivates the following detection conjecture for fibered links.
\begin{conjecture} Let $L$ be an oriented link. Then, $L$ is fibered if, and only if, there exists a unique integral vertex for which $(L,2v)$ is nice and the single variable $F_L(x,q)=F_L^{(v)}((x,\dots,x),q)$ is an element of the ring $\Z[q,q^{-1}][[x]]$.
\end{conjecture}

\subsection{Paper organization}

In \cref{sec:preliminaries}, we review the topological and combinatorial background, including the Thurston and Alexander norms, link Floer homology, random walks on braids, and the inverted state sum construction.
In \cref{sec:alexander-inverses}, we develop the theory for the cone-supported expansions of the inverse Alexander polynomial, associated to exposed vertices of its Newton polytope.
In \cref{sec:main-results}, we relate nice inversion data to positively exposed integral vertices and prove our main theorem. We then study their leading terms and single-variable specializations, and conclude by formulating the conjectures relating nice vertices to fibered faces of the Thurston norm ball.
In \cref{sec:worked-examples} we demonstrate our results on various links.
We conclude our exposition with \cref{sec:surgery}, in which we apply our results to Dehn surgeries on unknotted components of 2-component links.

\subsection*{Acknowledgements}
The authors are grateful to Sergei Gukov for his guidance and support in completing this project.
The authors would like to thank Paul Orland for offering advanced access to computational software which was used to produce nice braids.
The authors would further like to thank Josef Svoboda and Sunghyuk Park for their comments on the first draft of the paper.
The work of Davide Passaro is supported by a Sherman Fairchild Postdoctoral Fellowship sponsored by the Walter Burke Institute for Theoretical Physics.
The work of Lara San Mart\'in Su\'arez is supported by a fellowship from ``La Caixa'' Foundation (ID 100010434), with fellowship code LCF/BQ/EU23/12010094.

\section{Preliminaries}\label{sec:preliminaries}

Let $L=L_1\sqcup L_2\sqcup\cdots\sqcup L_\ell\subset S^3$ be an oriented, multi-component link of $\ell$ components. Let $\mathcal{N}(L)$ be a tubular neighbourhood of $L$. The link complement of $L\subset S^3$ is the manifold $M\defeq S^3\setminus \mathcal{N}(L)$ such that $\mathrm{int}(M)=S^3\setminus L$.

The meridians $\mu_1,\mu_2,\dots,\mu_\ell$ of the link $L$ determine a basis of $H_1(M,\Z)\cong \Z^\ell$ and a dual basis of $H^1(M,\Z)$. By Poincaré--Lefschetz duality,
\begin{equation*}
    H_2(M,\partial M;\Z)\cong H^1(M,\Z)\hspace{1.25em}\text{and}\hspace{1.25em} H^2(M,\partial M;\Z)\cong H_1(M,\Z)~.
\end{equation*}
Therefore, the properly embedded oriented surfaces in $M$ may be represented by cohomology classes via Poincaré duality, which we denote by PD.

Denote by $\AL$ the affine lattice over $H_1(S^3\setminus L,\Z)$ given by the elements $\sum_{i=1}^n a_i[\mu_i]$ with $a_i\in\Q$ satisfying $2a_i+lk(L,L_i)-1\in2\Z$. Analogously, denote by $\mathbb{H}=\HL$ the affine lattice over $H_1(S^3\setminus L,\Z)$ given by the elements $\sum_{i=1}^n a_i[\mu_i]$ with $a_i\in\Q$ satisfying $2a_i+lk(L,L_i)\in 2\Z$.

\subsection{Fibered classes and the link complement}

In \cite{Thurston}, Thurston introduced a semi-norm on compact and oriented three-manifolds with boundary.
For our purposes, we will restrict this exposition to the case where $M$ is a link complement as above.

Given a compact, connected surface $S$, we define the complexity of $S$ as
\begin{equation}
    \chi_-(S)\defeq\begin{cases}
        -\chi(S) &\text{ if } \chi(S)\leq 0, \\
        0 &\text{ otherwise}
    \end{cases}~.
\end{equation}
If $S = \bigsqcup_{i=1}^m S_i$ is disconnected, we set $\chi_-(S) \defeq \sum_{i=1}^{m}\chi_-(S_i)$.

The \emph{Thurston norm} is the semi-norm defined as
\begin{equation}
    \thurston{\phi} \defeq \inf_{\substack{(S, \partial S)\hookrightarrow (M, \partial M)\\ [S]=\operatorname{PD}[\phi]}}\chi_-(S)
\end{equation}
where $\phi\in H^{1}(M;\Z)\cong H_2(M,\partial M;\Z)$ and $S$ ranges over oriented embedded surfaces representing the Poincar\'e dual of $\phi$. This semi-norm is extended continuously to all of $H^1(M,\R)$.

Writing $\phi=(\phi_1,\dots,\phi_\ell)$ in the basis dual to the meridians, we can equivalently define $\thurston{\phi}$ as the minimal complexity of the embedded surfaces whose boundary wraps algebraically $\phi_i=\inner{\phi}{\mu_i}$ times around the component $L_i$ for $i=1,\dots,\ell$.

The dual of the Thurston norm, namely the \emph{dual Thurston norm}, is again a semi-norm on $H_1(M,\Z)\cong H^2(M,\partial M;\Z)$, defined as
\begin{equation}
    \thurstondual{v} \defeq \sup_{\{\phi\in H^1(M,\Z) \mid \thurston{\phi}=1} |\inner{\phi}{v}|.
\end{equation}
It is again extended continuously to all $H_1(M,\R)$.

The norms $\thurston{\cdot}$ and $\thurstondual{\cdot}$ are convex and linear on rays through the origin. Therefore, they are uniquely determined by their values on their unit balls
\begin{equation}
B_T\defeq\{\phi\in H^1(M,\Z) \mid \thurston{\phi}=1\} \hspace{0.75em}\text{and}\hspace{0.75em} B_T^\vee\defeq\{v\in H_1(M,\Z) \mid \thurstondual{v}=1\}~,
\end{equation}
which define convex polytopes on $\R^{\ell}$. We refer to $B_T$ and $B_T^\vee$ as the \emph{Thurston polytope} and the \emph{dual Thurston polytope}, respectively.

Assume that $M$ fibers over a circle, so that there exists a fibration $M\rightarrow S^1$. Each fiber is an embedded, oriented surface in $M$, and therefore determines, by Poincaré duality, a cohomology class $\phi\in H^1(M,\Z)$.
\begin{definition} A cohomology class $\phi\in H^1(M;\Z)$ is \emph{fibered} if PD[$\phi$] is represented by a fiber of a fibration $M\rightarrow S^1$.
\end{definition}

A fundamental theorem of Thurston \cite{Thurston} states that there exist distinguished, top-dimensional faces \[F_1,\dots,F_s\subseteq B_T\] such that $\phi\in H^1(M,\Z)$ is fibered if, and only if, the ray $\{\inner{c}{\phi}\mid c\in\R_{>0}\}$ that passes through $\phi$ intersects the Thurston polytope precisely at the interior of one of these top-dimensional faces. In other words, the normalized homology class $\frac{\phi}{\thurston{\phi}}$ belongs to $\mathrm{relint}(F_i)$ for some $i=1,\dots, s$.

It is a standard result in polytope theory that the relative interior of a top-dimensional face corresponds, under duality, to an extremal vertex in the dual polytope (cf. \cite{Ziegler1995}\footnote{In the book, the `duality' we refer to is called `polarity'.}). Therefore, there exist $v_1,\dots,v_s\in\mathrm{Vert}(B_T^\vee)$ such that $v_i$ is the dual face to $F_i$; that is, \[\arg\max_{\phi\in B_T}\inner{ \phi}{v_i} = F_i.\]
Under this identification, we say that $v_i$ is a \emph{fibered vertex}, and a cohomology class $\phi\in H^1(M,\Z)$ is fibered if, and only if, the equality
\begin{equation}
    \inner{ \phi}{v_i} = \thurston{\phi} \max_{\psi\in B_T} \inner{\psi}{v_i} = \chi_-(P.D.(\phi))
\end{equation}
holds for some $i=1,\dots,s$.

The usual notion of \emph{link fiberedness} is recovered by requiring the fiber to be a Seifert surface of $L$, so that its oriented boundary is $L$. Under Poincaré duality, it corresponds to the oriented cohomology class $\mathbb{1}^*\defeq (1,\dots,1)\in H^1(M,\Z)$.
\begin{definition} A link $L$ is \emph{fibered} if the oriented cohomology class $\mathbb{1}^*\in H^1(M,\Z)$ is a fibered cohomology class.
\end{definition}

\subsection{Random walks on braid closures}\label{sec:random-walks}

Let $\bd\in B_m$ be a braid of $n$ crossings. After identifying the labels of the incoming and outgoing ends of each strand, with the exception of the first strand which we leave open (see \cref{fig:openstrand}), fix a labeling of the resulting segments $\{a_i\}_{i=1}^{2n+1}$. At any crossing, we denote the incoming overstrand segment by $o$, the incoming understrand by $u$, and the corresponding outgoing strands by $o^+$ and $u^+$ according to the convention of \cref{fig:ouo+u+}.
\begin{figure}
    \centering
    \includegraphics[width=0.225\linewidth]{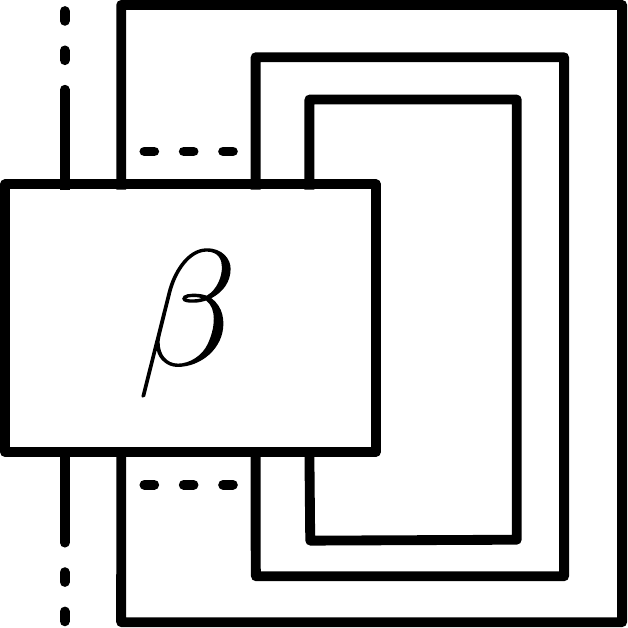}
    \caption{Incoming and outgoing label identification with an open strand.}
    \label{fig:openstrand}
\end{figure}

\begin{figure}
\centering
\includegraphics[width=0.13\textwidth]{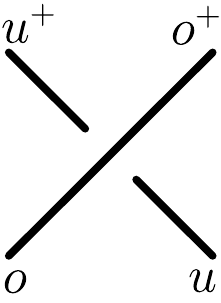}
\hspace{1em}
\includegraphics[width=0.13\textwidth]{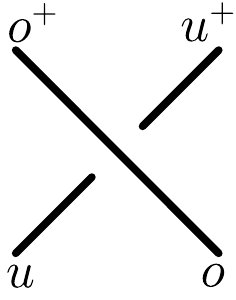}
\hspace{5em}
\includegraphics[width=0.13\textwidth]{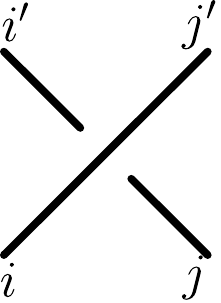}
\hspace{1em}
\includegraphics[width=0.13\textwidth]{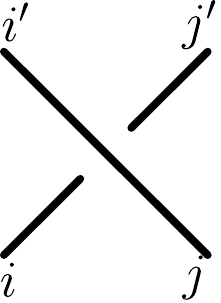}
\caption{Two labelling conventions used in this paper for segments adjacent to positive and negative crossings. The labels $\{o,\,u,\,o^+,\,u^+\}$ (left) highlight the overstrand and understrand of the crossing, while the labels $\{i,\,j,\,i',\,j'\}$ (right) are consistent with the literature on state sum models on braids.}
\label{fig:ouo+u+}
\end{figure}
Let $\mathcal{A}\in \mathrm{M}_{2n+1}(\Z[\mathbf{x}])$ be a matrix such that $\mathcal{A}_{r,k}\neq 0$ only when $r\in\{o,\,u\}$ and $k\in\{o^+,\,u^+\}$ for some crossing whose local labels are $\{o,\,u,\,o^+,\,u^+\}$ as in \cref{fig:ouo+u+}. Then, the matrix $\mathcal{A}$ defines a random walk\footnote{We use the notion of ``random walk'' loosely to mean a weighted directed walk.} on the braid $\bd$ as follows.

For every pair of segments $a_r,a_k$ with $\mathcal{A}_{r,k}\neq 0$, let $d_{r,k}$ denote the jump from $a_r$ to $a_k$. We impose the following commutation rule: $d_{r,k}$ and $d_{s,t}$ commute if and only if $r\neq s$. Two words in $\{d_{r,k}\}$ are identified if they differ by a finite sequence of exchanges of adjacent commuting jumps.

\begin{definition}
    A \emph{walk} on $\bd$ weighted by $\mathcal{A}$ is a path along the labeled segments of the braid such that, whenever it arrives at an incoming segment $r\in\{o,\,u\}$ of a crossing, it may continue to segment $k\in\{o^{+},\,u^{+}\}$ provided that $\mathcal{A}_{r,k}\neq 0$. Equivalently, a walk determines a word in the jumps $\{d_{r,k}\}$. We say a walk is \emph{primitive} if it passes through each segment at most once.

    A \emph{cycle} on $\bd$ weighted by $\mathcal{A}$ is a walk with the same starting and ending point. A \emph{based path} is a walk beginning at the incoming end of the first strand and terminating at its outgoing end. We say a based path or a cycle is \emph{primitive} if it is primitive as a walk.

    We say two cycles commute with each other if they are disjoint.

    A \emph{multicycle} $c=(c_1,\dots,c_s)$ is an ordered tuple of primitive cycles, modulo exchanges of adjacent commuting cycles.
    A multicycle is \emph{simple} if all of its components are pairwise disjoint.
\end{definition}
Let $\mathcal{Q}$ denote the set of simple multicycles.
Every simple multicycle $c\in\mathcal{Q}$ determines a permutation $\sigma_c \in S_{2n+1}$: For each $i=1,\dots,2n+1$, set
\begin{equation}
    \sigma_c(r) = \begin{cases}
        r & \text{if $c$ does not pass through $a_r$},\\
        k & \text{if $c$ travels from $a_r$ to $a_k$}, \\
    \end{cases}~.
\end{equation}
\begin{definition}The \emph{weight} of a multicycle $c$, denoted by $\weight{c}$, is given by
\begin{equation}
    \weight{c} = \prod_{d_{r,k}\text{ in } c} \mathcal{A}_{r,k}.
\end{equation}
Equivalently, for a simple multicycle $c$,
\begin{equation}
    \weight{c} = \prod_{\substack{1\leq i \leq 2n+1 \\ \sigma_c(i)\neq i} }\mathcal{A}_{i,\sigma_c(i)}.
\end{equation}
\end{definition}

Given the above, we can reorganize the Leibniz expansion of the determinant of $I-\mathcal{A}$ as a sum over simple multicycles.
More explicitly,
\begin{equation}\label{eq:transition-matrix-det}
    \begin{aligned}
        \det(I-\mathcal{A}) &= \sum_{\sigma\in S_{2n+1}} \sgn{\sigma} \prod_{i=1}^{2n+1}\left(\delta_{i,\sigma(i)}-\mathcal{A}_{i,\sigma(i)}\right)= \sum_{c\in \mathcal{Q}}(-1)^{|c|}\weight{c},
    \end{aligned}
\end{equation}
where $|c|$ denotes the number of closed components of $c$.

\subsection{The multivariable Alexander polynomial}\label{sec:multialexpoly}

Let $G=\pi_1(M)$ denote the link group.
Abelianization gives
\[
G\longrightarrow G/G'\cong H_1(M;\Z)\cong\Z^\ell,
\]
where the class of $\mu_i$ is the $i$th standard basis element.
The \emph{universal abelian cover} $\widetilde M\to M$ is the cover corresponding to the kernel of this homomorphism, namely the commutator subgroup $G'$.
Thus its sheets are indexed by $G/G'$.

The group $G/G'$ acts on the cover by relabeling these sheets, and hence acts on its cellular chains and on its homology.
After choosing the meridian basis, its integral group ring is
\[
\Z[G/G']\cong\Lambda_\ell=\Z[x_1^{\pm1},\ldots,x_\ell^{\pm1}],
\]
where $x_i$ denotes the class of $\mu_i$.
It follows that $H_1(\widetilde M;\Z)$ is naturally a $\Lambda_\ell$-module.
This is the Crowell--Fox construction of the Alexander module \cite{CrowellFox1977}.

Equivalently, a monomial $x_1^{a_1}\cdots x_\ell^{a_\ell}$ records the homology class $a_1[\mu_1]+\cdots+a_\ell[\mu_\ell]$ of a loop in the complement.
The separate exponents therefore retain the winding data around the individual components.

\begin{definition}
The \emph{Alexander module} of $L$ is the $\Lambda_\ell$-module $H_1(\widetilde M;\Z)$.
The \emph{multivariable Alexander polynomial}
\[
\Delta_L(x_1,\ldots,x_\ell)\in\Lambda_\ell
\]
is the greatest common divisor of its first elementary ideal.
It is defined only up to multiplication by a unit $\pm x_1^{a_1}\cdots x_\ell^{a_\ell}$.
If that elementary ideal is zero, we set $\Delta_L\equiv0$.
\end{definition}
The one-variable Alexander polynomial is defined analogously from the infinite cyclic cover corresponding to the kernel of
\begin{equation*}
    \pi_1(M)\to H_1(M;\Z) \cong \Z^{\ell}\to \Z\quad [\mu_i]\to1
\end{equation*}
recording the total winding number around the link components.

To fix a convention for the multivariable Alexander polynomial, we follow \cite{Rozansky98}. In particular,
\begin{equation}\label{eq:symmetry-Alex}
    \Delta_L(x_1^{-1},\ldots,x_\ell^{-1})
= (-1)^\ell \Delta_L(x_1,\ldots,x_\ell)~,
\end{equation}
and
\begin{equation}\label{eq:Hopf-Alex}
    \Delta_{L}(x,y)=1 \text{ if } L \text{ is the positive Hopf link}~.
\end{equation}
It follows that the support of the Alexander polynomial lives in $\AL$. This is the convention that we will use throughout the paper.

Let $L^{\mathrm{rev}}$ denote the link obtained by reversing every component, let $L^{(i)}$ denote the result of reversing only $L_i$, and let $\mirror{L}$ denote the mirror of $L$.

\begin{proposition}\label{prop:Alex-properties} Given a link $L$, let $\Delta_L(\mathbf{x})\in\Z[x_1^{\pm \frac12},\dots,x_\ell^{\pm \frac12}]$ denote the multivariable Alexander polynomial of $L$, normalized as in \eqref{eq:symmetry-Alex} and \eqref{eq:Hopf-Alex}. Then, $\Delta_L$ satisfies the following properties:
\begin{enumerate}
\item $\Delta_{L^{(i)}}(x_1,\ldots,x_\ell)=-\Delta_L(x_1,\ldots,x_i^{-1},\ldots,x_\ell)$.

\item $\Delta_{L^{\mathrm{rev}}}(x_1,\dots,x_\ell) = (-1)^{\ell}\Delta_L(x_1,\dots,x_\ell).$

\item $\Delta_{\mirror{L}}(x_1,\dots,x_\ell) = -\Delta_{L}(x_1^{-1},\dots,x_\ell^{-1}) = (-1)^{\ell+1}\Delta_L(x_1,\dots,x_\ell).$

\item{\emph{(Torres formula)}}  Let $L'=L\setminus L_\ell$ and $a_i=\operatorname{lk}(L_i,L_\ell)$.
If $\ell\geq2$, then
\begin{equation*}\label{eq:torres-multivariable}
\Delta_L(x_1,\ldots,x_{\ell-1},1)
=\begin{cases}
\left(\prod_{i=1}^{\ell-1}x_i^{\frac{a_i}{2}}-\prod_{i=1}^{\ell-1} x_i^{-\frac{a_i}{2}}\right)\Delta_{L'}(x_1,\ldots,x_{\ell-1}) & \ell > 2 \\
\frac{x_1^{\frac{a_1}{2}}- x_1^{- \frac{a_1}{2}}}{x_1^{\frac12}- x_1^{- {\frac12}}}\Delta_{L_1}(x_1) & \ell = 2
\end{cases}
\end{equation*}

\item{\emph{(Single-variable specialization)}}
\begin{equation}\label{eq:alexander-one-variable-specialization}
\Delta_L^{(1)}(x) =
\begin{cases}
\Delta_L(x),&\ell=1,\\
(x^{\frac12}-x^{-\frac12})\Delta_L(x,\ldots,x),&\ell>1.
\end{cases}
\end{equation}

\end{enumerate}
\end{proposition}

The specialization in \eqref{eq:alexander-one-variable-specialization} is a useful warning: the one-variable polynomial sees only the total exponent of a monomial, whereas the multivariable polynomial records its full exponent vector.
Distinct multivariable terms can therefore coincide, or even cancel, after the specialization.

\subsubsection{Wirtinger presentation of the link group}\label{sec:writinger}

The Wirtinger presentation is a presentation of $\pi_1(L)$ that can be constructed from a link diagram.

\begin{figure}
    \centering
    \includegraphics[width=0.25\linewidth]{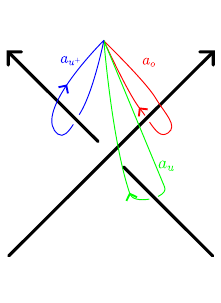}
    \caption{Generators for the upper Wirtinger presentation near a crossing.
    Each colored loop is a positively oriented meridian of the corresponding arc: $a_o$, in red, surrounds the overpassing arc, while $a_u$, in green, and $a_{u^{+}}$, in blue, surround the incoming and outgoing underpassing arcs respectively.}
    \label{fig:writinger}
\end{figure}

The \emph{upper} Wirtinger presentation is generated by the set of variables $\{a_i\}_{i=0}^{2n}$ and the relations
\begin{equation}
    a_{u^+} = a_o^{-1} a_u a_o, \hspace{2em} a_{o^+}=a_o
\end{equation}
for every set of generators $\{a_{o},a_{u},a_{u^+}\}$. See \cref{fig:writinger}.

The \emph{lower} Wirtinger presentation is defined as the upper Wirtinger presentation applied to the mirror of the link diagram, where we swap every overstrand for an understrand and vice versa. Equivalently, it is generated by $\{b_i\}_{i=0}^{2n}$ and the relations
\begin{equation}
    b_{o^+} = b_u^{-1} b_o b_u, \hspace{2em} b_{u^+} = b_u
\end{equation}
for every set of labels $\{o,u,o^+,u^+\}$ related in the same way.

\subsubsection{Fox free-derivative calculus}\label{sec:Fox-free-derivatives}

Given a presentation of the link group \[\pi_1(M)=\langle a_1,\dots,a_n\mid r_1,\dots,r_m\rangle,\] Fox differentiation gives the matrix
\[
  \left(\frac{\partial r_k}{\partial a_j}\right)_{k,j}
  \quad\text{over}\quad \Z[F_n],
\]
where the Fox derivatives are determined by
\[
  \frac{\partial a_i}{\partial a_j}=\delta_{ij},
  \qquad
  \frac{\partial(uv)}{\partial a_j}
  =\frac{\partial u}{\partial a_j}
   +u\frac{\partial v}{\partial a_j},
  \qquad
  \frac{\partial a_i^{-1}}{\partial a_j}
  =-a_i^{-1}\delta_{ij}.
\]
Abelianizing this matrix sends the generator associated to an arc on $L_i$ to $x_i$ and produces a matrix over $\Lambda_\ell$.
After the usual deletion of one redundant relation, the matrix is a presentation of the Alexander module.
By Fox's fundamental formula
\begin{equation*}
   r_k-1=\sum_{j=1}^{n}\frac{\partial r_k}{\partial{a_j}}(a_j-1)~.
\end{equation*}
After imposing the relations $r_k=1$ and abelianizing, the Fox matrix annihilates the nonzero vector with $j$-th entry $x_{\kappa(j)}-1$.
Therefore
\begin{equation}\label{eq:detiszer}
    \det \left[\frac{\partial r_k}{\partial a_j}\right]=0~.
\end{equation}
The corresponding maximal minors generate its first elementary ideal, with greatest common divisor $\Delta_L$, up to a unit.

\subsubsection{Gassner representation}\label{sec:Gassner}

Given a braid $\bd$, consider the link obtained from its closure. Enumerate each of the segments in the resulting diagram in the order that they appear in the braid. Denote by $\phi$ the function that sends every segment to the index of the component it belongs to. We construct a square matrix $A=A(\beta)\in M_{2n}(\Z[x_1,\dots,x_n])$ of size $2n$, where $n$ is the number of crossings of $\bd$, in the following way: at each crossing, denote by $o$ and $u$ the labels of the incoming overstrand and understrand, and by $o^+$ and $u^+$ the labels of the outgoing overstrand and understrand, respectively, as in \cref{fig:ouo+u+}. Let $s\in\{\pm1\}$ denote the sign of the crossing. Then, the values of $A$ in the $o$ and $u$ rows are given by the following rules:
\begin{equation}\label{eq:A-rules}
\begin{split}
    &A_{o,o^+}=\begin{cases}
        x_{\kappa(u)} &\quad \text{if } s=+1 \\
        \frac{1}{x_{\kappa(u)}} &\quad \text{if } s=-1 \\
    \end{cases}
    \, , \hspace{1em}
    A_{o,u^+}=\begin{cases}
        1-x_{\kappa(o)} &\quad \text{if } s=+1 \\
        \frac{x_{\kappa(o)}}{x_{\kappa(u)}}\left(1-\frac{1}{x_{\kappa(o)}}\right) &\quad \text{if } s=-1 \\
    \end{cases}
    \, , \hspace{1em}
    A_{u,o^+}=0
    \, , \\&
    A_{u,u^+}=1
    \, , \hspace{1em}\text{and}\hspace{1em}
    A_{o,k}=A_{u,k}=0 \quad \text{if } k\neq o^+,u^+.
    \end{split}
\end{equation}

This matrix, also called the \emph{Gassner matrix}, is the result of applying the free Fox derivative calculus of \cref{sec:Fox-free-derivatives} to the
lower Wirtinger presentation of \cref{sec:writinger}, derived from a braid representative of the link.

For any two components $L_i,L_j$ of the closure of $\bd$, define
\[
\operatorname{lk}_{\bd}(L_i,L_j)=
\begin{cases}
\operatorname{lk}(L_i,L_j), & i\neq j,\\
w(\bd_i), & i=j,
\end{cases}
\]
where $\operatorname{lk}(L_i,L_j)$ is the usual linking number—half the signed number of crossings between the two components—and $w(\bd_i)$ is the signed number of crossings whose two strands belong to $L_i$.

Denote by $\hat{A}_j$ the matrix obtained from deleting the $j$-th row and column from the matrix $A$. Then, we can compute
\begin{equation}\label{eq:Alexander-transition-matrix}
    \Delta_{\mathrm{cl}(\bd)}(x_1,\dots,x_\ell) = (-1)^\ell \left(\prod_{i=1}^\ell x_i^{-\frac{1}{2}\left(-N_i+\sum_{j=1}^\ell \mathrm{lk}_\bd(L_i,L_j)\right)}\right) \frac{\det{\left(I_{2n-1}-\hat{A}_j\right)}}{1-x_{\kappa(a_j)}},
\end{equation}

for any $1 \leq j\leq 2n$, where $N_j$ denotes the number of strands associated to the $j$-th link component, $L_j$. Throughout the rest of the paper, we will set $j=1$.

\begin{remark}\label{rmk:Alexander-open-strand} Alternatively, we can construct a matrix on the long link defined by $\beta$: consider the tangle defined by identifying the endpoints of the braid for every strand, except the first one, which we leave open. This tangle has now $2n+1$ distinct segments. Then, we can construct a matrix $\tilde{A}\in M_{2n+1}(\Z[x_1,\dots,x_n])$ using the same rules as before, namely \eqref{eq:A-rules}, and we have $\det{\left(I_{2n+1}-\tilde{A}\right)}=\det{\left(I_{2n-1}-\hat{A}_1\right)}$. Therefore, we can compute
\begin{equation}
    \Delta_{\mathrm{cl}(\bd)}(x_1,\dots,x_\ell) =(-1)^\ell \left(\prod_{i=1}^\ell x_i^{-\frac{1}{2}\left(-N_i+\sum_{j=1}^\ell \mathrm{lk}_\bd(L_i,L_j)\right)}\right) \frac{\det{\left(I_{2n+1}-\tilde{A}\right)}}{1-x_1}~.
\end{equation}
\end{remark}
\begin{figure}
    \centering
    \includegraphics[width=0.4\linewidth]{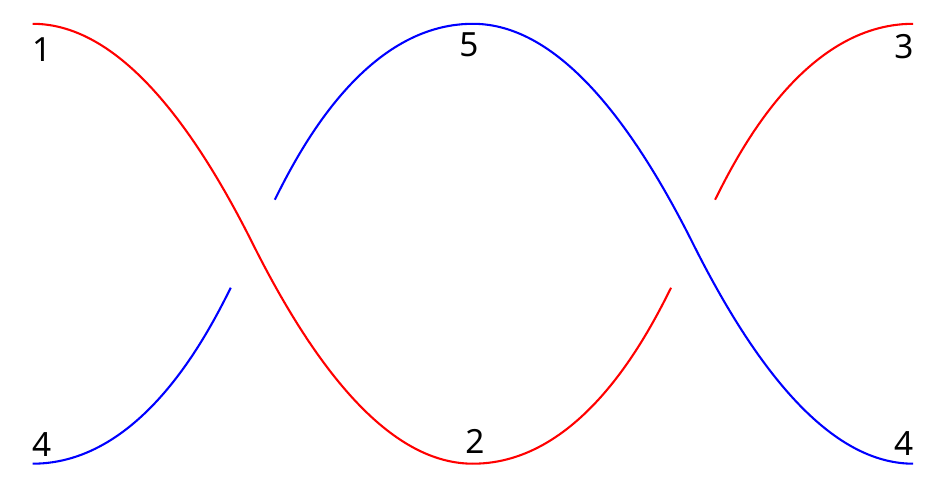}
    \caption{Braid and labeling for the Hopf link $H^{+}$.}
    \label{fig:hopf}
\end{figure}

\begin{example} ($H^\pm$) Let $H^{\pm}$ denote the positive and negative Hopf links, with braid words $\sigma_1^2$ and $\sigma_1^{-2}$, respectively. Labeling the braid segments as in Figure \ref{fig:hopf}, we construct the matrices
\begin{equation}
    \tilde{A}^{H^+}=\begin{pmatrix}
    0 & x_2 & 0 & 0 & 1-x_1 \\
    0 & 0 & 1 & 0 & 0\\
    0 & 0 & 0 & 0 & 0 \\
    0 & 0 & 0 & 0 & 1 \\
    0 & 0 & 1-x_2 & x_1 & 0 \\
    \end{pmatrix},
    \hspace{1em}
    \tilde{A}^{H^-}=\begin{pmatrix}
    0 & 1 & 0 & 0 & 0 \\
    0 & 0 & \frac{1}{x_2} & \frac{x_1}{x_2}\left(1-\frac{1}{x_1}\right) & 0 \\
    0 & 0 & 0 & 0 & 0 \\
    0 & \frac{x_2}{x_1}\left(1-\frac{1}{x_2}\right) & 0 & 0 & \frac{1}{x_1} \\
    0 & 0 & 0 & 1 & 0 \\
    \end{pmatrix}
\end{equation}
and
\begin{equation}
\begin{split}
    \Delta_{H^+}(x_1,x_2) &= \frac{\det\left(I_5-\tilde{A}^{H^+}\right)}{1-x_1}=\frac{1-x_1}{1-x_1}=1, \\
    \Delta_{H^-}(x_1,x_2 )&= x_1 x_2\frac{\det\left(I_5-\tilde{A}^{H^-}\right)}{1-x_1}=x_1 x_2\frac{-x_1^{-1}x_2^{-1}(1-x_1)}{1-x_1}=-1.
\end{split}
\end{equation}
\end{example}

\begin{example}\label{ex:T24-star} Let $T(2,4)^*$ denote the boundary of the twisted annulus with two full twists.
This link is equivalent to the $T(2,4)$ torus link, represented by the positive braid $\sigma_1^4$, as unoriented links. However, as oriented links they are different: in particular, $T(2,4)$ is fibered, while $T(2,4)^*$ is not. Consider the braid representative of $T(2,4)^*$ given by $\bd=\sigma_2^2\sigma_1\sigma_2^{-1}\sigma_1$.
\begin{figure}
    \centering
    \includegraphics[width=0.5\linewidth]{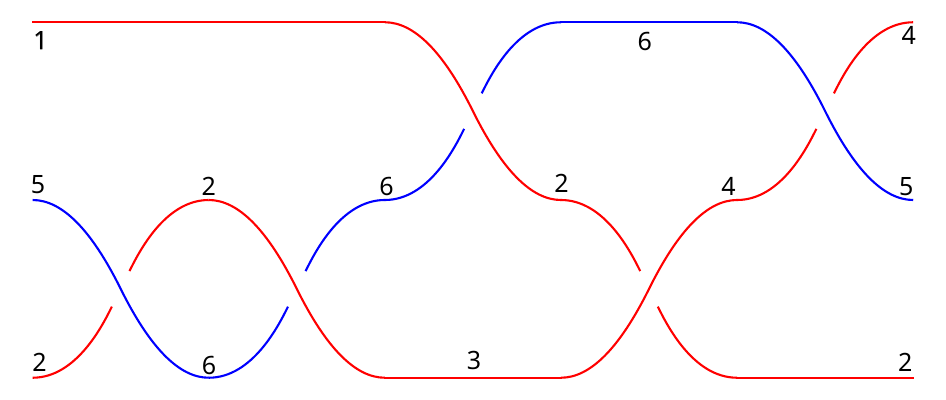}
    \caption{Braid and its associated segment labeling for the torus link $T(2,4)^*$.}
    \label{fig:T24*br}
\end{figure}
In \cref{fig:T24*br} we assign labels to each of the segments of the braid $\bd$. For simplicity, we will maintain the same label when going through the understrand, since this will not change the end result. We obtain the matrix:
\begin{equation}\label{eq:T24*-matrix}
    \tilde{A} = \begin{pmatrix}
        0 & x_2 & 0 & 0 & 0 & 1-x_1 \\
        0 & 0 & x_2 & 0 & 0 & 1-x_1 \\
        0 & 1-x_1^{-1} & 0 & x_1^{-1} & 0 & 0 \\
        0 & 0 & 0 & 0 & 0 & 0 \\
        0 & 1-x_2 & 0 & 0 & 0 & x_1 \\
        0 & 0 & 0 & 1-x_2 & x_1 &0 \\
    \end{pmatrix}
\end{equation}
and the polynomial
\begin{equation}\label{eq:T24*-alex}
\begin{split}
    \Delta_{T(2,4)^*}(x_1,x_2)&=x_1^{-\frac{1}{2}(-2-1+2)}x_2^{-\frac{1}{2}(-1+2)}\frac{\det\left(I_7-\tilde{A}\right)}{1-x_1} \\&= x_1^{\frac{1}{2}}x_2^{-\frac{1}{2}}\frac{(1-x_1)x_1^{-1}(x_1+x_2)}{1-x_1} = x_1^{-\frac12}x_2^{-\frac12}\left(x_1+x_2\right),
\end{split}
\end{equation}
which agrees with the Alexander polynomial of $T(2,4)^*$.
\end{example}

\subsubsection{Random walk interpretation of the Alexander polynomial}\label{sec:random-walk-Alexander}

In this section, we show that we can compute the Alexander polynomial from a ``jump-down'' random walk model on the braid closure, in which the simple multicycles either all avoid, or all go through the endpoint of the first strand.

Let $\mathcal{Q}$ be the set of simple multicycles in the random walk model defined by $\hat{A}_1$ as in \cref{sec:random-walks}. These are the simple multicycles that do not go through the first segment $a_1$ and for which, at every crossing, either they go through the overstrand, go through the understrand, go through both, ``jump down'' from the overstrand to the understrand, or they do not go through any segment adjacent to the crossing.

Let $\mathcal{Q}_{\text{tot}}$ be the set of simple cycles in the random walk model defined by $A$, and let $\mathcal{Q}'$ be the set of simple cycles as above that always go through the first segment. Notice that we can split $\mathcal{Q}_{\text{tot}}=\mathcal{Q} \sqcup \mathcal{Q}'$ since $\mathcal{Q}\cap\mathcal{Q}'=\emptyset$. Let $W$ be the weight function defined by $A$.

From \eqref{eq:Alexander-transition-matrix} we can write:
\begin{equation}\label{eq:Alex-simple-cycles}
    \Delta_{\mathrm{cl}(\bd)}(x_1,\dots,x_n) = (-1)^\ell \left(\prod_{i=1}^\ell x_i^{-\frac{1}{2}\left(-N_i+\sum_{j=1}^n \mathrm{lk}(L_i,L_j)\right)}\right) \frac{1}{1-x_1} \sum_{c\in \mathcal{Q}} (-1)^{|c|} W(c).
\end{equation}
Moreover, from \eqref{eq:detiszer} we have that
\begin{equation*}
    \det (I_{2n}-A)\factoreq \Delta_L + (-1)^\ell \left(\prod_{i=1}^\ell x_i^{-\frac{1}{2}\left(-N_i+\sum_{j=1}^n \mathrm{lk}(L_i,L_j)\right)}\right) \frac{1}{1-x_1} \sum_{c\in \mathcal{Q}'} (-1)^{|c|} W(c)=0,
\end{equation*}
from which we derive
\begin{equation}\label{eq:Alex-simple-cycles-through-open-strand}
    \Delta_{\mathrm{cl}(\bd)}(x_1,\dots,x_n) = - (-1)^\ell \left(\prod_{i=1}^\ell x_i^{-\frac{1}{2}\left(-N_i+\sum_{j=1}^n \mathrm{lk}(L_i,L_j)\right)}\right) \frac{1}{1-x_1} \sum_{c\in \mathcal{Q}'} (-1)^{|c|} W(c).
\end{equation}
\begin{figure}
\centering
\begin{subfigure}{0.15\textwidth}\centering
  \includegraphics[width=\linewidth]{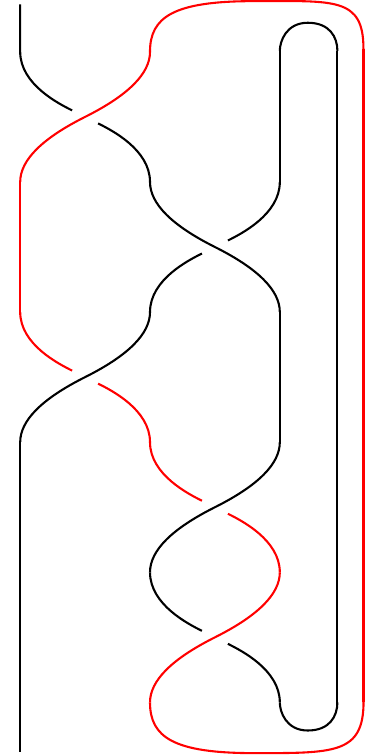}
\caption{$c_1$}
\end{subfigure}
\hspace{2em}
\begin{subfigure}{0.15\textwidth}\centering
  \includegraphics[width=\linewidth]{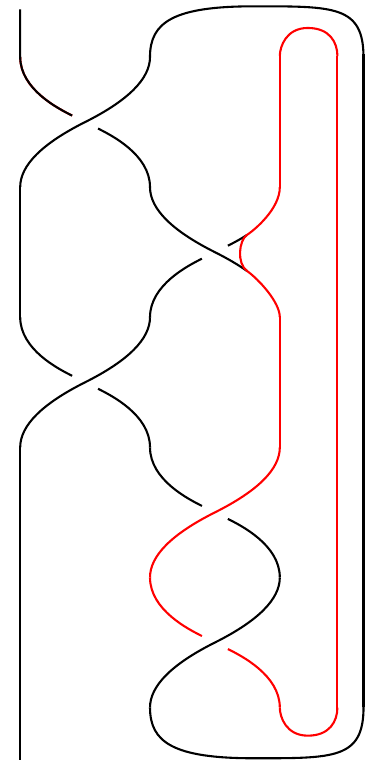}
\caption{$c_2$}
\end{subfigure}
\hspace{2em}
\begin{subfigure}{0.15\textwidth}\centering
  \includegraphics[width=\linewidth]{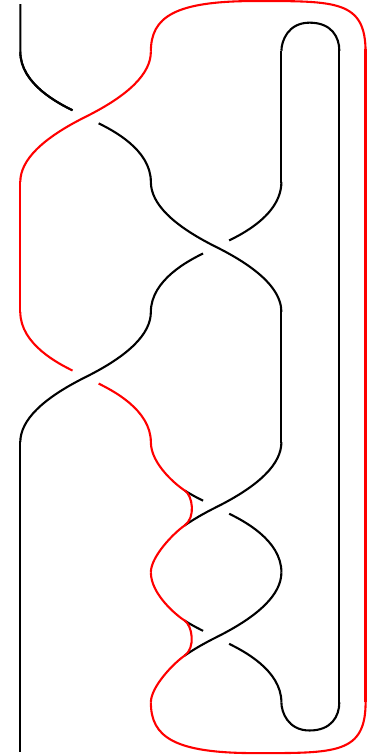}
\caption{$c_3$}
\end{subfigure}
\caption{A bottom-to-top-oriented braid for the torus link $T(2,4)^*$ and the simple cycles $c_1$, $c_2$ and $c_3$.}
\label{fig:T24stcycles}
\end{figure}
\begin{example}\label{ex:T24*-cycles} ($T(2,4)^*$) Let $c_1,c_2$ and $c_3$ be the cycles in \cref{fig:T24stcycles}. Then, the set of simple multicycles that do not go through the open strand is $\mathcal{Q}=\{\emptyset,c_1,c_2,(c_1,c_2),c_3\}$, and
\begin{equation*}
\begin{split}
    \sum_{c\in \mathcal{Q}} (-1)^{|c|}W(c)&=1+(-1)^{|c_1|}W(c_1)+(-1)^{|c_2|}W(c_2)+(-1)^{|c_1|+|c_2|}W(c_1)W(c_2)+(-1)^{|c_3|}W(c_3)
    \\&=1-x_1^2-x_2(1-x_1^{-1})+x_1^2x_2(1-x_1^{-1})-x_1(1-x_1)(1-x_2)\\&=(1-x_1)x_1^{-1}(x_1+x_2),
\end{split}
\end{equation*}
which agrees with the determinant computed from \eqref{eq:T24*-matrix}.
\end{example}

\subsection{Alexander norm}

In \cite{McMullen}, McMullen defined a norm from the Alexander polynomial of a link\footnote{More generally, McMullen introduced the Alexander norm as a norm on any connected, compact, orientable 3-manifold of toroidal boundary.} as follows: Given $\phi\in H^1(M,\Z)$, write $\phi=(\phi_1,\dots,\phi_\ell)$ in the dual meridian basis. Write $\Delta_L(\mathbf{x})=\sum_{\mathbf{k}} a_{\mathbf{k}}\mathbf{x}^\mathbf{k}$. Then, we define
\begin{equation}\label{eq:Alexander-norm}
    ||\phi||_A\defeq\max_{\mathbf{j},\mathbf{k}\in\supp \Delta_L} \inner{\phi}{\mathbf{k}-\mathbf{j}} = \max_{\mathbf{j},\mathbf{k}\in\supp \Delta_L} \sum_{i=1}^\ell \phi_i(k_i-j_i).
\end{equation}

Let $\newtonalex\subset H_1(M,\R)\cong \R^{\ell}$ denote the Newton polytope of the Alexander polynomial; i.e. the smallest convex set containing its support. Then, the unit norm ball of $||\cdot||_A$ is dual to twice the Newton polytope.

In the same work, McMullen proved that the Alexander norm bounds the Thurston norm on $H^1(M,\Z)$. This is a natural generalization of the established result for knots, which states that the total degree of the Alexander polynomial is bounded from above by twice the knot genus.
\begin{theorem}[{\cite[Theorem 1.1]{McMullen}}]\label{thm:Thurston-Alexander-norms} Given $\phi\in H^1(M,\Z)$, its Alexander and Thurston norms satisfy
\begin{equation}
    ||\phi||_A\leq\thurston{\phi} + \begin{cases}
        0 &\text{ if } \ell\geq 2 \\
        1 &\text{ if } \ell=1 \text{ and } H^1(M,\Z)=\Z\phi.
        \end{cases}
\end{equation}
Moreover, equality holds if $\phi$ is a fibered cohomology class.
\end{theorem}

Assume that $\ell \geq 2$. If we translate this statement in terms of polytopes, we have a polytope inclusion $B_T\subseteq 2\newtonalex^\vee$ or, equivalently, $2\newtonalex\subseteq B_T^\vee$. The equality of both norms for fibered cohomology classes implies that, for every fibered face $F_i$, there exists a top-dimensional face $\tilde{F_i}\subseteq \newtonalex^\vee$ such that
\[F_i\subseteq 2\tilde{F_i} \text{ under the polytope inclusion }B_T\subseteq 2\newtonalex^\vee.\] Therefore, for every fibered vertex $v\in B_T^\vee$, both $\frac{v}{2}$ and $-\frac{v}{2}$ are in the support of the Alexander polynomial. Furthermore, $\mathbf{x}^{\frac{v}{2}}$ and $\mathbf{x}^{-\frac{v}{2}}$ must have coefficients $\pm 1$ \cite[Theorem 4.1]{Dunfield}.

\subsection{Link Floer homology norm} A practical way to compute the Thurston polytope of a link complement, and its associated fibered faces, is by first computing its link Floer homology.

Given an oriented link $L\subset S^3$, its link Floer homology assigns to it a finitely generated, graded vector space over $\mathbb{F}=\Z/2\Z$, which splits as
\begin{equation}
    \HFL(L)=\bigoplus_{h\in\HL, d\in\Z} \HFL_d(L,h)
\end{equation}
where $\HL$ denotes the lattice defined at the beginning of \cref{sec:preliminaries}. Here, $h\in \HL$ denotes the Alexander-grading and $d\in\Z$ the Maslov (homological) grading.

Link Floer homology is a categorification of the multivariable Alexander polynomial, in the sense that we recover the latter when we take the Euler characteristic of the former. Precisely, if we write $h=(h_1,\dots,h_\ell)\in \HL\subset H_1(M,\R)$ in the meridian basis, we have
\begin{equation}\label{eq:HFL-recovers-Alex}
    \sum_{h\in\HL} \chi{(\HFL_*(L,h))}\, \mathbf{x}^\mathbf{s}  = \left(\prod_{i=1}^\ell( x_i^{\frac12}-x_i^{-\frac12})\right)\Delta_L(\mathbf{x})
\end{equation}
where $\HFL_*(L,h)\defeq\bigoplus_{d\in\Z}\HFL_d(L,h)$.

The link Floer homology shares properties similar to those of the Alexander polynomial, gathered in \cref{prop:Alex-properties}. In particular, it satisfies the symmetry \cite{Ozsvth2008}
\begin{equation}
    \HFL_{*}(L,h)\cong \HFL_{*-2\inner{\mathbb{1}^*}{h}}(L,-h).
\end{equation}
Under partial orientation reversal, we have the following result.
\begin{proposition}[\cite{Manolescu2007}]\label{prop:HFL-orientation-reversal} Let $L=\bigcup_{i=1}^\ell L_i$ be an oriented link. Let $L^{(i)}$ be the link defined by
\[L^{(i)}=L_1\cup\cdots\cup L_i^{\text{rev}}\cup\cdots\cup L_\ell\]
where $L_i^{\text{rev}}$ denotes the orientation reversal of $L_i$. Then, writing $h=(h_1,\dots,h_\ell)\in \HL$,
\begin{equation}
    \HFL_{*}(L,h)\cong \HFL_{*-2h_i+l_i}(L,(h_1,\dots,-h_i,\dots,h_\ell))
\end{equation}
where $l_i=\sum_{j=1}^\ell\mathrm{lk}(L_i,L_j)$.
\end{proposition}

In a similar manner to the multivariable Alexander polynomial, the link Floer homology induces a norm as well. As a natural generalization of \eqref{eq:Alexander-norm}, let
\begin{equation}
    ||\phi||_{\HFL}\defeq \max_{\{s\in\HL\mid \HFL(L,s)\neq 0\}} |\inner{ \phi}{s}|~,
\end{equation}

and define the \emph{link Floer polytope} as the convex hull of its Alexander-graded support:
\[
\HFLpolytope\defeq\operatorname{conv}\{h\in\mathbb H(L)\mid
\widehat{\mathit{HFL}}(L,h)\neq0\}~.
\]
By \eqref{eq:HFL-recovers-Alex}, the link Floer homology norm bounds the Alexander norm from above, as does the Thurston norm. Surprisingly, the two are equivalent.

\begin{theorem}[{\cite[Theorem 1.1]{OzsSza2008-2}}]\label{thm:HFL-poly-Thurston} Let $L$ be a link with no trivial components. Then, for every $\phi\in H^1(M,\R)$,
\begin{equation}
    \thurston{\phi} + \sum_{i=1}^\ell |\inner{\phi}{\mu_i}| = 2||\phi||_{\HFL}~.
\end{equation}
In other words, there is a polytope equality
\begin{equation}\label{eq:Thurston-HFL-polytopes}
    B_T^\vee + Q = 2\HFLpolytope
\end{equation}
where $Q\defeq\{\sum_{i=1}^\ell t_i\mu_i\mid -1\leq t_i\leq 1\}$ denotes the symmetric hypercube in $H^1(M,\Z)$ with edge-length two.
\end{theorem}

Given a vertex $v\in B_T^\vee$, define $s(v)\defeq\frac{1}{2}(v+Q)\subseteq \HL$. By \eqref{eq:Thurston-HFL-polytopes}, the fibered faces of the Thurston polytope $(F_1,\dots,F_s)$ correspond to sets of points in the link Floer polytope. Moreover, for every face $F_i$ and associated vertex $v_i$, there is at least one extremal vertex in $s(v_i)\subset\HFLpolytope$. The following result explains how link Floer homology is able to detect the fibered faces.

\begin{theorem}\label{thm:HFL-fibered-classes} Let $v\in\mathrm{Vert}(B_T^\vee)$ be an extremal vertex of the dual Thurston polytope. Then, $v$ corresponds to a fibered face of the Thurston polytope if, and only if, there exists $h\in s(v)$ with $h\in\mathrm{Vert}(\HFLpolytope)$ and $\rank{\HFL(L,h)}=1$.
\end{theorem}

The converse direction was proven in \cite{OzsSza2008-2} and the forward direction in \cite{Altman2015}.

In particular, link Floer homology detects link fiberedness, as it had been previously shown for its single-Alexander-graded counterpart, knot Floer homology \cite[Corollary~1.2]{Ni2007}.

\begin{theorem}\label{thm:fibered-link-HFL}
Let $L$ be an oriented link. Then, $L$ is fibered if, and only if, there is a unique
\[
h_{\max}\in\supp{\HFL(L)}
\]
maximizing the total Alexander grading $\inner{\mathbb{1}^*}{h}=\sum_i h_i$ and $\HFL(L,h_{\max})\cong\mathbb{F}$.

Moreover, if $L$ is fibered then $\HFL(L,h_{\max})$ is supported in Maslov grading $\lambda(L)$.
\end{theorem}

\begin{proof}

The Theorem follows directly by combining standard results of knot Floer homology and link Floer homology; namely \cite[Theorem 1.1]{Ozsvth2008}, to transform the statement into one about the knot Floer homology of a fibered link, and \cite[Corollary 1.2]{Ni2007} for the detection result.

Alternatively, we can prove it by combining \cref{thm:HFL-poly-Thurston} and \cref{thm:HFL-fibered-classes}:
Let $\mathbb{1}^*\in H^1(S^3\setminus L;\mathbb Z)$ be the oriented cohomology class satisfying $\mathbb{1}^*(\mu_i)=1$. Since $\mathbb{1}^*$ is fibered, the normalized class $\mathbb{1}^*/||\mathbb{1}^*||_T$ lies in the relative interior of a top-dimensional face $F$ of $B_T$ and uniquely exposes an extremal vertex $v$ of $B_T^\vee$, dual to $F$. Moreover, since $\mathbb{1}^*(\mu_i)=1$ for every $i$, it uniquely exposes the vertex $\mathbb{1}=\sum_i\mu_i$ of $Q$ as well.
Hence, by \cref{thm:HFL-poly-Thurston}, $\mathbb{1}^*$ uniquely exposes the vertex $h_{\max}=\frac{1}{2}\left(v+\mathbb{1}\right)$ in $\HFLpolytope$.
It follows that $h_{\max}$ is the unique supported grading of $\HFL$ maximizing $\inner{\mathbb{1}^*}{\cdot}$, and by \cref{thm:HFL-fibered-classes} it must have rank $1$. The reverse direction can be argued in a similar way.

We deduce that $\HFL(L,h_{\text{max}})$ must be supported in Maslov grading $\lambda(L)$ from the analogous statement in knot Floer homology \cite{OzsSza2005,ChengHeddenSarkar2024}.
\end{proof}

For the special class of alternating links, the equality in \eqref{eq:HFL-recovers-Alex} holds after substituting the Euler characteristic $\chi(\HFL_*(L,s))$ for the rank of $\HFL(L,s)$ \cite[Theorem 1.3]{Ozsvth2008}; i.e. the information captured by link Floer homology is equivalent to that captured by the multivariable Alexander polynomial. As a result, we obtain the following corollary:

\begin{corollary}\label{cor:Alex-monic-fibered} Let $L$ be an alternating link. Then, the vertices of $\newtonalex$ associated to fibered faces are precisely the extremal vertices with associated coefficients $\pm1$.
\end{corollary}

\subsection{Random walk models induced by state sum models}\label{sec:random-walks-from-state-sums}

Let $\mathrm{R}=\mathrm{R}(\alpha,\beta,\gamma):V\otimes V\rightarrow V\otimes V$ be an endomorphism of $V\otimes V$ depending on three independent parameters $(\alpha,\beta,\gamma)$, where $V$ is a graded vector space over $\C$:
\begin{equation*}
    V=\bigoplus_{i\geq 0} \C v_i, \hspace{1.5em} \deg{(v_i)}=i~.
\end{equation*}
The tensor product $V^{\otimes m}$ is graded by the total degree $\deg{(v_{i_1}\otimes \cdots\otimes v_{i_m})}=i_1+\cdots+i_m$.

The morphism $\mathrm{R}$ is uniquely determined by its entries $\mathrm{R}_{i,j}^{i',j'}$, where
\begin{equation}\label{eq:R-morphism-into-matrix-elements}
    \mathrm{R} (v_i\otimes v_j) = \sum_{i',j'\geq 0} \mathrm{R}_{i,j}^{i',j'} (v_{i'}\otimes v_{j'}).
\end{equation}
Assume that these entries take the form
\begin{equation}\label{eq:R-entries}
    \mathrm{R}_{i,j}^{i',j'}= \mathrm{R}_{i,j}^{i',j'}(\alpha,\beta,\gamma)= \delta_{i+j,i'+j'} \binom{i}{i-j'} \alpha^{j}\beta^{j'}\gamma^{i-j'}~.
\end{equation}
Then, $\mathrm{R}$ preserves the grading of $V\otimes V$.

Identifying $v_i$ with the monomial $z^i$, we can identify $V\cong \C[z]$ and $V^{\otimes m}\cong \C[z_1,\dots,z_m]$ for every $m\geq 0$. Under this identification, $\mathrm{R}$ induces an algebra morphism, since

\begin{equation}\label{eq:R-algebra-morphism}
    \mathrm{R}(z_1^iz_2^j) = \left(\gamma z_1+\beta z_2\right)^i\left(\alpha z_1\right)^j=\left(\mathrm{R}(z_1)\right)^i\left(\mathrm{R}(z_2)\right)^j~.
\end{equation}

Given a braid let $m$ denote its number of strands and $n$ its number of crossings. A state on the braid diagram is an assignment of a basis vector $v_i\in V$, or equivalently a nonnegative integer $i$, to each segment of the braid. At a crossing, let $i,j$ denote the labels assigned by the state on the incoming strands and $i',j'$ the labels on the outgoing strands (\cref{fig:ouo+u+}). For every crossing, with associated labels $\{i,j,i',j'\}$, $\mathrm{R}_{i,j}^{i',j'}$ is the contribution of the crossing.

Throughout the section, we work with the partial closure obtained by identifying the incoming and outgoing ends of every strand, except the first, and we require all states to assign $0$ to the incoming and outgoing segments of the first strand. The state sum associated to the partial closure, denoted by $\mathrm{Z}(\boldsymbol{\alpha},\boldsymbol{\beta},\boldsymbol{\gamma})$, is given by
\begin{equation}\label{eq:state-sum-Z}
    \mathrm{Z}(\boldsymbol{\alpha},\boldsymbol{\beta},\boldsymbol{\gamma}) = \sum_{s} \prod_{c} \mathrm{R}^{i_c',j_c'}_{i_c,j_c}(\alpha_c,\beta_c,\gamma_c)~.
\end{equation}
Here, the sum ranges over all states that are compatible with the identifications imposed by the partial braid closure, and $c$ ranges over the crossings of the diagram. The $\mathrm{R}$--morphisms are evaluated at different parameters $(\alpha_c,\beta_c,\gamma_c)$ for every crossing. We use the notation $\boldsymbol{\alpha}$ for the set of the $n$ $\alpha_c$--variables, and similarly for the rest of the parameters. The infinite sum is understood as a formal power series in $\boldsymbol{\alpha},\boldsymbol{\beta},\boldsymbol{\gamma}$.

Given a braid, and a labeling of the segments of its partial closure, we construct a $2n+1$ square matrix $\mathcal{A}$ from $\mathrm{R}$ in the following way. For every set of labels $\{o,u,o^+,u^+\}$ associated to the adjacent segments of a crossing, as in \cref{fig:ouo+u+}, we have
\begin{equation}\label{eq:transition-matrix-from-R}
    \mathcal{A}_{o,o^+} = \mathrm{R}_{1,0}^{0,1}, \hspace{1em} \mathcal{A}_{o,u^+} = \mathrm{R}_{1,0}^{1,0}, \hspace{1em} \mathcal{A}_{u,o^+} = \mathrm{R}_{0,1}^{0,1}, \hspace{1em} \mathcal{A}_{u,u^+} = \mathrm{R}_{0,1}^{1,0}
\end{equation}
and any other pair $(k,r)\in\{o,u,o^+,u^+\}^2$ satisfies $\mathcal{A}_{k,r}=0$.

A level of the braid is a regular horizontal (resp. vertical) slice containing no crossing, when the braid is read from bottom to top (resp. from left to right). Each level intersects the braid in exactly $m$ segments. Since $\mathrm{R}$ preserves the total degree, the sum of the state values assigned to the $m$ strands at any level of the braid is fixed. We call this value the total degree of a state $s$.

The matrix $\mathcal{A}\in M_{2n+1}(\Z[\boldsymbol{\alpha},\boldsymbol{\beta},\boldsymbol{\gamma}])$ satisfies the conditions of \cref{sec:random-walks}, and therefore induces a random walk model on the braid closure.

Denote by $\mathcal{P}$ the set of all multicycles. Then, we have that
\begin{equation}\label{eq:Z-inverse-det}
    \mathrm{Z}(\boldsymbol{\alpha},\boldsymbol{\beta},\boldsymbol{\gamma}) = \sum_{c\in\mathcal{P}} \weight{c} = \frac{1}{\det{(I-\mathcal{A})}}~.
\end{equation}
The first equality follows from the map induced between the set of multicycles $\mathcal{P}$ and the set of states: For every multicycle, we assign the state which records the number of times the multicycle passes through every segment. Conversely, different multicycles may determine the same state, since the state records only the occupation number of each segment and forgets the ordering information of the multicycle. Note that this map turns into a one-to-one correspondence if we restrict to the set of primitive cycles and the set of states of total degree $1$. The second equality, which we understand as elements of $\Z[[\boldsymbol{\alpha},\boldsymbol{\beta},\boldsymbol{\gamma}]]$, follows from work of Foata--Zeilberger\footnote{See \cite[Section 2.3]{GiscardRochet2017} for an exposition in the case of \emph{hikes} on graphs. The directed graph has as vertices the segments of the braid, and as edges the allowed jumps between segments, determined by the morphism $\mathrm{R}$. While \cite{LinWang2001} first applied the Foata--Zeilberger formula to a state sum model on a braid, it was noticed in \cite[Remark 2.5]{Boninger2025} that, as stated, their formula is incorrect, and \cite{GiscardRochet2017} provides the right formulation. The transition matrix $\mathcal{A}$ is there denoted by $\mathrm{G}$.} \cite{FoataZeilberger1999}.

\begin{remark} More generally, we could have considered entries of the form
\begin{equation}
    \mathrm{R}_{i,j}^{i',j'}(\alpha,\beta,\gamma,\zeta)=\sum_{k=0}^i\sum_{r=0}^j \delta_{k+j-r,i'}\delta_{i-k+r,j'} \binom{i}{k}\binom{j}{r}\gamma^k\beta^{i-k}\alpha^{j-r}\zeta^{r},
\end{equation}
which satisfy $\mathrm{R}(z_1^iz_2^j) = \left(\gamma z_1+\beta z_2\right)^i\left(\alpha z_1+\zeta z_2\right)^j=\left(\mathrm{R}(z_1)\right)^i\left(\mathrm{R}(z_2)\right)^j$, and therefore also induce an algebra morphism. The difference between the induced random walk models is in what kind of jumps are allowed. For our goals, the model induced by \eqref{eq:R-entries} (also known as the ``jump-down'' model) suffices.
\end{remark}

\subsection{Inverted state sum}\label{sec:inverted-state-sum}

Park initially introduced the large-color $R$-matrix to provide a construction of the Gukov--Manolescu series for positive braid knots \cite{Park20}.
The $n$--colored Jones polynomial is obtained from finite-dimensional $U_q(\mathfrak{sl}_2)$--modules by assigning an $R$--matrix to every crossing of a braid and then taking a reduced quantum trace.
In the large-color limit, one lets $n\to\infty$ while keeping the component variable $x=q^n$ fixed.
The finite-dimensional modules are then replaced by infinite-dimensional Verma modules, and the resulting state sum has infinitely many states. As a result, one obtains a two-variable power series which recovers the Melvin--Morton--Rozansky expansion of the Introduction.

Subsequently, in \cite{Park21} Park modified the large-color $R$-matrices and introduced the inverted state sum, which allowed him to compute a bigger class of examples, including homogeneous braid knots. In this section, we review this construction.

Let $\beta\in B_n$ be a braid, oriented from bottom to top, with closure $L$.
Denote by $V_\beta$ the set of crossings and by $E_\beta$ the set of oriented segments obtained by cutting the closed braid at its crossings.
At a crossing $c\in V_\beta$, we write $i,j$ for the labels on the bottom-left and bottom-right segments, and $i',j'$ for those on the top-left and top-right segments, respectively (cf. \cref{fig:ouo+u+}).

\begin{definition}[Inversion datum]\label{def:inversion-datum-ss}
An \emph{inversion datum} on $\beta$ is a sign assignment
\begin{equation}\label{eq:inversion-datum-background}
    \iota:E_\beta\longrightarrow\{+,-\}
\end{equation}
such that, at every crossing, the signs on the adjacent segments are among
the following possibilities:
\begin{equation}\label{eq:allowed-sign-patterns}
\begin{alignedat}{2}
&\stackanchor{-+}{+-},\quad
 \stackanchor{+-}{-+},\quad
 \stackanchor{--}{--},\quad
 \stackanchor{-+}{-+},\quad
 \stackanchor{++}{++}
 &&\quad\text{for a positive crossing},\\
&\stackanchor{-+}{+-},\quad
 \stackanchor{+-}{-+},\quad
 \stackanchor{--}{--},\quad
 \stackanchor{+-}{+-},\quad
 \stackanchor{++}{++}
 &&\quad\text{for a negative crossing}.
\end{alignedat}
\end{equation}
Here the outgoing pair is displayed above the incoming pair.
\end{definition}

In particular, at every crossing the number of incoming and outgoing
$-$--segments is the same.
Consequently, the union of the $-$--segments forms an oriented simple multicycle in the braid, where the local patterns in \eqref{eq:allowed-sign-patterns} prescribe how the multicycle passes through, or jumps at, a crossing; cf. \cref{sec:random-walk-Alexander}.
We refer to this as the \emph{inversion multicycle} associated to $\iota$.

\begin{definition}[State]\label{def:state-inverted-ss}
    Fix an inversion datum $\iota$.
    A \emph{state} compatible with $\iota$ is an integer labeling
    \begin{equation*}
        s:E_\beta \longrightarrow\Z
    \end{equation*}
    such that
    \begin{equation*}
        s(e)\geq 0 \quad\text{if }\iota(e)=+,
        \qquad
        s(e)<0\quad \text{if }\iota(e)=-
    \end{equation*}
    and satisfying the conservation law
    \begin{equation}\label{eq:inverted-state-conservation}
        i+j=i'+j'
    \end{equation}
    at every crossing.
\end{definition}

There is a distinguished state determined entirely by the inversion datum.
Identifying the sign $\pm$ with $\pm 1$, set
\begin{equation}
    s_0(e) = \frac{\iota(e)-1}{2}=
    \begin{cases}
        0,&\iota=+,\\
        -1,&\iota=-
    \end{cases}~.
\end{equation}
The local condition in \eqref{eq:allowed-sign-patterns} guarantees that $s_0$ satisfies \eqref{eq:inverted-state-conservation}.
We call $s_0$ the \emph{ground state}.
The extended $R$--matrices determine which other states compatible with $\iota$ actually contribute to the state sum.

\subsubsection{Inverted $R$--matrices}\label{sec:inverted-R-matrices}

At a crossing $c$, let $i,\,j,\,i',\,j'$ denote the state labels on the bottom-left, bottom-right, top-left, and top-right segments, respectively.
Every nonzero matrix element contains the factor $\delta_{i+j,i'+j'}$, and therefore automatically enforces \eqref{eq:inverted-state-conservation}.
We use the notation
\begin{equation}
 (a;q)_k=\prod_{r=0}^{k-1}(1-aq^r),
 \qquad
 \qbinom{m}{k}_q
 =
 \frac{(q^m;q^{-1})_k}{(q;q)_k},
 \qquad k\geq0~.
\end{equation}

At a positive crossing, let $x$ be the variable associated with the overstrand and let $y$ be the variable associated with the understrand.
For $i\geq j'\geq0$ or $0>i\geq j'$, the corresponding matrix entry is
\begin{align}
 \check R(x,y)_{i,j}^{i',j'}={}&
 \delta_{i+j,i'+j'}
 q^{\frac{j+j'+1}{2}+jj'}
 x^{\frac{i'+j+1}{4}}
 y^{\frac{3j'-i+1}{4}}
 \qbinom{i}{i-j'}_q
 \prod_{r=1}^{i-j'}(1-q^{j+r}y),
 \label{eq:positive-rmatrix-a}
\end{align}
whereas, for $j'\geq0>i$,
\begin{align}
 \check R(x,y)_{i,j}^{i',j'}={}&
 \delta_{i+j,i'+j'}
 q^{\frac{j+j'+1}{2}+jj'}
 x^{\frac{i'+j+1}{4}}
 y^{\frac{3j'-i+1}{4}}
 \qbinom{i}{j'}_q
 \prod_{r=0}^{j'-i-1}(1-q^{j-r}y)^{-1}~.
 \label{eq:positive-rmatrix-b}
\end{align}
At a negative crossing, let $x$ be the variable associated to the understrand and $y$ that associated to the overstrand.
For $j\geq i'\geq 0$ or $0>j\geq i'$, we have
\begin{align}
 \check R^{-1}(x,y)_{i,j}^{i',j'}={}&
 \delta_{i+j,i'+j'}
 q^{-\frac{i+i'+1}{2}-ii'}
 x^{-\frac{3i'-j+1}{4}}
 y^{-\frac{j'+i+1}{4}}
 \qbinom{j}{j-i'}_{q^{-1}}
 \prod_{r=1}^{j-i'}(1-q^{-i-r}x^{-1}),
 \label{eq:negative-rmatrix-a}
\end{align}
while, for $i'\geq0>j$,
\begin{align}
 \check R^{-1}(x,y)_{i,j}^{i',j'}={}&
 \delta_{i+j,i'+j'}
 q^{-\frac{i+i'+1}{2}-ii'}
 x^{-\frac{3i'-j+1}{4}}
 y^{-\frac{j'+i+1}{4}}
 \qbinom{j}{i'}_{q^{-1}}
 \prod_{r=0}^{i'-j-1}(1-q^{-i+r}x^{-1})^{-1}~.
 \label{eq:negative-rmatrix-b}
\end{align}
All remaining matrix entries are zero. A state $s$ compatible with $\iota$ is called \emph{valid} if the corresponding $R$--matrix entry is nonzero at every crossing.

In the formulas above, we regard $\check{R}$ and $\check{R}^{-1}$ as elements of $\Z[[x,y]]$, by identifying any rational expression in $x$ or $y$ with its Laurent series expansion at $x=0$ or $y=0$.\footnote{Note that our formulas differ from \cite{Park21} by inverting the $x,y$ variables and shifting by an overall power of $q$.
}

\subsubsection{Inverted state sums}\label{sec:inverted-state-sums-subsubs}

Let $b_1, \ldots, b_n$ denote the bottom segments of the braid. Consider the partial closure of the braid, where we identify the bottom and top ends of every strand but the first. We refer to $b_1$ as the ``open strand''.
Write
\begin{equation*}
    \kappa(r)\in\{1,\ldots,\ell\}
\end{equation*}
for the component of $L$ containing $b_r$, so that $x_{\kappa(r)}$ denotes the associated $x$-variable. We take the convention that $\kappa(1)=1$.

Following \cite{OSSS25}, we denote by $\Omega(\iota)$ the set of valid states for which the label on the open strand is fixed to its ground-state value:
\begin{equation}\label{eq:Omega-iota}
    \Omega(\iota)
    \defeq
    \bigl\{
        s\text{ valid for }\iota
        \ \big|\
        s(b_1)=s_0(b_1)
    \bigr\}.
\end{equation}
When a continuous description is needed below, we will simply take the real relaxation $\Omega_\R(\iota)$ by extending the equations and inequalities appearing in \eqref{eq:positive-rmatrix-a}--\eqref{eq:negative-rmatrix-b} to $\R^{E_\beta}$.

For $s\in \Omega(\iota)$, the reduced quantum trace contributes the factor
\begin{equation}\label{eq:trace-factor-link}
 \mu_r(s)
 =
 x_{\kappa(r)}^{-\frac12}
 q^{-\frac12-s(b_r)},
 \qquad 2\leq r\leq n~.
\end{equation}
If $R_c(s)$ denotes the corresponding entry of \eqref{eq:positive-rmatrix-a}--\eqref{eq:negative-rmatrix-b} at the crossing $c$, define
\begin{equation}\label{eq:state-contribution-link}
 P(s)
 =
 \left(\prod_{r=2}^n\mu_r(s)\right)
 \left(\prod_{c\in V_\beta}R_c(s)\right).
\end{equation}

Let $\mathbf{s}(\iota)$ be the number of closed components in the inversion multicycle, not counting the component containing the open strand.
The inverted state sum associated to $(\beta,\iota)$ is
\begin{equation}\label{eq:multivariable-inverted-state-sum}
 Z^{\mathrm{inv}}_{(\beta,\iota)}(\mathbf{x},q)
 =
 (-1)^{\mathbf{s}(\iota)}
 \sum_{s\in\Omega(\iota)}P(s),
\end{equation}
and we set
\begin{equation}\label{eq:FL-inverted-state-sum}
    F_{(\bd,\iota)}(\mathbf{x},q)
    =
    \left( x_{1}^{\frac12} - x_{1}^{-\frac12} \right) Z^{\mathrm{inv}}_{(\beta,\iota)}(\mathbf{x},q)~.
\end{equation}
This is the multivariable form of Park's inverted state-sum construction used throughout the paper;
see \cite{Park21,FKCompute2026}.

The local admissibility conditions do not by themselves guarantee that the inverted state sum is well defined: since $\Omega(\iota)$ may be infinite, one must additionally require that only finitely many states contribute below any
fixed degree.

For each component $L_i$ of $L$, let $X_i(s)$ denote the minimal $x_i$--degree of
the summand $P_\iota(s)$.
Given $N\in\mathbb N$, define
\begin{equation}\label{eq:degree-truncation-ss}
    Q_N(\iota)
    :=
    \left\{
        s\in \Omega_{\mathbb R}(\iota)
        \ \middle|\
        \left\langle \mathbb{1},X(s)\right\rangle\leq N
    \right\}~.
\end{equation}

\begin{definition}[Nice]\label{def:nice-ss}
A pair of a braid and an inversion datum $(\beta,\iota)$ is \emph{nice} if there exists an $N\in \N$ such that
$Q_N(\iota)$ is nonempty and bounded. Equivalently, for some $N$, there are finitely many $s \in\Omega(\iota)$ for which the total minimal degree of $P(s)$ is at most $N$.

We say that a link is \emph{nice} if it admits a nice braid
presentation and inversion datum.
\end{definition}

\begin{remark}
    By \cref{cor:one-cutoff}, if $Q_{N}(\iota)$ is nonempty and bounded for some $N$, then it is bounded for all $N\in\N$. This justifies our definition of \emph{niceness}: If a pair $(\bd,\iota)$ is nice, the associated state sum can be written as
    \begin{equation*}
        F_{(\bd,\iota)}(\mathbf{x},q)=\sum_{\mathbf{k}\in \Z^\ell} g_\mathbf{k}(q) \mathbf{x}^\mathbf{k}
    \end{equation*}
    with $g_{\mathbf{k}}(q)\in\Z[q,q^{-1}]$.

\end{remark}

\subsection{Random walk model from the inverted state sum}\label{sec:random-walk-inverted-state-sum}

In this section, we review how to use the formalism from \cref{sec:random-walks-from-state-sums} to define a random walk model from the inverted state sum of \cref{sec:inverted-state-sum}, as in \cite{Park21}.

Define the parametrized inverted $R$-matrices as
\begin{equation}\label{eq:parametrized-inverted-R}
\begin{split}
    \mathrm{R}_{+,+}^{+,+}(\alpha,\beta,\gamma)_{i,j}^{i',j'}&=\delta_{i+j,i'+j'} \binom{i}{j'} \alpha^j \beta^{j'} \gamma^{i-j'}, \\
    \mathrm{R}_{+,-}^{-,+}(\alpha,\beta,\gamma)_{i,j}^{i',j'}&=\delta_{i+j,i'+j'} \binom{i}{j'} \alpha^{-(-j-1)} \beta^{j'} \gamma^{i-j'}, \\
    \mathrm{R}_{-,+}^{+,-}(\alpha,\beta,\gamma)_{i,j}^{i',j'}&=\delta_{i+j,i'+j'} \binom{-(-i-1)+i-j'}{i-j'} \alpha^{j} \beta^{-(-j'-1)} (-\gamma)^{-(-i-1)+(-j'-1)}, \\
    \mathrm{R}_{-,+}^{-,+}(\alpha,\beta,\gamma)_{i,j}^{i',j'}&=\delta_{i+j,i'+j'} \binom{j'+(-i-1)}{j'} \alpha^j (-\beta^{j'}) \gamma^{-(-i-1)-j'}, \\
    \mathrm{R}_{-,-}^{-,-}(\alpha,\beta,\gamma)_{i,j}^{i',j'}&=\delta_{i+j,i'+j'} \binom{-(-i-1)+i-j'}{i-j'} \alpha^{-(-j-1)} \beta^{-(-j'-1)} (-\gamma)^{-(-i-1)+(-j'-1)}. \\
\end{split}
\end{equation}
where $i,j,i',j'\geq 0$. The labeling convention $\{i,j,i',j'\}$ of \cref{fig:ouo+u+} is most useful when thinking about $\mathrm{R}$ as matrix elements (cf. \eqref{eq:R-morphism-into-matrix-elements}), while the convention $\{o,u,o^+,u^+\}$ of \cref{fig:ouo+u+} makes the formulas simpler when thinking of $\mathrm{R}$ as morphisms associated to crossings. Indeed, the expressions for the negative crossings are the same as above in the $\{o,u,o^+,u^+\}$ convention (which exchanges the roles of $i,j$ and $i',j'$). Therefore:
\begin{equation}
    (\mathrm{R}^{-1})_{s_1,s_2}^{s_3,s_4}(\alpha,\beta,\gamma)_{i,j}^{i',j'} = \mathrm{R}_{s_2,s_1}^{s_4,s_3}(\alpha,\beta,\gamma)_{j,i}^{j',i'}.
\end{equation}

Consider the following set of rules: at each crossing, evaluate

\begin{equation}\label{eq:rules-abg}
    \alpha = x_{\kappa(o)}^{\frac{s}{2}}, \;\beta = x_{\kappa(u)}^{\frac{s}{2}}, \;\gamma = (x_{\kappa(i)}^{-1}x_{\kappa(u)})^{-\frac{s}{4}}(1-x_{\kappa(u)}^s)
\end{equation}
where $\kappa(o)$ denotes the variable assigned to the overstrand, $\kappa(u)$ the variable assigned to the understrand, and $s$ denotes the sign of the crossing.

The matrix $\mathrm{R}_{+,+}^{+,+}(\alpha,\beta,\gamma)$ agrees with the positive parametrized matrix of the state sum model in \cref{sec:random-walks-from-state-sums}. Let $\mathcal{B}$ be the transition matrix defined using $\mathrm{R}_{+,+}^{+,+}$ and $(\mathrm{R}^{-1})_{+,+}^{+,+}$. It has entries
\begin{equation}\label{eq:B-rules}
\begin{split}
    &\mathcal{B}_{i,i^+}=\beta
    \, , \hspace{1em}
    \mathcal{B}_{i,j^+}=\gamma
    \, , \hspace{1em}
    \mathcal{B}_{j,j^+}=\alpha
    \, , \hspace{1em}
    \mathcal{B}_{j,i^+}=0
    \, , \hspace{1em}
    \mathcal{B}_{i,k}=\mathcal{B}_{j,k}=0 \quad \text{if } k\neq i^+,j^+
    \end{split}
\end{equation}
in the same convention as \eqref{eq:transition-matrix-from-R}.

Define $\rho(\beta)\in\Z^\ell$ such that
\begin{equation}\label{eq:rho-beta-def}
    \mathbf{x}^{\rho(\beta)}= \prod_{i=1}^\ell x_i^{-\frac12(-N_i+\sum_{j=1}^\ell \mathrm{lk}_\beta(L_i,L_j))}~.
\end{equation}
Then, we have the following result.
\begin{lemma}\label{lemma:B-matrix-Alexander} $\det{\left(I-\mathcal{B}\right)}= (1-x_1) \mathbf{x}^{-\rho(\beta)}\Delta_L(\mathbf{x})$ under \eqref{eq:rules-abg}.
\end{lemma}
\begin{proof}
To prove the result, it suffices to find an equality between the left-hand side and the determinant described in \cref{sec:Gassner,sec:random-walk-Alexander}.

It is immediate, after comparing the formulas \eqref{eq:A-rules} and \eqref{eq:rules-abg}--\eqref{eq:B-rules}, that the entry $(o,u^+)$ has a factor of either $1-x_{\phi(i)}^s$ or $1-x_{\phi(u)}^s$, depending on the model. Therefore, we will compare the random walk defined by $\mathcal{B}$ with a slightly different model than the one in \cref{sec:random-walk-Alexander}, which has the correct polynomial factor and will turn out to be equivalent.

Define the reversed braid $\bd_L^\mathrm{r}$ as the image of $\bd_L$ under the composition of the endomorphisms of the braid group $B_m$ given by $\sigma_i \mapsto \sigma_{m-i}$ and by $\prod_{i=1}^n \sigma_{g(i)}^{s(i)} \mapsto \prod_{i=1}^n \sigma_{g(n-i)}^{s(n-i)}$ for any $g:\{1,\dots,n\}\rightarrow \{1,\dots,n\}$ and $s:\{1,\dots,t\}\rightarrow \{\pm1\}$. Equivalently, $\bd_L^{\mathrm{r}}$ can be drawn as $\bd_L$ with all its strands pointing in the opposite direction, which corresponds to a $180$--degree rotation in the plane of the braid.

Let $\mathcal{A}$ be the transition matrix defined from the lower Wirtinger presentation of $\bd^{\mathrm{r}}$ (cf. \cref{sec:Gassner} for the upper Wirtinger presentation). The simple cycles associated with $\mathcal A$ and described in \cref{sec:random-walk-Alexander} are the same, up to orientation reversal. Since the closure of $\bd^{\mathrm{r}}$ is the reverse link of the closure of $\bd$, by \cref{prop:Alex-properties}, $\det{\left(I-\mathcal{A}\right)}$ agrees with $(1-x_1)\Delta_L(\mathbf{x})$, up to multiplication by the unit $\mathbf{x}^{-\rho(\beta)}$.

Consider the transpose of the $\mathcal{A}$-model. The variables associated to the ``jumps'' between overstrand and understrand are the same at every crossing for the transposed model and the $\mathcal{B}$-model. More concretely, the only difference between the weights of both models is given by monomials of non-zero weight. Therefore, there is an invertible matrix $U$ such that

\begin{equation*}
    \det{\left(I-\mathcal{B}\right)}=\det{\left(U^{-1}\left(I-\mathcal{A}\right)U\right)^T} = \det{\left(I-\mathcal{A}\right)}.
\end{equation*}
We conclude that the result agrees with $\Delta_L(\mathbf{x})$ up to the same units as $\det{\left(I-\mathcal{A}\right)}$. The units needed are worked out for the lower Wirtinger presentation model of the reversed braid in \cite{Rozansky98}, and agree with \eqref{eq:Alexander-transition-matrix} up to the sign $(-1)^\ell$.
\end{proof}

In \cite{Park21}, given an inversion datum $\iota$, Park considers a random walk model induced by the parametrized inverted $R$-matrices of \eqref{eq:parametrized-inverted-R}. At every crossing, $\iota$ determines which formula of \eqref{eq:parametrized-inverted-R} we use, as in \cref{sec:inverted-state-sum}. The transition matrix of this model is denoted by $\mathcal{B}^\mathrm{inv}$.

The inversion datum $\iota$ is equivalent to a simple multicycle in the random walk model induced by $\mathcal{B}$, for which jumps are only allowed from $o$ to $o^+$, $u$ to $u^+$ or $o$ to $u^+$ (also known as a ``jump-down'' model). The associated cycle $c_\iota$ is constructed by connecting the segments labeled with a $-$ sign: for a segment $e\in E_\bd$,
\begin{equation}
    e\in c_\iota \Leftrightarrow \iota(e)=-1.
\end{equation}
See, for example, \cref{fig:L7-plus-cycle}.

Upon inversion, there is a one-to-one correspondence between the cycles of non-zero weight in the random walk models induced by $\mathcal{B}$ and by $\mathcal{B}^{\textrm{inv}}$ (cf.  \cref{sec:random-walks,sec:random-walks-from-state-sums}). We denote these sets by\footnote{Here, $\mathcal{Q}$ corresponds to the set of simple multicycles that go through the open strand if $c_\iota$ does, or that do not go through the open strand if $c_\iota$ does not. The result gives the same invariant, up to overall sign, by an argument similar to \cref{sec:random-walk-Alexander}.} $\mathcal{Q}$ and $\mathcal{Q}(\iota)$, respectively. For more details on this correspondence, see \cite{Park21}. In particular, the simple multicycle $c_\iota$ in the $\mathcal{B}$-model corresponds to the empty cycle $\emptyset$ in the $\mathcal{B}^\textrm{inv}$-model.

\begin{lemma}\label{lemma:B-matrix-inverted-Alexander} $(-1)^\mathbf{s}\weight{c_{\iota}}\det{\left(I-\mathcal{B}^\mathrm{inv}\right)} = (1-x_1) \mathbf{x}^{-\rho(\beta)} \Delta_L(\mathbf{x})$ under \eqref{eq:rules-abg}.
\end{lemma}
\begin{proof}

By the lemma above, it suffices to show that $(-1)^\mathbf{s}\weight{c_{\iota}}\det{\left(I-\mathcal{B}^\mathrm{inv}\right)} = \det{\left(I-\mathcal{B}\right)}$ for arbitrary $\boldsymbol{\alpha},\boldsymbol{\beta},\boldsymbol{\gamma}$. This follows from the computation in terms of simple multicycles of \cref{sec:random-walks}, and the relation
\begin{equation}\label{eq:Winv-rel-W}
    \weightinv{\tilde{c}}=(-1)^{-\mathbf{s}(\tilde{c})+\mathbf{s}(c)-\mathbf{s}(c_\iota)}\frac{\weight{c}}{\weight{c_{\iota}}}
\end{equation}
for every simple multicycle $\tilde{c}\in\mathcal{Q}(\iota)$ and its associated simple multicycle $c\in\mathcal{Q}$, proven in \cite{Park21}. Here, $\weight{c}$ denotes the weight induced by the transition matrix $\mathcal{B}$, and $\weightinv{\tilde
c}$ denotes the weight induced by $\mathcal{B}^\mathrm{inv}$.
\end{proof}

Fix an inversion datum $\iota$; let $\mathrm{Z^{\mathrm{inv}}}(\boldsymbol{\alpha},\boldsymbol{\beta},\boldsymbol{\gamma})$ denote the state sum computed using the parametrized inverted R-matrices \eqref{eq:parametrized-inverted-R}, normalized by $\weight{c_\iota}^{-1}$. Denote by $\mathrm{Z^{\mathrm{inv}}}(\mathbf{x})$ the result of evaluating the state sum under \eqref{eq:rules-abg}.

Note that, under \eqref{eq:rules-abg}, the parametrized inverted $\mathrm{R}$-matrices \eqref{eq:parametrized-inverted-R} and the inverted $\check{R}$-matrices \eqref{eq:positive-rmatrix-a}--\eqref{eq:negative-rmatrix-b} differ from each other by a factor of $(x_{\kappa(u)}x_{\kappa(o)})^{\frac{s}{4}}$, where $s$ denotes the sign of the crossing. Moreover, taking into account the extra factor \eqref{eq:trace-factor-link} that appears in the inverted state sum $Z^{\mathrm{inv}}_{(\beta,\iota)}(\mathbf{x},q)$, we obtain the relation
\begin{equation}\label{eq:parametrized-inverted-state-sum-vs-usual}
\begin{split}
    Z^{\mathrm{inv}}_{(\beta,\iota)}(\mathbf{x},1) &= \left(\prod_{r=2}^\ell x_{\kappa(r)}^{-\frac12}\right)\left(\prod_{i=1}^\ell x_i^{\frac{1}{2}\sum_{j=1}^\ell\mathrm{lk}_{\beta}(L_i,L_j)}\right)\mathrm{Z^{\mathrm{inv}}}(\mathbf{x}) \\
    &= \left(x_1^{-\frac{N_1-1}{2}}\prod_{i=2}^\ell x_{i}^{-\frac{N_i}{2}}\right)\left(\prod_{i=1}^\ell x_i^{\frac{1}{2}\sum_{j=1}^\ell\mathrm{lk}_{\beta}(L_i,L_j)}\right)\mathrm{Z^{\mathrm{inv}}}(\mathbf{x})~,
\end{split}
\end{equation}
where the left-hand side corresponds to the inverted state sum of \cref{sec:inverted-state-sum} evaluated at $q=1$.

After normalization, we obtain
\begin{equation}\label{eq:FL-parametrized-state-sum}
    F_{(\bd,\iota)}(\mathbf{x},1)=\left(x_1^{\frac12}-x_1^{-\frac{1}{2}}\right) Z_{(\bd,\iota)}^{\mathrm{inv}}(\mathbf{x},1) = -(1-x_1) \mathbf{x}^{-\rho(\beta)} \mathrm{Z^{\mathrm{inv}}}(\mathbf{x})
\end{equation}

On the other hand, we have that
\begin{equation}
    \mathrm{Z^{\mathrm{inv}}}(\boldsymbol{\alpha},\boldsymbol{\beta},\boldsymbol{\gamma}) = \frac{1}{\weight{c_{\iota}}}\sum_{\tilde{c}\in \mathcal{P}(\iota)}  \weightinv{\tilde{c}}~,
\end{equation}
where $\mathcal{P}(\iota)$ denotes the set of multicycles in the random walk model induced by $\mathcal{B}^{\textrm{inv}}$ and $\iota$. In particular, we have a map between the set of multicycles $\mathcal{P}(\iota)$ and the set of states $s\in \Omega(\iota)$ (see \cref{sec:random-walks-from-state-sums}). Moreover, as shown in \cite{Park21},
\begin{equation}\label{eq:Zinv-and-inverse-det-Binv}
    \mathrm{Z^{\mathrm{inv}}}(\boldsymbol{\alpha},\boldsymbol{\beta},\boldsymbol{\gamma}) = \frac{1}{\weight{c_\iota}\det{(I-\mathcal{B}^\mathrm{inv})}}
\end{equation}
where the equality is understood as formal power series in $\boldsymbol{\alpha},\boldsymbol{\beta}$ and $\boldsymbol{\gamma}$.

\section{The inverse of the Alexander polynomial}
\label{sec:alexander-inverses}

We are interested in the inverse of the multivariable Alexander polynomial when we consider the variables $x_1,\dots,x_n$ near $0$. When we are dealing with a knot $K$, the Alexander polynomial satisfies the crucial property that $x^{-d}\Delta_K(0)\neq 0$, where $d$ denotes the lowest $x$-degree of $\Delta_K(x)$, which makes $\Delta_K$ a unit in the ring $\Q[[x]]$. However, in the case of a link $L$ with more than one component, $\mathbf{x}^{-\mathbf{d}}\Delta_L(0,\dots,0)$ may be zero, where now $\mathbf{d}$ denotes the vector of lowest $\mathbf{x}$-degrees of $\Delta_L$. Therefore, we distinguish two cases, depending on whether $\Delta_L$ defines a unit in the ring $\Q[[x_1^\frac12,\dots,x_n^\frac12]]$ or not.

In the first case, there is a (unique) inverse of $\Delta_L$ in the ring $\Q[[x_1^\frac12,\dots,x_n^\frac12]]$. Take as an example the link $T(2,4)$, computed as the closure of the braid word $\sigma_1^4$. Up to multiplication by monomials in $x_1,x_2$, its multivariable Alexander polynomial is $1+x_1x_2$, and its inverse in the ring $\Z[[x_1,x_2]]\subset \Q[[x_1,x_2]]$ is the power series $\sum_{n\geq 0}(-1)^n x_1^n x_2^n$.

In the second case, we need to extend the ring in which we are working. With the same goal in mind, \cite{AK13} introduced a formalism to compute formal Laurent series in several variables associated to the inverse of a multivariable polynomial, by defining suitable rings in which these series live. We base our construction on their procedure, with some slight modifications in order to fit it to our goals.

\subsection{The inverse of a multivariable polynomial}

In this section, we develop the necessary theory to define the inverse of a multivariable polynomial. We primarily follow \cite{AK13}.

Let $\mathcal{C}$ be a rational, finitely generated, line-free, $n$-dimensional cone; i.e. a subset of $\R^n$ such that there exist $s\in\N$ and $\mathbf{v}_1,\cdots,\mathbf{v}_s\in\Z^n$ with $\mathcal{C}=\{z_1\mathbf{v}_1+\cdots+z_s\mathbf{v}_s \mid z_1,\dots,z_s \geq 0 \}$ and $-\mathbf{v}_i\notin \mathcal{C}$ for all $i$.

Let $R$ be a ring and $R^\times\subseteq R$ its ring of units. Consider the set of formal infinite series of the form
\begin{equation}
    f(\mathbf{x})=\sum_{\mathbf{k} \in \Z^{n}} a_\mathbf{k} \mathbf{x}^\mathbf{k}
\end{equation}
where $a_\mathbf{k}\in R$ for every $\mathbf{k}\in \Z^{n}$. The support of $f$ is the set $\supp f(\mathbf{x})\defeq\{\mathbf{k}\in \Z^{n}\mid a_{\mathbf{k}}\neq 0\}$. Define
\begin{equation}
    R_\mathcal{C}[[\mathbf{x}]]\defeq\{f(\mathbf{x}) \mid \textrm{supp}\;f(\mathbf{x})\subseteq \mathcal{C}\}
\end{equation}
In our case, we will consider $R$ to be the field of rational numbers $\Q$, the ring of integers $\Z$ or the ring of Laurent polynomials $\Z[q,q^{-1}]$.

Then, by the same argument as in \cite[Theorem 10]{AK13}, $R_{\mathcal{C}}[[\mathbf{x}]]$ is a ring if the cone $\mathcal{C}$ is line-free.

\begin{definition} Let $f(\mathbf{x})\in R[\mathbf{x}]$ be a multivariable polynomial. Then, a functional $\phi\in(\Z\setminus \{0\})^\ell$ \emph{exposes} $\mathbf{v}$ if
\begin{equation}
    \operatorname*{argmin}_{\mathbf{k}\in\supp f}\inner{\phi}{\mathbf{k}} = \{\mathbf{v}\}~;
\end{equation}
i.e., the functional defined by $\phi$ is minimized uniquely at $\mathbf{v}$ over $\supp f$. When there is no confusion, we say that $\mathbf{v}$ is an \emph{exposed} exponent.

If, moreover, $\phi\in\Z_{>0}^\ell$ is a \emph{positive} functional, we say that $\phi$ \emph{positively exposes} $v$, or that $v$ is a \emph{positively exposed} exponent.
\end{definition}
\begin{remark} Given a multivariable polynomial $f(\mathbf{x})\in R[\mathbf{x}]$, every extremal vertex of its Newton polytope $N(f)$ is an exposed exponent for $f$. In that case, we may refer to $\mathbf{v}\in\mathrm{Vert}(N(f))$ as an \emph{exposed vertex}.
\end{remark}
Given an exposed exponent $\mathbf{v}$, it follows from \cite[Lemma 7]{AK13} that the cone centered at $\mathbf{v}$, defined by
\begin{equation}
    \mathcal{C}_\mathbf{v}\defeq\mathrm{Cone}\{\mathbf{w}-\mathbf{v} \mid \mathbf{w}\neq \mathbf{v} \text{ and } \mathbf{w}\in\supp f\}
\end{equation}
is line-free. Therefore, $R_{\mathcal{C}_\mathbf{v}}[[\mathbf{x}]]$ is a ring for every exposed exponent $\mathbf{v}$.

\begin{lemma} Let $f(\mathbf{x})\in R[\mathbf{x}]$ be a multivariable Laurent polynomial. Let $\mathbf{v}$ be an exposed exponent of $f$, and write $f = \sum_{\mathbf{k}\in \supp f} a_\mathbf{k} \mathbf{x}^\mathbf{k}$.

Then, $\mathbf{x}^{-\mathbf{v}}f$ is a unit in $R_{\mathcal{C}_\mathbf{v}}[[\mathbf{x}]]$ if, and only if, $a_\mathbf{v}\in R^\times$.
\end{lemma}

The proof is a straightforward modification of \cite[Theorem 12]{AK13} when we change the field of coefficients to a ring.

\begin{definition}\label{def:inverse-polynomial} Let $f\in R[\mathbf{x}]$ be a multivariable polynomial and let $\mathbf{v}$ be an exposed exponent of $f$ for which $a_\mathbf{v}\in R^\times$. Then, we say that $f$ is \emph{invertible around} $\mathbf{v}$, and we define the \emph{inverse of $f$ around $\mathbf{v}$} as the formal Laurent series
\begin{equation}\label{eq:inverse-polynomial}
    f_\mathbf{v}^{-1}(\mathbf{x})\defeq\mathbf{x}^{-\mathbf{v}}(\mathbf{x}^{-\mathbf{v}}f(\mathbf{x}))^{-1}\in\mathbf{x}^{-\mathbf{v}} R_{\mathcal{C}_\mathbf{v}}[[\mathbf{x}]].
\end{equation}
In this case, we say that $f_\mathbf{v}^{-1}(\mathbf{x})$ is a formal Laurent series associated to $f^{-1}(\mathbf{x})$ and $a_{\mathbf{v}}^{-1}\mathbf{x}^{-\mathbf{v}}$ is its \emph{initial term}.
\end{definition}

\begin{remark}\label{rmk:lexicographic}
The exposed exponents $\mathbf{v}$ that uniquely minimize the lexicographic order of $\sigma(\Z^\ell)$, for some permutation $\sigma\in S_\ell$, are positively exposed. The inverse of a polynomial around one such exponent, as in \cref{def:inverse-polynomial}, agrees with its inverse in the \emph{field of iterated Laurent series}
\begin{equation}
    \Q\langle\langle x_{\sigma(1)},\dots,x_{\sigma(\ell)}\rangle\rangle\defeq\left(\left(\left(\Q[[x_{\sigma(1)}]]\right)[[x_{\sigma(2)}]]\right)\cdots\right)[[x_{\sigma(\ell)}]]
\end{equation}
previously studied in \cite[Chapter 2]{Xin04}.
\end{remark}

We are particularly interested in the inverses of Laurent polynomials with integer coefficients. Given an integral Laurent polynomial $f\in\Z[\mathbf{x}]$, and an exposed vertex $\mathbf{v}$ for which $a_\mathbf{v}\in\{\pm 1\}$, $f$ is invertible around $\mathbf{v}$ and $f_\mathbf{v}^{-1}(\mathbf{x})\in\mathbf{x}^{-\mathbf{v}}\Z_{\mathcal{C}_\mathbf{v}}[[\mathbf{x}]]$. We say that such a vertex is \emph{integral}. In this case, the discussion above gives the following result.

\begin{corollary}\label{cor:integral-inverses} Let $f(\mathbf{x})\in\Z[\mathbf{x}]$ be a multivariable polynomial with integer coefficients. Then, the formal Laurent series with integer coefficients associated to $f^{-1}(\mathbf{x})$ are labeled by the set
\begin{equation*}
    \mathcal{S}=
    \{\mathbf{v}\in\supp f \mid a_\mathbf{v}\in\{\pm 1\}\text{ and } \mathbf{v} \text{ is positively exposed}\}\subseteq \mathrm{Vert}(N(f)).
\end{equation*}
In other words, $f$ is invertible around $\mathbf{v}$ for every $\mathbf{v}\in\mathcal{S}$, and its inverse has integral coefficients.
\end{corollary}

\begin{example} Let $f(x_1,x_2,x_3)=x_1x_2^2+x_2^2x_3+x_1x_2x_3^2$. The set of extremal vertices of $N(f)$ is $\{(1,2,0),(0,2,1),(1,1,2)\}$. We define the three cones
\begin{equation}
\begin{split}
    C_{(1,2,0)}&=\{z_1(-1,1)+z_2(0,-1,2)\mid z_1,z_2\geq 0\}, \\
    C_{(0,2,1)}&=\{z_1(1,0,-1)+z_2(1,-1,1)\mid z_1,z_2\geq 0\}, \\
    C_{(1,1,2)}&=\{z_1(0,1,-2)+z_2(-1,1,-1)\mid z_1,z_2\geq 0\}.
\end{split}
\end{equation}
Consider the cone $C_{(1,2,0)}$. The Laurent polynomial $h_{(1,2,0)}(x_1,x_2,x_3)=\frac{1}{x_1x_2^2}f(x_1,x_2,x_3)=1+\frac{x_3}{x_1}+\frac{x_3^2}{x_2}$ has support in the cone $C_{(1,2,0)}$ and its initial term is $1$. It is invertible in the ring $\Z_{C_{(1,2,0)}}[[x_1,x_2,x_3]]$, with inverse
\begin{equation}
    \sum_{n\geq 0} (-1)^n\left(\frac{x_3}{x_1}+\frac{x_3^2}{x_2}\right)^n
\end{equation}
which is clearly an element of $\Z_{C_{(1,2,0)}}[[x_1,x_2,x_3]]$. It is therefore natural to consider the inverse of $f$ to be $\frac{1}{x_1x_2^2}\sum_{n\geq 0} (-1)^n\left(\frac{x_3}{x_1}+\frac{x_3^2}{x_2}\right)^n$.
\end{example}

\subsection{Inverses of the Alexander polynomial}

By the results of the previous section, there is a family of formal Laurent series, defined in suitable rings, that correspond to the inverse of the Alexander polynomial. Moreover, each of these inverses will be associated to an extremal vertex of the Newton polytope of $\Delta_L$. In the rest of the paper, we will be concerned only with the cases in which the selected vertices are integral.

\subsubsection{The positive convention}\label{sec:positive-convention} It suffices to only consider the inverse of $\Delta_L$ around \emph{positively} exposed vertices $v\in\supp \Delta_L$.

Given an oriented link $L=\bigcup_{1\leq i\leq \ell}L_i$ and $\nu=(\nu_1,\dots,\nu_\ell)\in \{\pm 1\}^\ell$, denote by $L^\nu=L^{(\nu_1,\dots,\nu_\ell)}$ the link $L^{\nu}\defeq \bigcup_{1\leq i\leq \ell } L_i^{\nu_i}$, where
\begin{equation}
    L_i^{\nu_i}=\begin{cases}
        L_i &\text{ if } \nu_i=+1 \\
        L_i^{\mathrm{rev}}&\text{ if } \nu_i=-1
    \end{cases} \hspace{1.5em}\text{for every } 1\leq i\leq \ell~.
\end{equation}
Note that $L$ and $L^{(\nu_1,\dots,\nu_\ell)}$ are isotopic as unoriented links.

\begin{proposition}\label{prop:alex-positive-convention}

Let $v$ be an exponent of $\newtonalex$ exposed by some functional $\phi\in(\Z\setminus \{0\})^\ell$. Let $\nu_i=\sgn \phi_i$ and denote
\begin{equation}
    v^\nu=(\nu_1v_1,\dots,\nu_\ell v_\ell)\in\Z^\ell \hspace{1.25em} \text{and} \hspace{1.25em} \phi^\nu=(\nu_1\phi_1,\dots,\nu_\ell\phi_\ell)\in\Z_{>0}^\ell~.
\end{equation}
Then, $\phi$ exposes $v$ in $\newtonalex$ if, and only if, $\phi^{\nu}$ positively exposes $v^\nu$ in $N(\Delta_{L^{\nu}})$. Moreover, the inverses of $\Delta_L(x_1,\dots,x_\ell)$ and $\Delta_{L^{(\nu_1,\dots,\nu_\ell)}}(x_1,\dots,x_\ell)$ around $v$ and $v^\nu$, respectively, are related via:
\begin{equation}
    (\Delta_L)^{-1}_v(x_1,\dots,x_\ell) = (-1)^{\frac{1}{2}\left(\ell-\sum_{i=1}^{\ell}\nu_i\right)}(\Delta_{L^{\nu}})^{-1}_{v^\nu}(x_1^{\nu_1},\dots,x_\ell^{\nu_\ell})~.
\end{equation}
\end{proposition}

The above proposition justifies our convention: We only consider the inverses of the multivariable Alexander polynomial that are expanded around \emph{positively} exposed exponents, since the remaining expansions can be obtained by changing the choice of underlying orientation of the link.

The choice of restricting to positively exposed exponents also has a natural interpretation in terms of the $x$-variables.
If $v$ is positively exposed by $\phi\in\Z_{>0}^{\ell}$, then the specialization $x_i=t^{\phi_i}$ sends every $x_i$ to $0$ as $t\to0$. From this point of view, we interpret $(\Delta_L)^{-1}_v$ as a formal Laurent series near $\mathbf{x}=\mathbf{0}$.

\subsubsection{Alternating links}

\begin{proposition} Let $L$ be an alternating link and $v\in\supp \Delta_L$ be a positively exposed vertex of $\newtonalex$. Then, the expansion of $\Delta_L^{-1}$ around $v$ has integer coefficients if, and only if, $2v$ is a fibered vertex of the dual Thurston polytope, under the inclusion $2\newtonalex \subseteq B_T^\vee$.
\end{proposition}

\begin{proof} The result follows from \cref{cor:integral-inverses,cor:Alex-monic-fibered}.
\end{proof}

\subsubsection{Homogeneous braid links}

A homogeneous braid $\bd\in B_n$ is a braid that can be written in the alphabet $\{\sigma_1^{\varepsilon(1)},\dots,\sigma_{n-1}^{\varepsilon(n-1)}\}$, where $\varepsilon(i)\in\{\pm 1\}$ for every $i=1,\dots,n-1$.

\begin{proposition}\label{prop:hom-braid-alex} Let $\bd$ be a homogeneous braid and $L$ its closure. Then, there exists a unique $v\in\supp{\Delta_L}$ which is a positively exposed exponent of $\Delta_L$, and it has coefficient $\pm1$.
\end{proposition}
\begin{figure}
    \centering
    \includegraphics[width=0.5\linewidth]{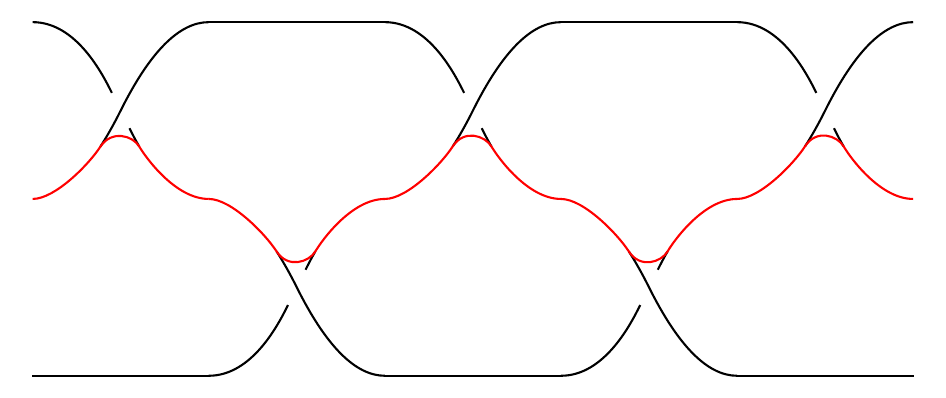}
    \caption{Inversion cycle of a homogeneous braid link, $L5a1\{0\}$.}
    \label{fig:L5a1}
\end{figure}

\begin{proof} Denote by $n_-$ the number of negative generators in $\bd$. For every $\varepsilon(i)<0$, let $c_i$ be the cycle that starts at the incoming overstrand of the first appearance of $\sigma_i^{-1}$ and jumps down to the outgoing understrand at every appearance of $\sigma_i^{-1}$. Let $c_-$ be the multicycle of $n_-$ components defined by (see \cref{fig:L5a1})
\begin{equation*}
    c_-=\bigcup_{\substack{1\leq i\leq n-1\\\text{s.t. }\varepsilon(i)<0}} c_i.
\end{equation*}
In particular, if $\beta$ is a positive braid, the cycle is empty and $\weight{\emptyset}=1$. Given \eqref{eq:A-rules}, the weight of $c_-$ in the determinant computation of $\Delta_L(\mathbf{x})$ is
\begin{equation}
    \weight{c_-} = (-1)^{n_-}\prod_{1\leq i\leq n-1} x_i^{-k_i} (1-x_i)^{k_i}
\end{equation}
where $k_i>0$ whenever $\varepsilon(i)<0$ and $k_i=0$ otherwise. Moreover, $c_-$ minimizes the exponents of all $x_i$ variables simultaneously, and the weight of any other multicycle $c$ in the braid must have minimal $x_i$-exponent strictly greater than $-k_i$ for some $1\leq i\leq n-1$. Therefore, by \eqref{eq:Alex-simple-cycles} we get that $\Delta_L(\mathbf{x}) \factoreq 1+p(\mathbf{x})$ for some $p(\mathbf{x})\in\Z[\mathbf{x}]$ whose minimal total degree is greater than $1$. Since any positive functional $\phi\in\Z_{>0}^\ell$ is uniquely minimized in $\N^\ell$ at $\mathbf{0}$, the result follows.
\end{proof}

\begin{corollary} The multivariable Alexander polynomial of a homogeneous braid link is invertible in $\Z[[\mathbf{x}]]$.
\end{corollary}

\subsubsection{Fibered links}

Recall that an oriented link $L$ is fibered if the oriented cohomology class $\mathbb{1}^*\in H^1(M,\Z)$ is fibered. In particular, homogeneous braid links are fibered \cite{Stallings1978}.

\begin{proposition}\label{prop:fibered-link-Alex} Let $L$ be a fibered link. Then, its multivariable Alexander polynomial has a unique monomial of maximal total degree, and it has coefficient $\pm 1$.
\end{proposition}
\begin{proof} It follows from \eqref{eq:HFL-recovers-Alex} and \cref{thm:fibered-link-HFL}.
\end{proof}

\begin{example}\label{ex:Dunfield-link-fiberedness} The reverse implication of \cref{prop:fibered-link-Alex} is not true. There are many known knot counterexamples that have monic Alexander polynomial, but are not fibered.

\begin{figure}
\centering
\includegraphics[width=0.25\linewidth]{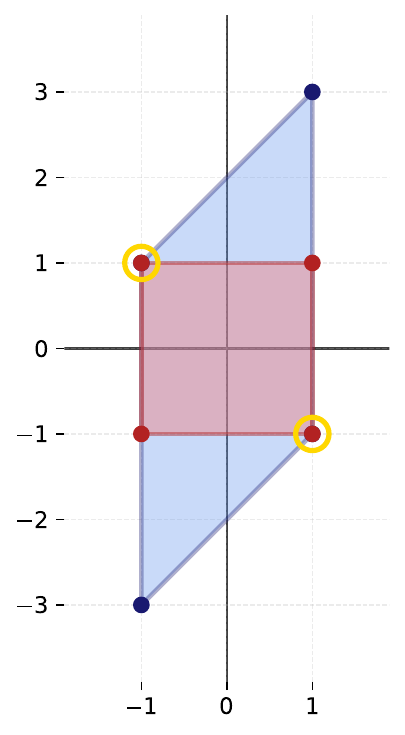}
\caption{Comparison of the dual Thurston ball $B_T^\vee$ (blue) with
$2\newtonalex$ (red). Gold rings indicate the vertices
dual to fibered faces, where the two boundaries agree as required by
McMullen's theorem.
}
\label{fig:dunfield-overlay}
\end{figure}

An interesting counterexample, which is not a knot, is $L=L10n14\{0\}$. This link already appeared in Dunfield's work \cite{Dunfield}, as an example of a link $L$ whose complement $S^3\setminus L$ is fibered, but the Thurston polytope and the dual Alexander polytope are not equal, as shown in \cref{fig:dunfield-overlay}. This link provides a counterexample to a question of McMullen in \cite{McMullen}, which is: If a link complement is fibered, must the Thurston polytope and the dual Alexander polytope be equal?

While the link complement is a fibered manifold, it is not a fibered link. Indeed, even though $\Delta_L=\frac{1}{x^{1/2}y^{1/2}}-\frac{x^{1/2}}{y^{1/2}}-\frac{y^{1/2}}{x^{1/2}}+x^{1/2}y^{1/2}$ has a unique monomial of maximal total degree, of coefficient $1$, its link Floer homology (computed from \cite{lfhcompute}) is
\begin{equation*}
\HFL\bigl(L,(i,j)\bigr)
=
\begin{cases}
\mathbb F_{(i+j)}
&
\text{if }(i,j)\in\{(-1,-1),(1,1)\},
\\[4pt]
\mathbb F_{(i+j)}\oplus\mathbb F_{(i+j+1)}
&
\text{if }(i,j)\in
\{0,1\}\times\{-2\}
\cup
\{-1,0\}\times\{2\},
\\[4pt]
\mathbb F_{(i+j)}^{\oplus 3}\oplus\mathbb F_{(i+j+1)}
&
\text{if }(i,j)\in\{-1,1\}\times\{0\},
\\[4pt]
\mathbb F_{(i+j)}^{\oplus 4}\oplus\mathbb F_{(i+j+1)}^{\oplus 2}
&
\text{if }(i,j)\in\{0\}\times\{-1,1\},
\\[4pt]
\mathbb F_{(i+j)}^{\oplus 3}\oplus\mathbb F_{(i+j+1)}^{\oplus 2}
&
\text{if }(i,j)\in\{(-1,1),(1,-1)\},
\\[4pt]
\mathbb F_{(0)}^{\oplus 6}\oplus\mathbb F_{(1)}^{\oplus 2}
&
\text{if }(i,j)=(0,0),
\\[4pt]
0
&
\text{otherwise}.
\end{cases}
\end{equation*}
There are two distinct lattice points, $(1,1)$ and $(0,2)\in\HFLpolytope$, that maximize the total sum of its components, so by \cref{thm:fibered-link-HFL}, it cannot be a fibered link.
\end{example}

\begin{example}\label{ex:L7n1-HFL}
A direct consequence of \cref{prop:fibered-link-Alex} is that the Alexander polynomial of a fibered link may not be a unit in $\Z[[\mathbf{x}]]$.

In \cite[Section 12.2.2]{Ozsvth2008}, the link Floer homology of $L=L7n1\{1\}$ is computed; it is given by
\begin{equation*}
\HFL\bigl(L,(i,j)\bigr)
=
\begin{cases}
\mathbb F_{(i+j-1)}
&
\text{if }(i,j)\in
\{0,-1\}\times\{1,2\}
\cup
\{0,1\}\times\{-1,-2\},
\\[4pt]
\mathbb F_{(0)}\oplus\mathbb F_{(-1)}
&
\text{if }(i,j)=(0,0),
\\[4pt]
0
&
\text{otherwise}.
\end{cases}
\end{equation*}
By \cref{thm:fibered-link-HFL}, the link is fibered, since the lattice point $(0,2)$ maximizes the sum of its components and $\rank{\HFL(L,(0,2))}=1$. In contrast, the multivariable Alexander polynomial of $L$ is
\begin{equation*}
    \Delta_L=\frac{x_1^{1/2}}{x_2^{3/2}}+\frac{x_2^{3/2}}{x_1^{1/2}},
\end{equation*}
which is not invertible in $\Z[[\mathbf{x}]]$. The two lexicographic orderings select different lattice points: $(1/2,-3/2)$ and $(-1/2,3/2)$. Therefore, even though the link is fibered, there exist two different expansions associated to $\Delta_L^{-1}$, namely
\begin{equation*}
    \sum_{n\geq 0} (-1)^n \frac{x_2^{(6n+3)/2}}{x_1^{(2n+1)/2}} \quad\text{ and }\quad \sum_{n\geq 0} (-1)^n \frac{x_1^{(2n-1)/2}} {x_2^{(6n-3)/2}}.
\end{equation*}
Note that this does not happen for the link $L7n1\{0\}$, which is related to $L7n1\{1\}$ by reversing the orientation of one of its link components. It is easy to see that this link is again fibered, and its Alexander polynomial is $\Delta_L=\frac{1}{x_1^{1/2}x_2^{3/2}}+x_1^{1/2}x_2^{3/2}$, whose inverse has a unique associated expansion, giving $\sum_{n\geq 0} (-1)^n x_1^{(2n+1)/2}x_2^{(6n+3)/2}$.
\end{example}

\section{Main results}
\label{sec:main-results}

\subsection{Statements}

Given a braid $\bd$ and an inversion datum $\iota$, let $c_\iota$ be the cycle induced by $\iota$. Define $\mathcal{Q}$ as the set of simple cycles in the model that do not go through the open strand if $\iota(b_1)=+1$, and that do go through the open strand if $\iota(b_1)=-1$; i.e., we will be considering the random walk model computing Alexander either from \eqref{eq:Alex-simple-cycles} or \eqref{eq:Alex-simple-cycles-through-open-strand} depending on where $c_\iota$ lives.

Given a cycle $c\in\mathcal{Q}$, we can always write its weight as
\begin{equation}\label{eq:exponent-weight-Alex-model}
    \weight{c} = \epsilon(c)\mathbf{x}^{\mathbf{a}(c)} \prod_{i=1}^\ell (1-x_i)^{b_i(c)}
\end{equation}
for some $\epsilon(c)\in\{\pm 1\}$, $\mathbf{a}(c)\in\Z^\ell$ and $b_i(c)\in\Z_{>0}^\ell$.

\begin{proposition}\label{prop:nice-iff-unital} The pair $(\bd,\iota)$ is nice if, and only if, there exists a functional $\phi \in \Z_{>0}^\ell$ for which there is a unique minimum in the set $\{\mathbf{a}(c) \mid c\in\mathcal{Q}\}$ and it is achieved at $\mathbf{a}(c_\iota)$.
\end{proposition}

From \cref{lemma:B-matrix-inverted-Alexander} it follows that, if such a functional $\phi$ exists, it must have a unique minimum in the Newton polytope of the Alexander polynomial, and this minimum must be $\mathbf{a}(c_\iota)+\rho(\beta)$. Therefore, we can identify $\mathbf{a}(c_\iota)$ with an extremal vertex in $\newtonalex$, after normalization.

\begin{corollary}\label{cor:inversion-datum-extremal-vertex} Let $(\beta,\iota)$ be nice. Then, there exists a (unique) extremal vertex $v\in\supp \Delta_L\subseteq\newtonalex$ and a positive functional $\phi_v\in\newtonalex^\vee$ such that
\begin{enumerate}
    \item $v=\rho(\beta)+\mathbf{a}(c_\iota)$.
    \item $\phi_v$ is uniquely minimized at $v$. In other words, $v$ is positively exposed by $\phi_v$.
 \end{enumerate}
\end{corollary}

Given the above, we propose the following replacement for \cite[Theorem 2]{Park21}.
\begin{theorem}\label{thm:main-thm} Let $\bd$ be a braid and $\iota$ an inversion datum such that the pair $(\bd,\iota)$ is nice.
Let $v$ and $\phi_v$ be the extremal vertex and the functional on the Newton polytope of the Alexander polynomial described in \cref{cor:inversion-datum-extremal-vertex}. Denote by $F_L^{(2v)}(\mathbf{x},q)\defeq F_{(\beta,\iota)}(\mathbf{x},q)$ the resulting inverted state sum. Then,
\begin{enumerate}

    \item $F_L^{(2v)}(\mathbf{x},q)$ is an invariant of the pair $(L,2v)$.
    \item $F_L^{(2v)}(\mathbf{x},q)\in \mathbf{x}^{-v}(\Z[q,q^{-1}])_{\mathcal{C}_v}[[\mathbf{x}]]$, where $\mathcal{C}_v$ is the cone on $\newtonalex$ centered at $v$. In other words, we can write
    \begin{equation*}
        F_L^{(2v)}(\mathbf{x},q)=\mathbf{x}^{-v}\sum_{\mathbf{k}\in \mathcal{C}_v} f_\mathbf{k}(q)\mathbf{x}^\mathbf{k}\,,
    \end{equation*}
    where $f_\mathbf{k}(q)\in\Z[q,q^{-1}]$.
    \item Setting $q=1+\hbar$, its power expansion at $\hbar=0$ agrees with the Melvin--Morton--Rozansky expansion of the colored Jones polynomials in $\mathbf{x}^{-v}\Z_{\mathcal{C}_v}[[\mathbf{x}]]$.
    \item $F_L^{(2v)}(\mathbf{x},q)$ is annihilated by the quantum $A$-ideal.
\end{enumerate}
\end{theorem}

A first consequence of the theorem is that we can make the following definitions.

\begin{definition} Let $L$ be an oriented link and $v$ an element of the dual Thurston polytope of $L$. Then, the pair $(L,v)$ is nice if there exists a braid $\bd$ and inversion datum $\iota$ such that $(\beta,\iota)$ is nice and $\frac{v}{2}=\mathbf{a}(c_\iota)+\rho(\beta)$.
\end{definition}
Define $\Hnice$ as follows:
\begin{equation}
    \Hnice \defeq \{v \mid v \in 2\mathrm{Vert}(\newtonalex) \text{ and }(L,v) \text{ is a nice pair}\}\subseteq 2\AL.
\end{equation}
Note that, by the inclusion of (twice) the Newton polytope of the Alexander polynomial into the dual Thurston polytope, we can identify $\Hnice$ as a discrete subset of the dual Thurston polytope. We say a link $L$ is nice if the set $\Hnice$ is not empty.

\begin{remark}\label{rmk:non-positively-exposed-FL}
    Our definition of niceness forces the $F_L$ series we obtain to be associated to positively exposed vertices of $\newtonalex$. It is possible to generalize the computation to vertices that are not positively exposed, by substituting the definition of niceness in \eqref{eq:degree-truncation-ss} with
    \begin{equation}
        Q_N^{\nu}(\iota)   \defeq \{s\in\Omega_\R(\iota) \mid \langle \nu, X(s)\rangle \leq N \}
    \end{equation}
    where $\nu\in \{\pm1\}^{\ell}$.
    In other words, we should be able to compute the series $F_L^{(v)}(\mathbf{x},q)$ associated to arbitrary extremal vertices of $\newtonalex$ by choosing a suitable $\nu$.
    Nevertheless, as in \cref{sec:positive-convention}, these series can be recovered by considering the partial orientation reversal of the link.
    More precisely, let $L^\nu$ be as above and let $v^{\nu}=(\nu_1v_1,\dots,\nu_{\ell}v_{\ell})$ be the corresponding vertex of $N(\Delta_{L^{\nu}})$. Then,
    \begin{equation}
        F_{L^\nu}^{(v^{\nu})}(\mathbf{x},q)
        \doteq
        F_L^{(v)}(\mathbf{x}^{\nu},q)~,
    \end{equation}
    where the overall sign and $q$-power will be dictated by their respective nice braids and inversion cycles. Consequently, allowing the generalized notion of niceness gives access to more series for a fixed orientation of $L$, but does not produce new series once all partial orientation reversals are considered.

    We expect the overall $q$-power to be the following: set
    \begin{equation}
        \lambda_\nu(L) \defeq \sum_{\substack{i\,:\,\nu_i=-1\\j\,:\,\nu_j=+1}}\operatorname{lk}(L_i,L_j).
    \end{equation}
    Thus $\lambda_\nu(L)$ is the linking number between the sublink whose orientation is reversed and its complement, computed before the reversal.
    We expect that, for some $\varepsilon_{v,\nu}\in\{\pm1\}$,
    \begin{equation}
        F_{L^\nu}^{(v^{\nu})}(\mathbf{x},q)
        = \varepsilon_{v,\nu} \, q^{-\lambda_\nu(L)}
        F_L^{(v)}(\mathbf{x}^{\nu},q)~.
    \end{equation}
    Note that this would follow from \cref{conj:maslov-links}.

\end{remark}

\subsection{Proofs}

Before proving \cref{prop:nice-iff-unital}, we set some definitions and establish some supporting lemmas.
The \emph{recession cone} of a nonempty polyhedron $P\subseteq\R^{n}$ is the cone
\[
\operatorname{rec}(P)=\{r\in \R^{n}: s+tr\in P\text{ for every }s\in P
\text{ and every }t\geq0\}.
\]
The cone records precisely the directions in which one can travel arbitrarily far without leaving $P$.
For the state polyhedron, write
\[
R(\iota)=\operatorname{rec}P_{\R}(\iota),
\]
where $P_{\R}(\iota)$ is the real relaxation of the state polyhedron.
\begin{lemma}\label{lemma:translated-cone}
For every locally admissible inversion datum $\iota$,
\[
P_{\R}(\iota)=s_0+R(\iota),
\qquad
P(\iota)=s_0+\bigl(R(\iota)\cap\Z^{E_\beta}\bigr),
\]
where $s_0$ is the ground state introduced in \cref{sec:inverted-state-sum}.
\end{lemma}

\begin{proof}
The defining relations of $P_{\R}(\iota)$ --- conservation \eqref{eq:inverted-state-conservation}, the sign bounds, the local nonvanishing inequalities for the extended $R$-matrix, and the ground condition on the opened strand --- are finitely many linear inequalities and equalities satisfied by any state $s$.
Collecting the inequalities into a matrix $A$ and vector $v$, and the equalities into $B$, this is exactly
\[
P_{\R}(\iota)=\{s\in\R^{E_\beta}: As\leq v,\ Bs=0\},
\qquad
R(\iota)=\{r\in\R^{E_\beta}: Ar\leq0,\ Br=0\},
\]
where $\leq$ between vectors means componentwise inequality, i.e.\ $(As)_k\leq v_k$ for every row $k$ of $A$. The description of $R(\iota)$ follows from the definition of the recession cone of a system of linear relations.

The ground state $s_0$ saturates the inequality $As\leq v$ by definition.

Both inclusions are now immediate matrix computations.
Suppose $r\in R(\iota)$, so $Ar\leq0$ and $Br=0$.
Then
\[
A(s_0+r)=As_0+Ar=v+Ar\leq v,
\qquad
B(s_0+r)=Bs_0+Br=w+0=w,
\]
so $s_0+r\in P_{\R}(\iota)$. Conversely, suppose $s\in P_{\R}(\iota)$, so $As\leq v$ and $Bs=0$, and set $r=s-s_0$. Then
\[
Ar=As-As_0=As-v\leq0,
\qquad
Br=Bs-Bs_0=w-w=0,
\]
so $r\in R(\iota)$ and $s=s_0+r\in s_0+R(\iota)$. Hence $P_{\R}(\iota)=s_0+R(\iota)$.

\end{proof}

For each component $L_i$, let $X_i(s)$ denote the minimum $x_i$-degree of the summand indexed by the state $s$.
The local degree formulas extend this to an affine function on the ambient real state space, which we write as
\[
X(s)=A(s)+d,
\]
where $A(s)\in\frac12\Z^\ell$ is linear and $d\in\frac12\Z^\ell$ is a constant vector\footnote{Notice that the vector $d$ actually agrees with $\rho(\beta)$, as defined in \eqref{eq:rho-beta-def}; however, this observation is not needed for this proof.}.
Thus, whenever $r\in R(\iota)$,
\[
X(s+tr)=X(s)+tA(r).
\]

\begin{lemma}\label{prop:recessioncriterion}
An inversion datum $\iota$ is nice if, and only if,
\begin{equation}\label{eq:recessioncriterion}
R(\iota)\cap
\{r:\left\langle \mathbb{1}, A(r)\right\rangle\leq0\}
=\{0\}.
\end{equation}
\end{lemma}
\begin{proof}
Suppose $Q_N(\iota)$ is bounded for some $N\in \N$, and a nonzero $r$ belongs to the intersection \eqref{eq:recessioncriterion}.
Choose $s\in Q_N(\iota)$, so that $\inner{\mathbb{1}}{X(s)}\leq N$.
Then, $s+tr\in P_{\R}(\iota)$ for every $t\geq0$, while
\[
\inner{\mathbb{1}}{X(s+tr)}=\inner{ \mathbb{1}}{ X(s)}+t\inner{ \mathbb{1}}{A(r)}
\leq \inner{\mathbb{1}}{X(s)}\leq N~,
\]
making $Q_N(\iota)$ unbounded.

Conversely, suppose that for every $N\in\N$ the set $Q_N(\iota)$ is non-empty and unbounded.
Choose $s\in Q_N(\iota)$ and a sequence $s_j\in Q_N(\iota)$ with $\lVert s_j-s\rVert\to\infty$.
After passing to a subsequence, the unit vectors
\[
r_j=\frac{s_j-s}{\lVert s_j-s\rVert}
\]
converge to a unit vector $r$.
Divide each defining affine equality or inequality of $P_{\R}(\iota)$, after subtracting its value at $s$, by $\lVert s_j-s\rVert$, and pass to the limit.
The resulting homogeneous equalities and inequalities show that $s+tr\in P_{\R}(\iota)$ for every $s\in P_{\R}(\iota)$ and $t\geq0$; hence $r\in R(\iota)$.
Finally,
\[
\inner{\mathbb{1}}{A(r_j)}
=\frac{\inner{\mathbb{1}}{ X(s_j)-X(s)}}{\lVert s_j-s\rVert}
\leq\frac{N-\inner{\mathbb{1}}{X(s)}}{\lVert s_j-s\rVert}.
\]
The right-hand side tends to zero.
Since $r_j\to r$ and $A$ is continuous, $A(r_j)\to A(r)$; hence $\inner{\mathbb{1}}{A(r)}\leq0$.
Since $r$ is a unit vector, it is nonzero, which proves that the intersection \eqref{eq:recessioncriterion} is nontrivial.
\end{proof}

Equation \eqref{eq:recessioncriterion} gives a necessary and sufficient condition for niceness that does not depend on the cutoff $N_*$.
We therefore have the following.
\begin{corollary}\label{cor:one-cutoff}
Suppose that $Q_{N_*}(\iota)$ is nonempty and bounded for some cutoff $N_*\in\N$.
Then, $Q_N(\iota)$ is bounded for every $N\in\N$.
\end{corollary}

\begin{lemma}
\label{lemma:degree-recession-cone-pointed}
Suppose that $\iota$ is nice.
Then the image cone $A(R(\iota))$ is line-free.
\end{lemma}

\begin{proof}
First observe that the homogeneous sign bounds give
\[
 r(e)\geq 0\quad\text{when }\iota(e)=+,
 \qquad
 r(e)\leq 0\quad\text{when }\iota(e)=-.
\]
Consequently, $R(\iota)$ itself is line-free.

Now suppose that both $v,-v\in A(R(\iota))$.
Choose $r_+,r_-\in R(\iota)$ such that $A(r_+)=v$ and $A(r_-)=-v$.
Since $R(\iota)$ is a cone, $r_++r_-\in R(\iota)$, and
\[
 A(r_++r_-)=0.
\]
The recession criterion and niceness imply that $r_++r_-=0$.
Hence $r_+=-r_-$.
Because $R(\iota)$ is line-free, $r_+=r_-=0$, and therefore $v=0$.
Thus $A(R(\iota))$ contains no nontrivial line.
\end{proof}

\begin{lemma}
\label{lemma:strict-separation-line-free-cones}
Let $C_1,C_2\subseteq\mathbb R^n$ be line-free rational polyhedral cones such that $C_1\cap C_2=\{0\}$.
Then there is a linear functional $\phi\colon\mathbb R^n\to\mathbb R$ such that
\[
  \phi(v)>0\quad\text{for every }v\in C_1\setminus\{0\},
  \qquad
  \phi(w)<0\quad\text{for every }w\in C_2\setminus\{0\}.
\]
\end{lemma}

\begin{proof}
Consider the rational polyhedral cone
\[
  D=C_1+(-C_2).
\]
We first show that $D$ is line-free.
Suppose that $z,-z\in D$.
We may write
\[
  z=x_1-y_1,\qquad -z=x_2-y_2,
  \qquad x_1,x_2\in C_1,\quad y_1,y_2\in C_2.
\]
It follows that $x_1+x_2=y_1+y_2$.
This vector belongs to $C_1\cap C_2$, so it is zero.
Since $C_1$ and $C_2$ are pointed, $x_1=x_2=y_1=y_2=0$, and hence $z=0$.

A pointed polyhedral cone admits a linear functional which is strictly positive away from the origin: equivalently, the dual cone $D^\vee$ has nonempty interior, and every $\phi\in\operatorname{int}(D^\vee)$ satisfies $\phi(u)>0$ for all $u\in D\setminus\{0\}$.
Choose such a $\phi$.
Since $C_1\subseteq D$, it is strictly positive on $C_1\setminus\{0\}$.
If $0\neq w\in C_2$, then $-w\in D\setminus\{0\}$, and therefore $\phi(w)=-\phi(-w)<0$.
\end{proof}
In particular, if either $C_1$ or $C_2$ is $\R^{\ell}_{\leq 0}$ then the functional $\phi$ can be chosen in $\R^{\ell}_{>0}$.
We are now ready to prove \cref{prop:nice-iff-unital}.
\begin{proof}[Proof of \cref{prop:nice-iff-unital}]
Suppose first that $(\bd,\iota)$ is nice.
By \cref{prop:recessioncriterion}, niceness implies that \cref{eq:recessioncriterion} holds.
It follows that the two rational polyhedral cones $A(R(\iota))$ and $-\mathbb R_{\geq0}^{\ell}$ meet only at the origin.
The first is line-free by \cref{lemma:degree-recession-cone-pointed}, while the second is trivially line-free.
We may therefore apply \cref{lemma:strict-separation-line-free-cones} to obtain a linear functional $\phi$ such that
\[
   \phi(A(r))>0\quad\text{for every }0\neq r\in R(\iota),
   \qquad
   \phi(w)<0\quad\text{for every }
   0\neq w\in-\mathbb R_{\geq0}^{\ell}.
\]
In particular, $\phi(-e_j)<0$, and hence $\phi(e_j)>0$, for every $j=1,\ldots,\ell$.
Thus $\phi$ lies in the interior of the positive orthant.  Since the cones
are rational polyhedral, we may choose $\phi$ rational; multiplying by a
positive integer, we may assume that $\phi\in\mathbb Z_{>0}^{\ell}$.

We next compare this functional with the weights in the inverted model.
Let $\widetilde c$ be a nonempty simple multicycle in that model, and let $r_{\widetilde c}$ be its occupation vector,
\begin{equation}
    (r_{\widetilde c})_k = \begin{cases}
    1 & \text{if }\widetilde c\text{ passes through the braid segment labeled by }k \\
    0 & \text{otherwise}
    \end{cases}~.
\end{equation}
This vector is nonzero and belongs to $R(\iota)$: adjoining any number of further copies of $\widetilde c$ preserves all the homogeneous state constraints and hence determines a ray in the state polyhedron.

By definition of $A$ we find that every exponent $\mathbf v$ in the support of the expanded inverted weight has the form
\[
   \mathbf v=A(r_{\widetilde c})+\mathbf u,
   \qquad
   \mathbf u\in\mathbb Z_{\geq0}^{\ell}.
\]
Since every coordinate of $\phi$ is positive and $\phi(A(r_{\widetilde c}))>0$, it follows that
\begin{equation}\label{eq:nice-to-unital-positive-minimum}
   \min_{\mathbf v\in\supp\weightinv{\widetilde c}}
      \phi(\mathbf v)
   \geq \phi(A(r_{\widetilde c}))>0.
\end{equation}

Now take $c\in\mathcal Q\setminus\{c_\iota\}$.
Under the inversion correspondence, $c$ determines a nonempty inverted multicycle $\widetilde c$, and \eqref{eq:Winv-rel-W} gives, up to sign,
\[
   |\weightinv{\widetilde c}|
   = \frac{\weight{c}}{\weight{c_\iota}}.
\]
The lowest $\phi$--degree is additive under multiplication of nonzero Laurent series in the prescribed cone.
Hence
\[
   \mathbf{a}(c)-\mathbf{a}(c_\iota)=\min_{\mathbf v\in\supp\weight{c}}\phi(\mathbf v)
   -\min_{\mathbf v\in\supp\weight{c_\iota}}\phi(\mathbf v)
   =\min_{\mathbf v\in\supp\weightinv{\widetilde c}}
      \phi(\mathbf v)>0.
\]
Therefore, $\mathbf a(c_\iota)$ is the unique $\phi$--minimum on $\{\mathbf a(c)\mid c\in\mathcal Q\}$.

Conversely, suppose that $\phi\in\mathbb Z_{>0}^{\ell}$ has a unique minimum on $\{\mathbf a(c)\mid c\in\mathcal Q\}$, achieved at $\mathbf a(c_\iota)$.
Let $R(\iota)$ be the recession cone of the real state polyhedron and let $A=(A_1,\ldots,A_\ell)$ be the linear part of its componentwise minimum-degree map.  By \cref{prop:recessioncriterion}, it is enough to show that \eqref{eq:recessioncriterion} holds.

Every recession vector $r\in R(\iota)$ has a decomposition
\begin{equation}\label{eq:unital-to-nice-cycle-decomposition}
   r=\sum_k\lambda_k r_{\widetilde c_k},
   \qquad \lambda_k\geq0,
\end{equation}
where the $\widetilde c_k$ are nonempty inverted primitive multicycles, which by the inversion correspondence determine a multicycle $c_k$.

The unique-minimum hypothesis now implies
\begin{align*}
   \phi\bigl(A(r_{\widetilde c_k})\bigr)
   &=\min_{\mathbf v\in\supp\weight{c_k}}\phi(\mathbf v)
     -\min_{\mathbf v\in\supp\weight{c_\iota}}\phi(\mathbf v) \\
   &=\phi(\mathbf a(c_k))-\phi(\mathbf a(c_\iota)) > 0~.
\end{align*}
It follows from \eqref{eq:unital-to-nice-cycle-decomposition} that
\[
   \phi(A(r))>0
   \qquad\text{for every }0\neq r\in R(\iota).
\]
Therefore, if a nonzero $r\in R(\iota)$ satisfied
$A_j(r)\leq0$ for every $j$, positivity of $\phi$ on the coordinate rays
would give
\[
   \phi(A(r))=\sum_{j=1}^\ell A_j(r)\phi(e_j)\leq0,
\]
a contradiction.
Hence \eqref{eq:recessioncriterion} holds, and the
recession criterion proves that $(\bd,\iota)$ is nice.
\end{proof}

To prove the main Theorem, we follow the notation already set up in \cref{sec:random-walk-inverted-state-sum}. Denote by $\mathrm{Z}^{\mathrm{inv}}(\mathbf{x})$ the result of the inverted state sum $\mathrm{Z}^{\mathrm{inv}}(\boldsymbol{\alpha},\boldsymbol{\beta},\boldsymbol{\gamma})$ under \eqref{eq:rules-abg}. The relation between $\mathrm{Z^{\mathrm{inv}}}(\mathbf{x})$ and the inverted state sum used to define $F_L(\mathbf{x},q)$ in \eqref{eq:FL-inverted-state-sum} is given by \eqref{eq:parametrized-inverted-state-sum-vs-usual}.

Denote by $(\Delta_L)_v^{-1}$ the inverse of the Alexander polynomial around $v$ as in \eqref{eq:inverse-polynomial}, with the coefficient ring $R=\Z$.

\begin{lemma} $(1-x_1)\mathrm{Z}^{\mathrm{inv}}(\mathbf{x})=\mathbf{x}^{\rho(\beta)}(\Delta_L)_v^{-1}$ as elements of $\mathbf{x}^{\rho(\beta)-v}\Z_{\mathcal{C}_v} [[\mathbf{x}]]$.
\end{lemma}

\begin{proof}[Proof of the lemma]

Recall that we can write the weighted sum of primitive cycles as
\begin{equation}\label{eq:det-B-inv-weights}
    \det{(I-\mathcal{B}^{\mathrm{inv}})}=\sum_{\tilde{c}\in\mathcal{Q}(\iota)} (-1)^{\mathbf{s}(\tilde{c})}\weightinv{\tilde{c}},
\end{equation}
and the sum over all cycles as
\begin{equation}
    \mathrm{Z}^{\mathrm{inv}}(\mathbf{x}) = \frac{1}{\weight{c_\iota}\det{(I-\mathcal{B}^{\mathrm{inv}})}} = \frac{1}{\weight{c_\iota}}\sum_{\tilde{c}\in\mathcal{P}(\iota)} \weightinv{\tilde{c}}.
\end{equation}

For every simple cycle $\tilde{c}$ in the random walk model induced by the inverted state sum, there exists a simple cycle $c$ in the jump-down model for which \eqref{eq:Winv-rel-W} holds (\cref{sec:random-walk-inverted-state-sum}). Using the definition of $\mathbf{a}(c)$ and $\mathbf{b}(c)$ in \eqref{eq:exponent-weight-Alex-model}, we can write
\begin{equation}\label{eq:weight-inv-exponent}
    \weightinv{\tilde{c}} = \epsilon(c)\epsilon(c_{\iota}) \mathbf{x}^{\mathbf{a}(c)-\mathbf{a}(c_\iota)} \prod_{i=1}^\ell (1-x_i)^{b_i(c)-b_i(c_\iota)}.
\end{equation}
Here $b_i(c)-b_i(c_\iota)\in\Z$ can take any value. Under the assumptions of \cref{sec:inverted-R-matrices}, any factor $(1-x_i)^{-1}$ is identified with its expansion at $x_i=0$, so that
\[
    \weightinv{\tilde{c}}\in \mathbf{x}^{\mathbf{a}(c)-\mathbf{a}(c_\iota)}\Z[[\mathbf{x}]]~.
\]

Therefore, for any functional induced by an element of $\Z_{>0}^\ell$, its minimum in the set $\supp{\weightinv{\tilde{c}}}$ will be attained at ${\mathbf{a}(c)-\mathbf{a}(c_\iota)}$. By \cref{prop:nice-iff-unital}, the functional $\phi_v\in\Z^\ell_{>0}$ attains its unique minimum at $v=\rho(\beta)+\mathbf{a}(c_\iota)$, and therefore

\[\phi_v\cdot((\rho(\beta)+\mathbf{a}(c))-(\rho(\beta)+\mathbf{a}(c_\iota)))=\phi_v\cdot(\mathbf{a}(c)-\mathbf{a}(c_\iota))> 0 ~\text{ for any }c\in\mathcal{Q}, ~c\neq c_\iota.\]

Define the cones
\begin{equation*}
    \mathcal{C}_{\mathbf{a}(c_\iota)}:= \mathrm{Cone}\;\{\mathbf{a}(c)-\mathbf{a}(c_\iota) \mid c\neq c_\iota\} \hspace{1.25em}\text{and}\hspace{1.25em}\mathcal{C}_+ := \{(z_1,\dots,z_\ell)\mid z_i\geq 0, i=1,\dots,\ell\}~.
\end{equation*}
Then,
\begin{equation}\label{eq:support-inverted-weight}
    \supp \weightinv{\tilde{c}} \subseteq \mathcal{C}_{\mathbf{a}(c_\iota)} + \mathcal{C}_+
\end{equation}
for every $\tilde{c}\in\mathcal{Q}(\iota)$, where $\mathcal{C}_{\mathbf{a}(c_\iota)} + \mathcal{C}_+$ denotes the Minkowski sum of the cones.

The cone $\mathcal{C}_{\mathbf{a}(c_\iota)}+\mathcal{C}_+$ is line-free, since it still has a unique minimum under the functional $\phi_v$. Therefore, the set of functions $\Z_{\mathcal{C}_{\mathbf{a}(c_\iota)} + \mathcal{C}_+}[[\mathbf{x}]]$ forms a ring, and there is a ring inclusion $\Z_{\mathcal{C}_v}[[\mathbf{x}]]\subseteq \Z_{\mathcal{C}_{\mathbf{a}(c_\iota)} + \mathcal{C}_+}[[\mathbf{x}]]$.

By \cref{lemma:B-matrix-inverted-Alexander},
\begin{equation}\label{eq:det-and-alex}
    \mathbf{x}^{-\mathbf{a}(c_\iota)} \frac{\weight{c_\iota}}{1-x_1}\det{(I-\mathcal{B}^\mathrm{inv})}
    = (-1)^{\mathbf{s}}\mathbf{x}^{-v} \Delta_L(\mathbf{x})~,
\end{equation}
and from \eqref{eq:det-B-inv-weights} and \eqref{eq:support-inverted-weight}, we deduce that
\begin{equation}
    \mathbf{x}^{-\mathbf{a}(c_\iota)}\weight{c_\iota} \in \Z_{\mathcal{C}_+}[[\mathbf{x}]]
    \hspace{1.25em}\text{and}\hspace{1.25em}
    \frac{1}{1-x_1}\det{(I-\mathcal{B}^\mathrm{inv})}\in\Z_{\mathcal{C}_{\mathbf{a}(c_\iota)}+\mathcal{C}_+}[[\mathbf{x}]]~.
\end{equation}
The left-hand side of \eqref{eq:det-and-alex} is an element of $\Z_{\mathcal{C}_{\mathbf{a}(c_\iota)}+\mathcal{C}_+}[[\mathbf{x}]]$, while the right-hand side is an element of $\Z_{\mathcal{C}_v}[[\mathbf{x}]]$. It follows that both live in $\Z_{\mathcal{C}_v}[[\mathbf{x}]]$, and since the right-hand side is a unit in this ring, the left-hand side is too. Moreover,
\begin{equation}
    \mathbf{x}^{\mathbf{a}(c_\iota)}\mathrm{Z^{\mathrm{inv}}}(\mathbf{x})= \frac{\mathbf{x}^{\mathbf{a}(c_\iota)}}{\weight{c_\iota}} \sum_{\tilde{c}\in\mathcal{P}(\iota)} \weightinv{\tilde{c}} \in \mathbb{Z}_{\mathcal{C}_{\mathbf{a}(c_\iota)}+\mathcal{C}_+}[[\mathbf{x}]]~.
\end{equation}
Since $\mathrm{Z^{\mathrm{inv}}}$ is related to \eqref{eq:det-and-alex} via \eqref{eq:Zinv-and-inverse-det-Binv}, we conclude that $\mathbf{x}^{\mathbf{a}(c_\iota)}\mathrm{Z^{\mathrm{inv}}}(\mathbf{x})$ is invertible around $\mathbf{a}(c_\iota)=v$, and we obtain:
\begin{equation}
    (1-x_1) \mathbf{x}^{v-\rho(\beta)} \mathrm{Z^{\mathrm{inv}}}(\mathbf{x})
    = (1-x_1)\left(\mathbf{x}^{-\mathbf{a}(c_\iota)} \weight{c_\iota}\det{(I-\mathcal{B}^\mathrm{inv})}\right)_v^{-1} = (-1)^{\mathbf{s}} \left(\mathbf{x}^{-v}\Delta_L(\mathbf{x})\right)^{-1}_v
\end{equation}
Therefore,
\begin{equation*}
    (1-x_1)\mathrm{Z^{\mathrm{inv}}}(\mathbf{x})\in\mathbf{x}^{\rho(\beta)-v}\Z_{\mathcal{C}_v}[[\mathbf{x}]]
\end{equation*}
and it agrees with the inverse of the Alexander polynomial, up to the monomial $\mathbf{x}^{\rho(\beta)}$.
\end{proof}

The previous Lemma, together with \eqref{eq:FL-parametrized-state-sum}, gives
\begin{equation}
    F_{(\beta,\iota)}(\mathbf{x},1) = (-1)^{\mathbf{s}+1} x^{-v}(x^{-v}\Delta_L(\mathbf{x}))^{-1}_v =(-1)^{\mathbf{s}+1}(\Delta_L)^{-1}_v(\mathbf{x})~.
\end{equation}

The rest of the proof is the same as the proof of Theorems 1 and 2 of \cite{Park21}: There is a differential operator $D_j$ in the parameters $\boldsymbol{\alpha},\boldsymbol{\beta},\boldsymbol{\gamma}$ such that the $\hbar^j$-coefficient of the MMR expansion \eqref{eq:MMR-link} agrees with the action of $D_j$ on $\mathrm{Z}(\boldsymbol{\alpha},\boldsymbol{\beta},\boldsymbol{\gamma})$. Similarly, the corresponding $\hbar^j$-coefficient of the inverted state sum is given by the action of $D_j$ on $\mathrm{Z^{\mathrm{inv}}}(\boldsymbol{\alpha},\boldsymbol{\beta},\boldsymbol{\gamma})$, modulo the overall sign. This finishes the proof.

\subsection{Leading term}

Denote by $\mathfrak{i}(L)$ the index of a link $L$, defined as
\begin{equation}
    \mathfrak{i}(L) = \frac{\ell-\chi(L)}{2}
\end{equation}
where $\chi(L)$ is the minimal Euler characteristic of all Seifert surfaces bounded by $L$. For a knot $K$, $\chi(K)=1-2g(K)$ and $\mathfrak{i}(K)=g(K)$.

Define the minimal first Betti number of $L$ as
\begin{equation}
    b_1(L)\defeq \min{\{b_1(S)=\rank{H_1(S,\Z)} \mid (S,\partial S)\hookrightarrow (M,\partial M) \text{ and }\partial S=L\}}.
\end{equation}
It follows that $b_1(L)=1-\chi(L)=2g(L)+\ell-1$.

Let $O_{\phi}(\mathbf{x}^{v})\defeq\{f \mid \inner{\phi}{v} < \inner{\phi}{w}  \text{ for every }w\in\supp{f}\}$.

\begin{proposition}\label{prop:L} Let $v\in\Hnice$. Let $\phi_v$ be the positive functional of \cref{cor:inversion-datum-extremal-vertex}. Then, there exists an integer $\alpha_v(L)$ and sign $\varepsilon_v\in\{\pm 1\}$ such that
\begin{equation}\label{eq:struct-nice-FL}
    F_L^{(v)}(\mathbf{x},q)=\varepsilon_v \, q^{\alpha_v(L)} \mathbf{x}^{-\frac{v}{2}} + O_{\phi_v}(\mathbf{x}^{-\frac{v}{2}}).
\end{equation}
\end{proposition}

\begin{proof} The same arguments used to prove \cref{thm:main-thm} work here. In particular, it follows that the ground state is the only state contributing to $\mathbf{x}^{-\frac{v}{2}}$ in the inverted state sum, it contributes a monomial in $q$ as shown in \cref{tab:xq-pos-contributions,tab:xq-neg-contributions}, and the structure in \eqref{eq:struct-nice-FL} follows since $F_L^{(v)}(\mathbf{x},q)\in x^{-\frac{v}{2}}(\Z[q,q^{-1}])_{\mathcal{C}_\frac{v}{2}}[[\mathbf{x}]]$. Alternatively, we could have proven this using optimization theory results, similar to \cite{OSSS25}.
\end{proof}

\begin{proposition}\label{prop:fib-index-xdeg} Let $L$ be a fibered link and $v\in \Hnice$ be the unique minimizer of the positive functional induced by the canonical cohomology class $\mathbb{1}^*$ in $\mathrm{Vert}(\newtonalex)$. Then, $\inner{\mathbb{1}^*}{v}=\chi(L)$.
\end{proposition}
If $L$ is a knot, we recover the statement $v=\inner{\mathbb{1}^*}{v}=1-2g(L)$, as expected \cite{OSSS25}.

\begin{proof} Since $L$ is a fibered link, $v$ is an extremal vertex of $B_T^\vee$, and $\frac{1}{2}(-v+\mathbb{1})$ is the unique maximizer of the functional induced by $\mathbb{1}^*\in H^1(M,\Z)$ in $\HFLpolytope$. From \cite{Ni2006} and \cite[Theorem 1.1]{OzsSza2008-2}, we have that
\begin{equation}
    \mathfrak{i}(L)=\inner{\mathbb{1}^*}{\frac{1}{2}(-v+\mathbb{1})}= \frac{1}{2}\left(\inner{\mathbb{1}^*}{-v} + |L|\right).
\end{equation}
It follows that $\inner{\mathbb{1}^*}{-v}=-\chi(L)$.
\end{proof}

Whenever $L$ is a homogeneous braid link, there is a unique $v\in\Hnice$ (cf. \cref{prop:hom-braid-alex,thm:main-thm}), and it is the one minimizing the functional of the Proposition above. In that case, define $\alpha(L)\defeq\alpha_v(L)$.

\begin{proposition}\label{prop:alpha-Hopf-homogenoeus-link}
Let $L$ be a homogeneous braid link. Then, $\alpha(L)$ is an invariant of $L$, and
\begin{equation}\label{eq:L-homogeneous-braid}
    \alpha(L) = \frac{b_1(L)}{2}-\lambda(L).
\end{equation}

\end{proposition}

\begin{proof} Let $\bd\in B_n$ be a homogeneous braid representative for $L$. Denote by $N_\pm$ its number of positive and negative crossings, $N:=N_++N_-$ its total number of crossings, and $k_{\pm}$ its number of positive and negative generators, such that $n-1=k_++k_-$.

Let us first focus on the total lowest $x$-degree. \cref{tab:xq-pos-contributions,tab:xq-neg-contributions} show that every crossing contributes $\frac12$ to the ground state, and every strand contributes $-\frac12$ to \eqref{eq:trace-factor-link}, when collapsing all $\mathbf{x}$-variables into one. Then, we have that
\begin{equation}
    -\sum_{i=1}^\ell v_i=\deg_x F_L(x,\dots,x,q)=\frac12(N-n).
\end{equation}
In a similar way, we deduce that the associated $q$-degree is given by
\begin{equation}
    \alpha(L)=\frac{1}{2}(N_+-N_-) -\frac{1}{2}(k_+-k_-)
\end{equation}
Let $H^{\pm}$ denote the positive and negative Hopf bands. Since these are twisted annuli, we have that $g(H^{\pm})=0$ and $b_1(H^{\pm})=1$. Furthermore, $\lambda(H^+)=0$ and $\lambda(H^-)=1$.

Stallings \cite{Stallings1978} proved that, if $\bd$ is a homogeneous braid link, it can be obtained by plumbing $N_+-k_+$ positive bands and $N_--k_-$ negative bands and taking its boundary. Since both the first Betti number \cite{Gabai1983Murasugi,Gabai1985MurasugiII} and the Hopf invariant \cite{neumann1987unfoldings} of a link are additive under this operation, the result follows from:
\begin{equation}
        b_1(L)= \left(N_{+}-k_+\right)b_1(H^+)+\left(N_{-}-k_-\right)b_1(H^-) = N - (n-1),
\end{equation}
and
\begin{equation}
    \lambda(L)=(N_+-k_+)\lambda(H^+)+(N_--k_-)\lambda(H^-)=N_--k_-.
\end{equation}
\end{proof}

\begin{table}
    \centering
    \begin{tabular}{c|*{6}{c}}

      \begin{tikzpicture}[scale = .25, rotate = 90, baseline=-2pt]
    \draw (-1,-1) -- (-0.2,-0.2);
    \draw (0.2,0.2) -- (1,1);
    \draw (-1,1) -- (1,-1);
    \end{tikzpicture}
    & \stackanchor{++}{++} & \stackanchor{+-}{-+} & \stackanchor{-+}{+-} & \stackanchor{-+}{-+} & \stackanchor{--}{--} \\
    \midrule

    $R(s_0)$ & $x^\frac14 y^\frac14 q^\frac12$ & $x^\frac14 y^{-\frac14}$ & $x^{-\frac14}y^{\frac14}$ & $y^{\frac12}q^\frac12$ & $x^{-\frac14}y^{-\frac14}q^{\frac12}$\\
    \end{tabular}
    \caption{Minimal $x,y$-exponents of $R$ evaluated at the state $s_0$, in terms of the value of $\id$ at a crossing $c$.}
    \label{tab:xq-pos-contributions}
\end{table}

\begin{table}
    \centering
    \begin{tabular}{c|*{6}{c}}

    \begin{tikzpicture}[scale = .25, baseline=-2pt]
    \draw (-1,-1) -- (-0.2,-0.2);
    \draw (0.2,0.2) -- (1,1);
    \draw (-1,1) -- (1,-1);
    \end{tikzpicture}

    & \stackanchor{--}{--} & \stackanchor{+-}{-+} & \stackanchor{-+}{+-} & \stackanchor{+-}{+-} & \stackanchor{++}{++} \\
    \midrule

    $R^{-1}(s_0)$ & $x^\frac14 y^\frac14 q^{-\frac12}$ & $x^{-\frac14} y^{\frac14} $ & $x^{\frac14}y^{-\frac14}$ & $x^{\frac12}q^{-\frac12}$ & $x^{-\frac14}y^{-\frac14}q^{-\frac12}$\\
    \end{tabular}
    \caption{Minimal $x,y$-exponents of $R^{-1}$ evaluated at the state $s_0$, in terms of the value of $\id$ at a crossing $c$.}
    \label{tab:xq-neg-contributions}
\end{table}

In the general case, we conjecture that $\alpha_v(L)$ can be obtained from known quantities in the following way.

\begin{conjecture}\label{conj:maslov-links}
Given an oriented, nice link $L$, let $v\in \Hnice$. Assume that $v\in B_T^\vee$ is a fibered vertex positively exposed by $\phi\in H^1(M,\Z)$. Then,
\begin{equation}
    \alpha_{v}(L) = \frac{1-\inner{\mathbb{1}^*}{v}}{2} - M(h)
\end{equation}

where $h=\frac{1}{2}(-2v+\mathbb{1})$ and $\HFL(L,h)=\mathbb{F}_{(M(h))}$.

In particular, if $L$ is fibered and $v$ is the unique minimizer of the functional induced by the canonical cohomology class $\mathbb{1}^*\in H^1(M,\Z)$,
\begin{equation}
    \alpha_v(L)= \frac{b_1(L)}{2} - \lambda(L)~.
\end{equation}
\end{conjecture}
This conjecture is supported by computations: we verified the conjectured identity for 854 distinct pairs \((L,v)\), with \(v\in\mathbb{G}(L)\), drawn from 630 oriented two- and three-component link variants with at most ten
crossings.

\subsection{The single-variable series}

\begin{theorem}\label{thm:single-var-FL} Let $L$ be an oriented link and $v\in \Hnice$. Assume that $v$ is the unique minimizer of the functional induced by the canonical cohomology class $\mathbb{1}^*$. Then, the evaluation $\mathbf{x}=(x,\dots,x)$ gives a well-defined power series
\begin{equation}
    F_L^{(v)}(x,q)\in x^{-\frac12\inner{\mathbb{1}^*}{v}}\Z[q,q^{-1}][[x]].
\end{equation}
\end{theorem}

In particular, the conditions of the theorem are satisfied by any fibered link (cf. \cref{thm:fibered-link-HFL}). Therefore, given a nice fibered link and such a $v\in\Hnice$, we can define
\begin{equation}
    F_L(x,q)\defeq F_L^{(v)}(x,\dots,x,q),
\end{equation}
and it follows from \cref{prop:fib-index-xdeg} that
\begin{equation}
    F_L(x,q)\in x^{-\frac{\chi(L)}{2}}\Z[q,q^{-1}][[x]]~.
\end{equation}

\begin{proof} Let $\mathcal{C}_\frac{v}{2}$ be the cone in the statement of \cref{thm:main-thm} with generators $z_1,\dots,z_s$, such that every element $z\in\mathcal{C}_\frac{v}{2}$ can be written as $z=\sum_{i=1}^s n_iz_i$ for some $n_1,\dots,n_s\in\Z_{\geq 0}$. Every generator $z_i$ must satisfy $\inner{\mathbb{1}^*}{z_i}>0$, since there exist $w_i\in\supp \Delta_L$ such that $z_i=w_i-v$ for every $i=1,\dots, s$, and $\inner{\mathbb{1}^*}{w_i} > \inner{\mathbb{1}^*}{\frac{v}{2}}$ by assumption.

For every $z\in\mathcal{C}_{\frac{v}{2}}$ as above,
\begin{equation}
    \inner{\mathbb{1}^*}{z}=\sum_{i=1}^s n_i\inner{\mathbb{1}^*}{z_i}>0 \hspace{1.25em}\text{ if }~n_i>0~\text{ for some }~1\leq i\leq s.
\end{equation}
Therefore, for every $m\in\left(\frac{\ell}{2}+\Z\right)_{\geq 0}$,
\begin{equation}\label{eq:single-variable-cardinality}
    0<|\{z\in\mathcal{C}_{\frac{v}{2}}\mid \inner{\mathbb{1}^*}{z} =m\}|< \infty~,
\end{equation}
and, if $m=0$, we have $|\{z\in\mathcal{C}_{\frac{v}{2}}\mid \inner{\mathbb{1}^*}{z} =0\}|=1$.

Write
\begin{equation}
    F_L^{(v)}(\mathbf{x},q)=\mathbf{x}^{-\frac{v}{2}}\sum_{z\in\mathcal{C}_\frac{v}{2}} f_{z}(q) \mathbf{x}^{z},
\end{equation}
where $f_z(q)\in\Z[q,q^{-1}]$. The evaluation of $\mathbf{x}=(x,\dots,x)$ sends $F_L^{(v)}(\mathbf{x},q)$ to
\begin{equation}\label{eq:single-variable-v}
\begin{split}
    F_L^{(v)}(x,q) &\defeq F_L^{(v)}(x,\dots,x,q) = x^{-\frac12\inner{\mathbb{1}^*}{v}} \sum_{z\in\mathcal{C}_{\frac{v}{2}}} f_{z}(q) x^{\inner{\mathbb{1}^*}{z}}~,
\end{split}
\end{equation}
from which we obtain
\begin{equation}
    F_L^{(v)}(x,q)=x^{-\frac12\inner{\mathbb{1}^*}{v}} \left(f_{\mathbf{0}}(q)+\sum_{m\in\left(\frac{\ell}{2}+\Z\right)_{>0}} \left(\sum_{\{z\in \mathcal{C}_{\frac{v}{2}} \mid \inner{\mathbb{1}^*}{z}=m\}} f_{z}(q)\right) x^{m}\right).
\end{equation}
By \eqref{eq:single-variable-cardinality}, the $x^m$-coefficient of $F_L^{(v)}(x,q)$ is a finite Laurent polynomial in $q$.

\end{proof}

\begin{remark} Given a non-fibered oriented link, there might be multiple vertices $v\in\Hnice$ which minimize the pairing $\inner{\mathbb{1}^*}{v}$. In this case, we expect that the evaluation $\mathbf{x}=(x,\dots,x)$ gives a Laurent power series
\begin{equation*}
    F_L(x,q)\defeq F_L^{(v)}((x,\dots,x),q) \in x^{-\frac{\chi(L)}{2}}\Z((q))[[x]]
\end{equation*}
which does not depend on the choice of such a $v$.
\end{remark}

\subsection{Topological interpretation}

We expect that the notion of niceness is intrinsically related to that of fibrations of the link complement.
\begin{conjecture}\label{conj:nice-fibered-face-Thurston} Let $v\in \newtonalex$ be a positively exposed vertex. Then, the pair $(L,2v)$ is nice if, and only if, $2v\in B_T^\vee$ is a fibered vertex.

In other words,
\begin{equation}
    \Hnice=\{v\in\mathrm{Vert}(B_T^\vee)\mid v \text{ is fibered}~\}\,.
\end{equation}
\end{conjecture}

Therefore, we conjecture that the Gukov--Manolescu invariant of a link $L$ is a family of formal Laurent power series in $\ell+1$ variables labeled by fibered vertices of the dual Thurston polytope or, equivalently, fibered faces of the Thurston polytope.

\begin{remark} An equivalent statement to the conjecture is: The pair $(L,2v)$ is nice if, and only if, $\mathrm{int}(\mathcal{C}_v^\vee)=\R_{>0}\,\mathrm{relint}(F_{2v})$, where $F_{2v}=(2v)^\vee$ is a fibered top-dimensional face of the Thurston polytope and $\mathrm{relint}(F_{2v})$ denotes its relative interior. This reformulation will come in handy in \cref{sec:surgery}.
\end{remark}

\begin{remark}\label{rmk:conj-nice-HFL} Given \cref{thm:HFL-fibered-classes}, another reformulation of the conjecture is the following: The pair $(L,v)$ is nice if, and only if, $h\defeq \frac{1}{2}(v-\mathbb{1})\in \HFLpolytope$ is an extremal vertex and
satisfies $\rank{\HFL(L,h)}=1$.
\end{remark}

\begin{remark} Given a link $L$ and a fibered cohomology class $\phi\in H^1(M,\Z)$ for which $\inner{\phi}{\mu_i}>0$ for all $1\leq i\leq \ell$, the pair $(L,\phi)$ induces a rational open book decomposition of $S^3$, as defined by \cite{BakerEtnyreVanHornMorris2012}\footnote{We note that their notion of \emph{fibered link} is different than ours, and only requires the link complement to admit a fibration over the circle.}. In the same paper, they consider \emph{resolutions} of rational open book decompositions, by substituting $L$ with the $(\mathbf{p},\mathbf{q})$-cable of $L$, where $p_i=\inner{\phi}{\mu_i}$ and $\mathbf{q}$ is chosen suitably. By \cite[Theorem 5.1]{BakerEtnyreVanHornMorris2012}, the resolution of $(L,\phi)$ gives an honest open book decomposition on the cable link, which is a fibered link. Moreover, both (rational or integral) open book decompositions support the same contact structure.

Inspired by the arguments in \cite[Proposition 3.3]{OzsSza2008-2}, we believe that the Hopf invariant of this contact structure is precisely what is captured by the Maslov grading appearing in \cref{conj:maslov-links}, and therefore also by $\alpha_v(L)$. This provides further support for \cref{conj:nice-fibered-face-Thurston}.
\end{remark}

The hard direction to prove is the `if' part of \cref{conj:nice-fibered-face-Thurston}. Indeed, by \cref{rmk:conj-nice-HFL}, and given the connection between fibered faces and integral expansions of $\Delta_L^{-1}$, we can immediately deduce the `only if' direction of the conjecture for alternating links.

\begin{proposition} Let $L$ be an alternating link. Then, $(L,v)$ is nice only if $v\in B_T^\vee$ is a fibered vertex.

\end{proposition}

\begin{proof}
\cref{cor:Alex-monic-fibered} shows that, otherwise, the expansion of $\Delta_L^{-1}$ around $\frac{v}{2}$ cannot have integer coefficients; therefore $(L,v)$ cannot be nice.
\end{proof}

\begin{conjecture} Let $L$ be a link and let $v\in \newtonalex$ be a fibered and positively exposed vertex. Then, there exists a braid $\bd$ and an inversion datum $\iota$ for which the assumptions of \cref{thm:main-thm} are satisfied with $\phi=\phi_v$.
\end{conjecture}

In support of the Conjecture, we have the following partial result:
\begin{proposition} Let $L$ be a prime link with at most $10$ crossings and two or three components. Then, for every fibered vertex $v\in \mathrm{Vert}\left(B_T^\vee\right)$, there exists a braid $\bd$ for which the pair $(\bd,\frac{v}{2})$ is nice.
\end{proposition}

This result was checked computationally by using \cite{lfhcompute} to compute the relevant fibered vertices of the Newton polytope, and \cite{OrlandBraidsSoftware,FKCompute2026} to check niceness.

\begin{remark} Even if a pair $(\bd,v')$ is nice, it does not mean that all pairs $(\bd,v)$ are nice for every $v\in \Hnice$. Nevertheless, we observed that in most cases it is possible to find a braid $\tilde{\bd}$ for which $(\tilde{\bd},v)$ is nice for every $v\in\Hnice$.
\end{remark}

Given a fibered link, there may be multiple vertices $v\in H_1(M,\Z)$ that make the pair $(L,v)$ nice (cf. \cref{sec:L7n1-FL}). By reducing to the single variable scenario, we still expect a detection result.

\begin{conjecture}\label{conj:single-var-FL-fibered} $L$ is a fibered link if, and only if, the following holds:
\begin{enumerate}
    \item There is a unique $v\in\Hnice$ satisfying $\inner{\mathbb{1}^*}{v}= \chi(L)$.
    \item There exists a power series
    \begin{equation*}
        F_L(x,q)\in x^{-\frac{\chi(L)}{2}}\Z[q,q^{-1}][[x]]
    \end{equation*}
    such that the evaluation $\mathbf{x}=(x,\dots,x)$ of $F_L^{(v)}(\mathbf{x},q)$ agrees with $F_L(x,q)$.
\end{enumerate}
\end{conjecture}
Note that the forward direction of the Conjecture is already covered by \cref{thm:single-var-FL}. The reverse direction is inspired by two facts: It is directly analogous to the knot case, for which the same statement was conjectured in \cite{OSSS25}, and we have not found any counterexample in our computational testing.

Together with \cref{conj:maslov-links}, the single-variable Gukov--Manolescu invariant of a fibered link is conjectured to take the form
\begin{equation}\label{eq:FL-fibered-expectation}
    F_L(x,q)= \varepsilon\, x^{-\frac{\chi(L)}{2}}q^{\frac{b_1(L)}{2}-\lambda(L)} + O(x^{-\frac{\chi(L)}{2}+1})\in \Z[q,q^{-1}][[x]]
\end{equation}
for some $\varepsilon\in\{\pm1\}$, in analogy to the expectation for fibered knots derived in \cite{OSSS25}.

\begin{remark} We notice that, for a fibered link, the leading $q$-exponent appearing in \eqref{eq:FL-fibered-expectation} agrees with the $q$-exponent of the leading term of its Akutsu-Deguchi-Ohtsuki (ADO) polynomial, as proven in \cite{NeumannVanDerVeen24}.
\end{remark}

\begin{remark}\label{rem:stratified} While in this paper we only study the Gukov--Manolescu series of nice pairs $(L,v)$, we expect there to exist an invariant associated to each pair $(L,v)$ for every link $L$ and extremal vertex of the dual Thurston polytope of $L$.

For certain non-fibered vertices $v\in B_T^\vee$, the associated $F_L^{(v)}$ series may be constructed by selecting a suitable stratification of the state space $\Omega(\iota)$ in the inverted state sum of \eqref{eq:multivariable-inverted-state-sum}, as was studied for knots (cf. \cite[Section~5]{OSSS25} and \cite{Park20}). The result should be series that live in the ring
\begin{equation*}
    \Z((q))_{\mathcal{C}_v}[[\mathbf{x}]]~.
\end{equation*}
Establishing convergence and invariance of such stratified state sums, however, lies beyond the scope of this paper.
\end{remark}

\section{Examples}\label{sec:worked-examples}

\subsection{\texorpdfstring{$L7n1$}{L7n1}:}\label{sec:L7n1-FL}

Let $L_+=L7n1\{1\}$, with braid representative
\[
 \beta_+=\sigma_1\sigma_2\sigma_1^{-1}\sigma_2\sigma_1\sigma_2^{-2}.
\]
Reversing the orientation of the first component of $L_+$ gives $L_-=L7n1\{0\}$, with braid representative
\[
 \beta_-=\sigma_1^{-1}\sigma_2^{-1}\sigma_1^{-1}
 \sigma_2^{-1}\sigma_1^{-1}\sigma_2^{-2}~.
\]
The braids with segment labels attached are represented in \cref{fig:L7-braids}.

\begin{figure}
\centering
\begin{subfigure}{0.45\textwidth}\centering
  \includegraphics[width=\linewidth]{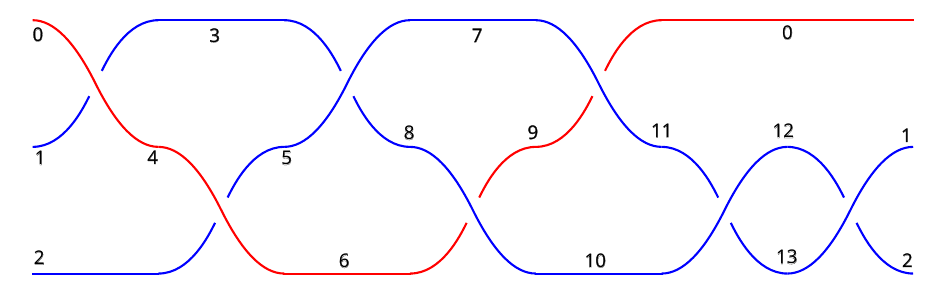}
\caption{$L_+$}
\end{subfigure}
\begin{subfigure}{0.45\textwidth}\centering
  \includegraphics[width=\linewidth]{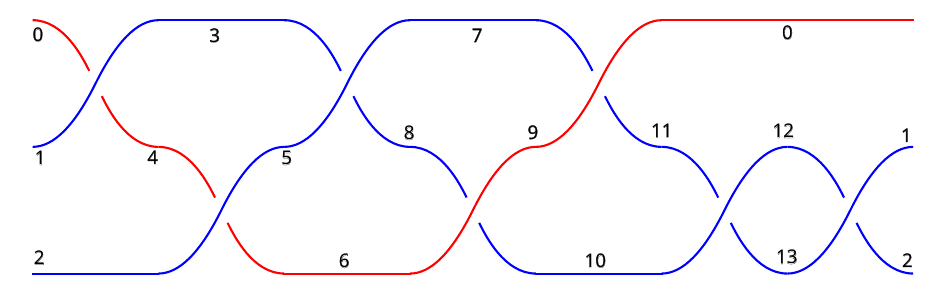}
\caption{$L_-$}
\end{subfigure}
\caption{Left-to-right oriented braid representatives and the segment labels used in the
transition matrices for $L7n1$.}
\label{fig:L7-braids}
\end{figure}

\subsubsection{Floer support, Alexander polynomial, and Thurston polytope}

The complete nonzero link Floer groups for both orientations appear in \cref{tab:L7-hfl}.
The Poincar\'e polynomials associated to the braids are
\[
P_{L_+}(t;x_0,x_1)=t^{-2}x_1^{-2}+t^{-1}x_1^{-1}+t^{-1}+t^{-1}x_0 x_1^{-2}+x_0^{-1}x_1+1+x_0 x_1^{-1}+t x_0^{-1}x_1^2+t x_1+t^2 x_1^2,
\]
and
\[
P_{L_-}(t;x_0,x_1)=t^{-1}x_0^{-1}x_1^{-2}+x_0^{-1}x_1^{-1}+x_1^{-2}+t x_1^{-1}+t+t^2+t^3 x_1+t^4 x_1^2+t^4 x_0 x_1+t^5 x_0 x_1^2~.
\]

\begin{table}
\centering
\begin{tabular}{c|c@{\qquad}c|c}
$A(L_+)$ & $\HFL(L_+)$ & $2A(L_-)$ & $\HFL(L_-)$ \\[0.3em] \hline
$(0,-2)$ & $\mathbb F_{(-2)}$ & $(0,-4)$ & $\mathbb F_{(0)}$ \\
$(0,-1)$ & $\mathbb F_{(-1)}$ & $(0,-2)$ & $\mathbb F_{(1)}$ \\
$(0,0)$ & $\mathbb F_{(-1)}\oplus\mathbb F_{(0)}$ & $(0,0)$ & $\mathbb F_{(1)}\oplus\mathbb F_{(2)}$ \\
$(1,-2)$ & $\mathbb F_{(-1)}$ & $(-2,-4)$ & $\mathbb F_{(-1)}$ \\
$(-1,1)$ & $\mathbb F_{(0)}$ & $(2,2)$ & $\mathbb F_{(4)}$ \\
$(1,-1)$ & $\mathbb F_{(0)}$ & $(-2,-2)$ & $\mathbb F_{(0)}$ \\
$(-1,2)$ & $\mathbb F_{(1)}$ & $(2,4)$ & $\mathbb F_{(5)}$ \\
$(0,1)$ & $\mathbb F_{(1)}$ & $(0,2)$ & $\mathbb F_{(3)}$ \\
$(0,2)$ & $\mathbb F_{(2)}$ & $(0,4)$ & $\mathbb F_{(4)}$ \\
\end{tabular}
\caption{Complete nonzero link Floer homology for $L7n1$ with both orientations.}
\label{tab:L7-hfl}
\end{table}
The two generators at the origin cancel in the Euler characteristic, which recovers the multivariable Alexander polynomial
\[
 \chi\left(\HFL(L)\right)(x,y)=P_{L}(-1,x,y)
 =(x^{1/2}-x^{-1/2})(y^{1/2}-y^{-1/2})\Delta_L(x,y)~,
\]
for $L\in\{L_+,L_-\}$, where
\[
 \Delta_{L_+}(x,y)=x^{1/2}y^{-3/2}+x^{-1/2}y^{3/2} \hspace{1.25em}\text{and}\hspace{1.25em}\Delta_{L_-}(x,y)=-x^{-1/2}y^{-3/2}-x^{1/2}y^{3/2}~.
\]
Indeed, the multivariable Alexander polynomials are related via $\Delta_-(x,y)=-\Delta_+(x^{-1},y)$, as described in \cref{prop:Alex-properties}. The link Floer polytope and the Thurston polytope of $L_+$ and $L_-$ are illustrated in \cref{fig:L7-hfl-thurston-polytopes}.

\begin{figure}
\centering
\begin{subfigure}{.48\textwidth}\centering
\includegraphics[width=\linewidth,height=.25\textheight,keepaspectratio]{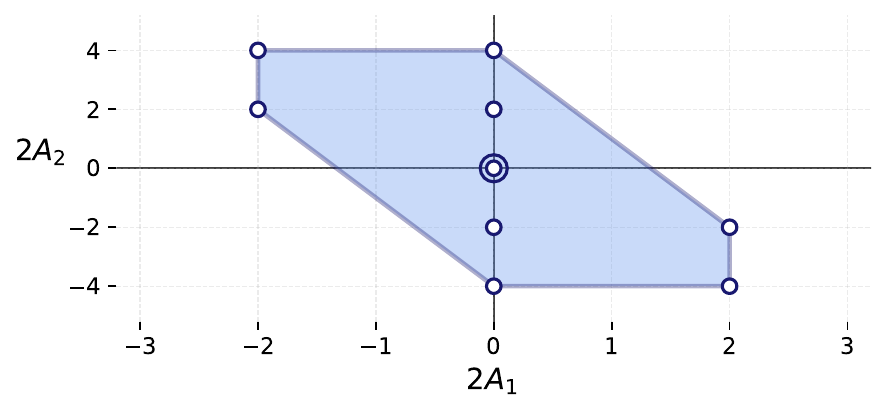}
\caption{$2\mathcal P_{\HFL}(L_+)$.}
\end{subfigure}\hfill
\begin{subfigure}{.48\textwidth}\centering
\includegraphics[width=\linewidth,height=.25\textheight,keepaspectratio]{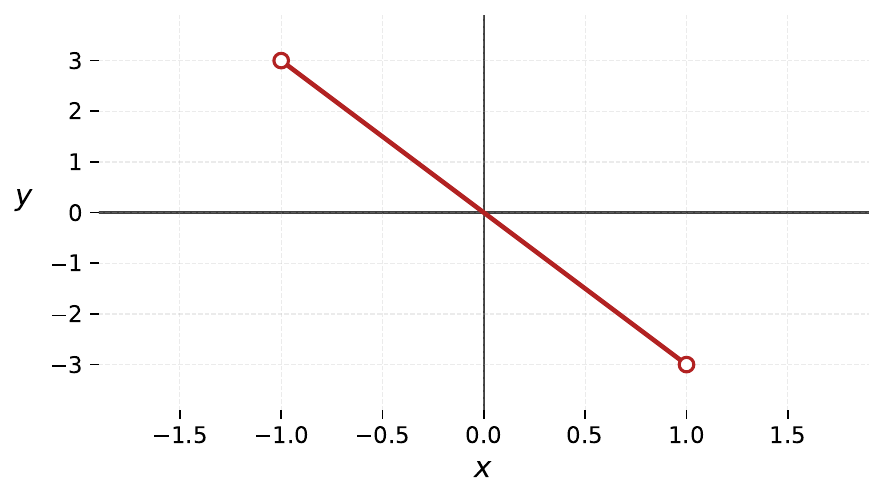}
\caption{$B_T^\vee(L_+)$.}
\end{subfigure}\\[1em]
\begin{subfigure}{.48\textwidth}\centering
\includegraphics[width=\linewidth,height=.25\textheight,keepaspectratio]{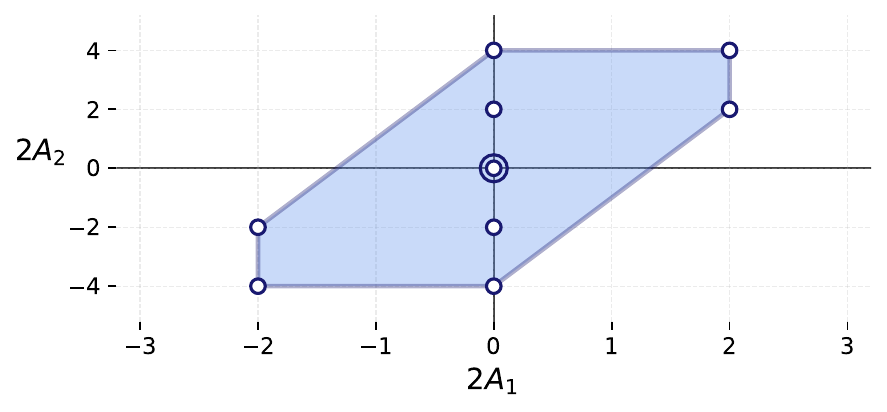}
\caption{$2\mathcal P_{\HFL}(L_-)$.}
\end{subfigure}\hfill
\begin{subfigure}{.48\textwidth}\centering
\includegraphics[width=\linewidth,height=.25\textheight,keepaspectratio]{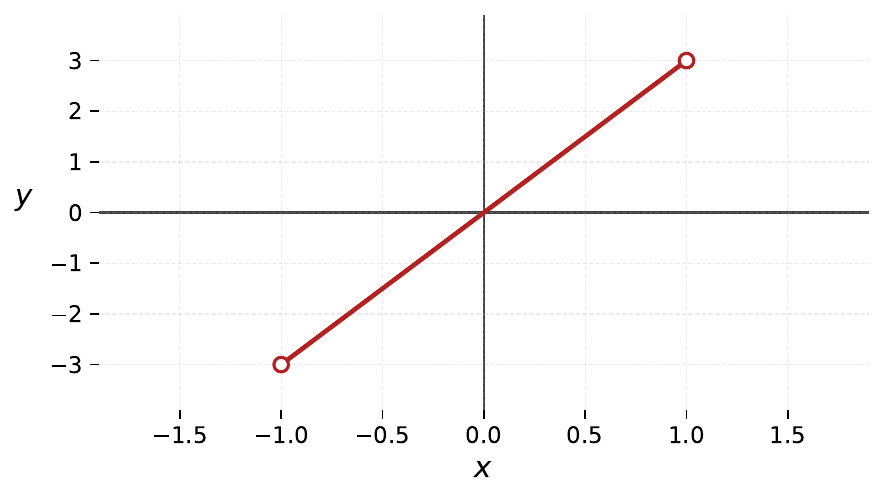}
\caption{$B_T^\vee(L_-)$.}
\end{subfigure}
\caption{The doubled $\HFL$ support polytope $2\mathcal P_{\HFL}$ (blue, on the
$2A_1,2A_2$ axes of \cref{tab:L7-hfl}) and dual Thurston ball $B_T^\vee$ (red,
on plain $x,y$ axes) for both orientations of $L7n1$.  Concentric circles
mark the rank at each point.}
\label{fig:L7-hfl-thurston-polytopes}
\end{figure}

\begin{table}
  \begin{center}
    \[
    \begin{array}{c|c|c|c|c}
    \text{orientation}&\text{HFL vertex}&
    \text{Thurston vertex}&\text{minimizing functional} & \text{inversion cycle}\\ \hline
    L_+&(0,-2)&(1,-3)&(1,1) & (2\,5\,8\,10\,13)\\
    L_+&(-1,1)&(-1,3)&(4,1) & \text{id}\\
    L_-&(-1,-2)&(-1,-3)&(1,1) & (1\,4\,6\,9\,11\,13)(2\,5\,8\,10\,12)
   \end{array}
  \]
  \end{center}
  \caption{$\HFL$ and Thurston vertices, minimizing functionals and inversion cycles associated to $L7n1$ for both relative orientations.}\label{tab:L7n1-tab-all}
\end{table}

\subsubsection{Transition matrix and unique minimizing assignments}
The transition $e\to f$ has weight $\mathcal A_{ef}$.
For a fixed functional $\phi$, let $x^{m_{ef}}$ be the monomial of this weight whose exponent is minimized by $\phi$;
that is,
\[
    m_{ef} = \operatorname*{argmin}_{m\in\operatorname{supp}\mathcal A_{ef}}\inner{\phi}{ m}~.
\]
Thus each such transition contributes $m_{ef}$ to the exponent vector $\mathbf{a}(c)$ of a cycle weight $\weight{c}$, and contributes $\inner{\phi}{m_{ef}}$ to $\inner{\phi}{\mathbf{a}(c)}$.
A zero entry of the transition matrix does not define an allowed transition, while a fixed point of an assignment contributes the diagonal entry $1$ of $I-\mathcal A$, hence exponent zero.

We find that the cycles $c_{\iota}$,  listed by their associated permutations in the last column of \cref{tab:L7n1-tab-all}, are associated to the two distinct fibered vertices of the dual Thurston polytope. Recall that the inversion datum $\iota$ can be recovered from the permutation $\sigma_\iota$ by assigning $-$ to the segments whose labels appear in the permutation and $+$ to the remaining segments.
The only nontrivial inversion datum for $L_+$, associated to the vertex $p=(1,-3)$ of the Thurston polytope, is shown in \cref{fig:L7-plus-cycle};
the other minimum is the identity assignment and therefore is associated to the empty cycle.
\begin{figure}
\centering
\includegraphics[width=.58\linewidth]{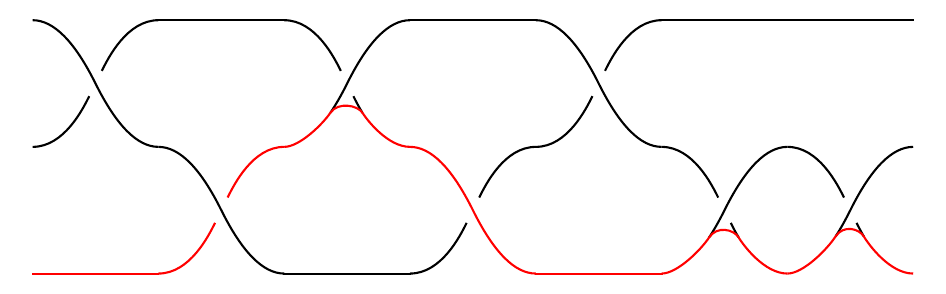}
\caption{The nontrivial inversion cycle for $L_+$ induced by the permutation $(2\ 5\ 8\ 10\ 13)$, where the labels are as in \cref{fig:L7-braids}.}
\label{fig:L7-plus-cycle}
\end{figure}

\subsubsection{The $F_L$ expansions}

The first few terms associated to the Gukov--Manolescu series of $L_+$ at $(1,-3)$ and $(-1,3)$ are
\begin{align*}
 F_+^{(1,-3)}(x,y,q)={}&-q^{-1/2}x^{-1/2}y^{3/2}
 +q^{-7/2}x^{-3/2}y^{9/2}
 -q^{-19/2}x^{-5/2}y^{15/2}+\cdots,\\
 F_+^{(-1,3)}(x,y,q)={}&- q^{1/2}x^{1/2}y^{-3/2} + q^{-7/2}x^{3/2}y^{-9/2} - q^{-19/2}x^{5/2}y^{-15/2}+\cdots ~.
\end{align*}
Note that since $(1,-3)$ and $(-1,3)$ are antipodal, their expansions are related by coordinate inversion.
At $q=1$ these are the two geometric-series expansions
\begin{align*}
  \lim_{t\to 0}\Delta_+^{-1}(tx,t^{-3}y)\big|_{t=1}
 &=\sum_{n\ge0}(-1)^n x^{-(2n+1)/2}y^{(6n+3)/2},\\
 \lim_{t\to 0}\Delta_+^{-1}(t^{-1}x,t^3y)\big|_{t=1}
 &=\sum_{n\ge0}(-1)^n x^{(2n+1)/2}y^{-(6n+3)/2}.
\end{align*}
For $L_-$, the same intrinsic fibered vertex whose coordinates were $(1,-3)$
for $L_+$ now has coordinates $(-1,-3)$.
Its first terms are
\[
 F_-^{(-1,-3)}(x,y,q)=-q^{-5/2}x^{1/2}y^{3/2}
 +q^{-11/2}x^{3/2}y^{9/2}
 -q^{-23/2}x^{5/2}y^{15/2}+\cdots.
\]
As expected, this expansion is related to the series $F^{(1,-3)}_+$.
In particular, we have
\begin{equation}\label{eq:L7-F-orientation}
  F_-^{(-1,-3)}(x,y,q)=q^{-2}F_+^{(1,-3)}(x^{-1},y,q)=q^{-2}F_+^{(-1,3)}(x,y^{-1},q).
\end{equation}
The factor $q^{-2}$ is the linking number correction (see \cref{rmk:non-positively-exposed-FL}).

\subsection{\texorpdfstring{$L9n18$}{L9n18}:}\label{ex:L9n18}
In most examples, the vertices $v\in\newtonalex$ that accompany the Gukov--Manolescu series are minimized by some lexicographic ordering of the variables (see \cref{rmk:lexicographic}). Here, we study a link for which this does not hold.

Let $L_+=L9n18\{1\}$ and $L_-=L9n18\{0\}$.
We use the braid representatives
\[
 \beta_+=\sigma_2\sigma_3^{-1}\sigma_2\sigma_3\sigma_1\sigma_2
 \sigma_3^{-1}\sigma_2\sigma_3\sigma_1~,
\]
and
\[
 \beta_-=\sigma_1^{-1}\sigma_2^{-3}\sigma_1^{-1}
 \sigma_2^{-3}\sigma_1^{-1}.
\]

\begin{figure}
\centering
\begin{subfigure}{0.45\textwidth}\centering
  \includegraphics[width=\linewidth]{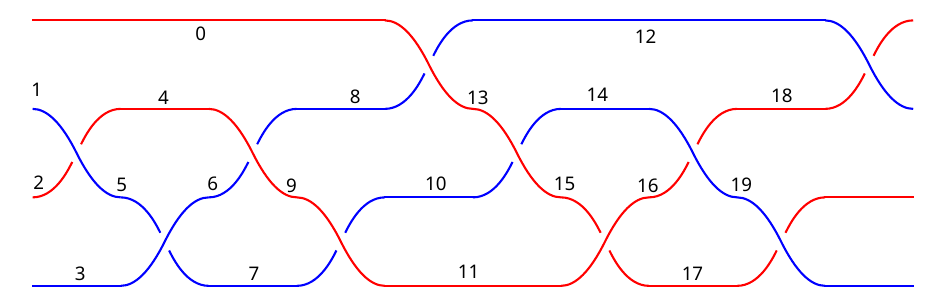}
\caption{$L_+$}
\end{subfigure}
\begin{subfigure}{0.45\textwidth}\centering
  \includegraphics[width=\linewidth]{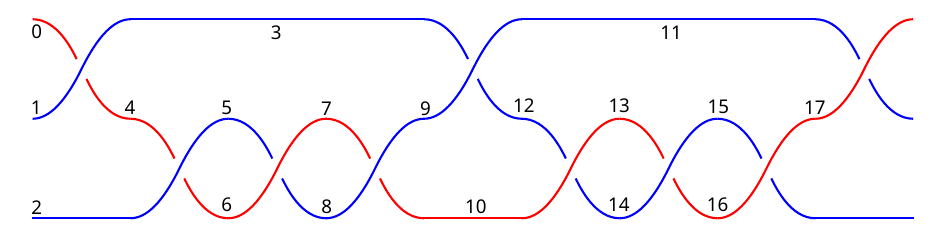}
\caption{$L_-$}
\end{subfigure}
\caption{Left-to-right oriented braid representatives and the segment labels used in the
transition matrices for $L9n18$.}
\label{fig:L9-braids}
\end{figure}

\subsubsection{Floer support, Alexander polynomial, and Thurston polytope}

The complete nonzero $\HFL$ groups for $L_+$ and $L_-$ are listed in \cref{tab:L9-plus-hfl}.
\begin{table}
\centering
\[
\begin{array}{c|c@{\quad}c|cc@{\qquad}c|c@{\quad}c|cc}
    A(L_+)&\HFL(L_+)&A(L_+)&\HFL(L_+)&&A(L_-)&\HFL(L_-)&A(L_-)&\HFL(L_-)\\ \cline{1-4}\cline{6-9}
  (-1,-1)&\mathbb F_{(-4)}&(-1,0)&\mathbb F_{(-3)}
  &&(-2,-2)&\mathbb F_{(-1)}&(-2,-1)&\mathbb F_{(0)}\\
  (0,-1)&\mathbb F_{(-3)}&(-2,1)&\mathbb F_{(-2)}
  &&(-1,-2)&\mathbb F_{(0)}&(-1,-1)&\mathbb F_{(1)}\\
  (0,0)&\mathbb F_{(-2)}^{\oplus2}&(1,-2)&\mathbb
  F_{(-2)}&&(-1,0)&\mathbb F_{(1)}&(-1,1)&\mathbb
  F_{(2)}\\
  (-2,2)&\mathbb F_{(-1)}&(-1,1)&\mathbb F_{(-1)}
  &&(0,-1)&\mathbb F_{(1)}&(0,0)&\mathbb F_{(2)}
  ^{\oplus2}\\
  (0,1)&\mathbb F_{(-1)}&(1,-1)&\mathbb F_{(-1)}
  &&(0,1)&\mathbb F_{(3)}&(1,-1)&\mathbb F_{(2)}\\
  (1,0)&\mathbb F_{(-1)}&(2,-2)&\mathbb F_{(-1)}
  &&(1,0)&\mathbb F_{(3)}&(1,1)&\mathbb F_{(5)}\\
  (-1,2)&\mathbb F_{(0)}&(1,1)&\mathbb F_{(0)}
  &&(1,2)&\mathbb F_{(6)}&(2,1)&\mathbb F_{(6)}\\
  (2,-1)&\mathbb F_{(0)}&&&&(2,2)&\mathbb F_{(7)}&&
  \end{array}
\]
\caption{Complete nonzero link Floer homology for $L_+=L9n18\{1\}$; the
Alexander coordinates are doubled.}
\label{tab:L9-plus-hfl}
\end{table}
Their Poincar\'e polynomials are
\[
\begin{aligned}
P_{L_+}(t,x_0,x_1)={}&t^{-4}x_0^{-1}x_1^{-1}
 +t^{-3}x_0^{-1}+t^{-3}x_1^{-1}+t^{-2}x_0^{-2}x_1+2t^{-2}\\
&+t^{-2}x_0x_1^{-2}+t^{-1}x_0^{-2}x_1^2
 +t^{-1}x_0^{-1}x_1+t^{-1}x_1+t^{-1}x_0x_1^{-1}\\
&+t^{-1}x_0+t^{-1}x_0^2x_1^{-2}
 +x_0^{-1}x_1^2+x_0x_1+x_0^2x_1^{-1}~.
\end{aligned}
\]
and
\[
\begin{aligned}
P_{L_-}(t,x_0,x_1)={}&t^{-1}x_0^{-2}x_1^{-2}
 +x_0^{-2}x_1^{-1}+x_0^{-1}x_1^{-2}+tx_0^{-1}x_1^{-1}
 +tx_0^{-1}+tx_1^{-1}\\
&+t^2x_0^{-1}x_1+2t^2+t^2x_0x_1^{-1}
 +t^3x_1+t^3x_0+t^5x_0x_1\\
&+t^6x_0x_1^2+t^6x_0^2x_1+t^7x_0^2x_1^2 ~,
\end{aligned}
\]

The multivariable Alexander polynomials are given by
\begin{equation}\label{eq:L9-Alexander}
  \Delta_+(x,y)=\frac{(x+y^2)(x^2+y)}{x^{3/2}y^{3/2}} \quad\text{and}\quad \Delta_-(x,y)=-\Delta_+(x,y^{-1})~.
\end{equation}
The link Floer and Thurston polytopes are displayed in \cref{fig:L9-hfl-thurston-polytopes}.
\begin{figure}
\centering
\begin{subfigure}{.48\textwidth}\centering
\includegraphics[width=\linewidth]{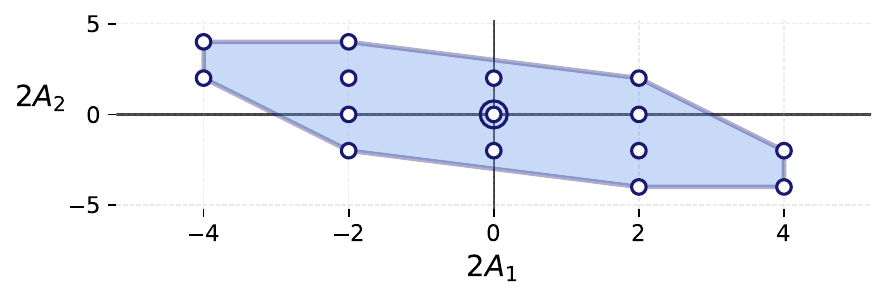}
\caption{$2\mathcal P_{\HFL}(L_+)$.}
\end{subfigure}\hfill
\begin{subfigure}{.48\textwidth}\centering
\includegraphics[width=\linewidth]{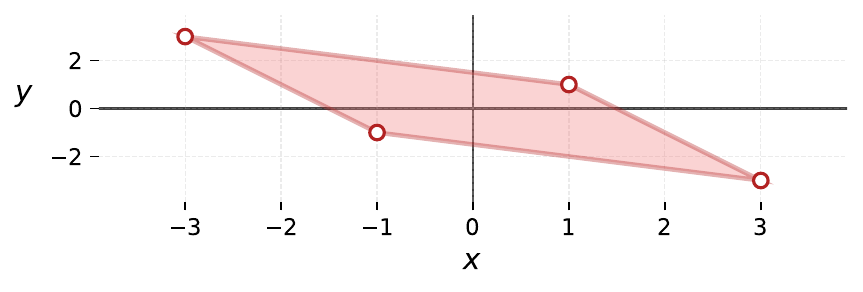}
\caption{$B_T^\vee(L_+)$.}
\end{subfigure}
\caption{The doubled $\HFL$ support polytope $2\mathcal P_{\HFL}$ (blue, on
$2A_1,2A_2$ axes) and dual Thurston ball $B_T^\vee$ (red, on plain $x,y$ axes)
for both orientations of $L9n18$.
The number of concentric circles marks the rank at each
point. The same diagrams for the other orientation can be obtained by a symmetric transformation along the vertical axis, similar to \cref{fig:L7-hfl-thurston-polytopes}.}
\label{fig:L9-hfl-thurston-polytopes}
\end{figure}

\begin{table}
  \begin{center}
    \[
    \begin{array}{c|c|c|c|c}
    \text{orientation}&\text{HFL vertex}&
    \text{Thurston vertex}&\text{minimizing functional} & \text{inversion cycle}\\ \hline
    L_+&(1,-2)&(3,-3)&(1,3)&(3\,7\,10\,14\,19)\\
    L_+&(-1,-1)&(-1,-1)&(1,1)&\text{id}\\
    L_+&(-2,1)&(-3,3)&(3,1)&(2\,4\,9\,11\,17)\\
    L_-&(-2,-2)&(-3,-3)&(1,1)&\substack{(1\,4\,6\,7\,10\,13\,16\,17)\\(2\,5\,8\,9\,12\,14\,15)}
   \end{array}
  \]
  \end{center}
  \caption{$\HFL$ and Thurston vertices, minimizing functionals and inversion cycles associated to $L9n18$ for both relative orientations.}\label{tab:L9n18-tab-all}
\end{table}

\begin{figure}
\centering
\begin{subfigure}{.48\textwidth}\centering
\includegraphics[width=\linewidth]{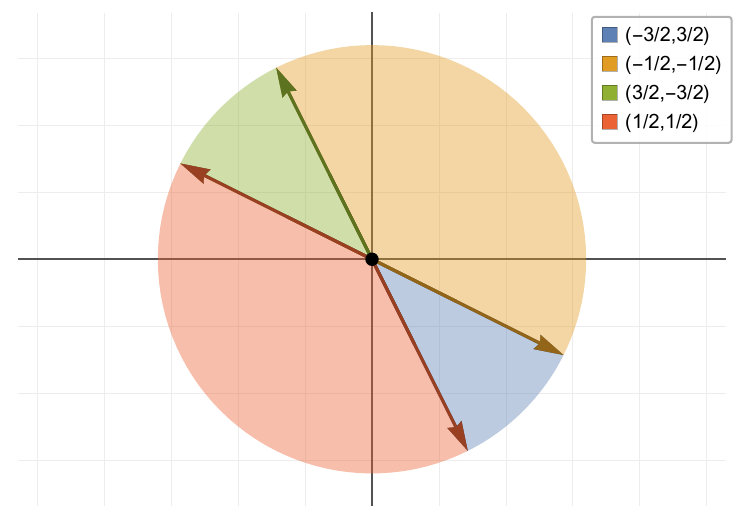}
\caption{The cones $\mathcal{C}_v$.}
\end{subfigure}\hfill
\begin{subfigure}{.48\textwidth}\centering
\includegraphics[width=\linewidth]{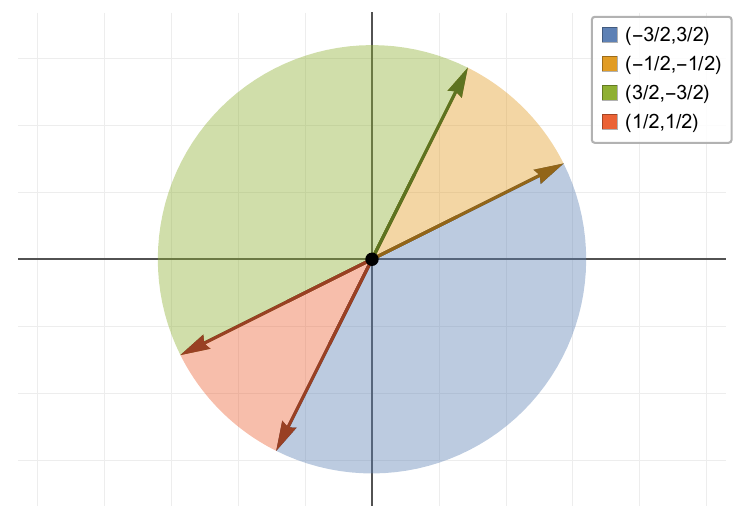}
\caption{The dual cones $\mathcal{C}_v^\vee$.}
\end{subfigure}
\caption{The cones $\mathcal{C}_v$ and dual cones $\mathcal{C}_v^\vee$ for the fibered vertices of $L9n18\{1\}$, in the multivariable Alexander convention. The cones $\mathcal{C}_v$ indicate the support of the associated $F_L^{(v)}$, while the dual cones $\mathcal{C}_v^\vee$ indicate the family of functionals that expose each vertex $v$. The vertices associated to blue, yellow, and green are positively exposed, since their associated dual cone intersects the positive quadrant $(\R_{\geq 0})^2$.}
\label{fig:L9-cones}
\end{figure}

\subsubsection{Transition matrix and unique minimizing assignments}
The minimizing functionals for each Thurston vertex, along with the inversion cycles at fibered Thurston vertices, are recorded in \cref{tab:L9n18-tab-all} and the non-trivial cycles are presented in \cref{fig:minimizing-L9_1}.
There are three positively exposed vertices, as illustrated in \cref{fig:L9-cones}. Only two of them, namely $(-3,3)$ and $(3,-3)$, minimize some lexicographic ordering.

\begin{figure}
\centering
\begin{subfigure}{0.45\textwidth}\centering
  \includegraphics[width=\linewidth]{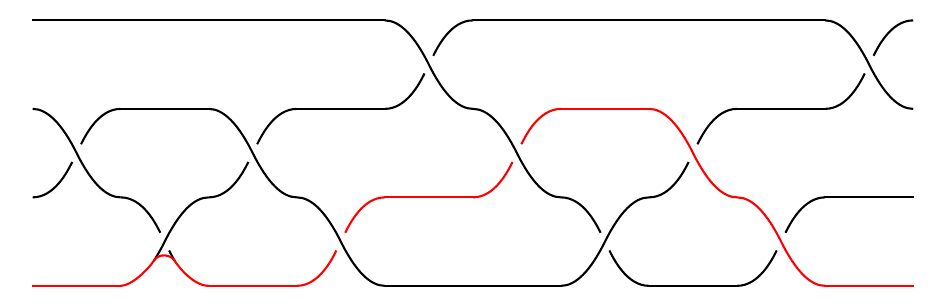}
\end{subfigure}
\begin{subfigure}{0.45\textwidth}\centering
  \includegraphics[width=\linewidth]{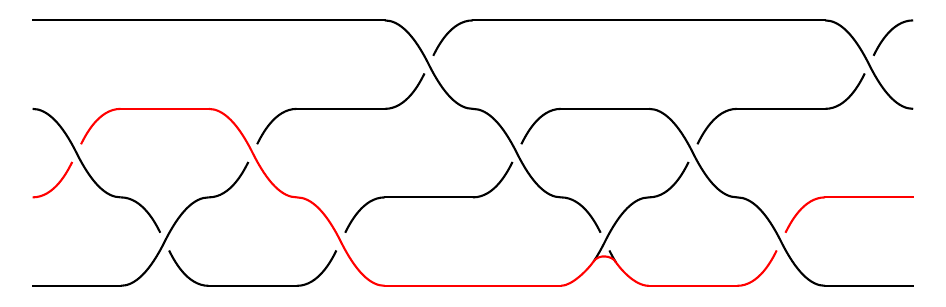}
\end{subfigure}
\caption{The nontrivial minimum-cost cycles $c_\iota = (3\ 7\ 10\ 14\ 19)$ and $c_\iota = (2\ 4\ 9\ 11\ 17)$ for $L_+$.}
\label{fig:minimizing-L9_1}
\end{figure}

\subsubsection{The $F_L$ expansions}

The leading terms of the three state sums of $L_+$ are
\begin{align*}
F_+^{(3,-3)}(x,y,q)={}&-q^{1/2}x^{-3/2}y^{3/2}
 +q^{-3/2}(x^{-5/2}y^{7/2}+x^{-7/2}y^{5/2})+\cdots,\\
F_+^{(-1,-1)}(x,y,q)={}&-q^{3/2}x^{1/2}y^{1/2}
 +q^{3/2}(x^{-1/2}y^{5/2}+x^{5/2}y^{-1/2})+\cdots,\\
F_+^{(-3,3)}(x,y,q)={}&-q^{1/2}x^{3/2}y^{-3/2}
 +q^{-3/2}(x^{5/2}y^{-7/2}+x^{7/2}y^{-5/2})+\cdots.
\end{align*}
Their $q=1$ specializations are the three cone expansions of $-\Delta_+^{-1}$ indexed by the indicated fibered vertices of $B_T^\vee$
\begin{align*}
  F^{(3,-3)}_+(x,y;1)=&{}-\lim_{t\to 0}\Delta_+^{-1}(tx,t^{3}y)\big|_{t=1}~,\\
  F^{(-1,-1)}_+(x,y;1)=&{}-\lim_{t\to 0}\Delta_+^{-1}(tx,ty)\big|_{t=1}~,\\
  F^{(-3,3)}_+(x,y;1)=&{}-\lim_{t\to 0}\Delta_+^{-1}(t^{3}x,ty)\big|_{t=1}~.
\end{align*}
Reversing the second component relates the series associated to the vertex $(-3,3)$ of $L_+$ to the one with coordinates $(-3,-3)$ of $L_-$.
It is the only one of these three vertices to be positively exposed in $L_-$.
The expansion begins with
\[
 F_-^{(-3,-3)}(x,y,q)=-q^{-7/2}x^{3/2}y^{3/2}
 +q^{-11/2}(x^{5/2}y^{7/2}+x^{7/2}y^{5/2})+\cdots~.
\]
It is related to the expansions of $F_+$ by
\begin{equation}\label{eq:L9-F-orientation}
  F_-^{(-3,-3)}(x,y,q)=q^{-4}F_+^{(-3,3)}(x,y^{-1},q)=q^{-4}F_+^{(-3,3)}(x^{-1},y,q)~.
\end{equation}

\subsection{\texorpdfstring{$T(2,4)$}{T(2,4)} and \texorpdfstring{$T(2,4)^*$}{T(2,4)}: a fibered and non-fibered orientation}\label{sec:example-T24}
We now revisit \cref{ex:T24-star} and exhibit the two orientations of the torus link $T(2,4)$ and its partial orientation reversal, as the simplest possible pair in which one orientation is fibered and the other is not. We use it to showcase \cref{conj:single-var-FL-fibered}.

Let $L_+=T(2,4)$ and $L_-=T(2,4)^*$, with braid representatives
\begin{equation}
    \beta_+=\sigma_1^{4} \qquad \beta_-=\sigma_2^{-2}\sigma_1^{-1}\sigma_2\sigma_1^{-1}~.
\end{equation}

\subsubsection{Floer support, Alexander polynomial and Thurston polytope}

The link Floer homology for both orientations is displayed in \cref{tab:L4-hfl}.
\begin{table}
    \centering
    \begin{equation}
    \begin{array}{c|c@{\qquad}c|c}
      A(L_+) & \HFL(L_+) & A(L_-) & \HFL(L_-) \\[0.3em]
      \hline
      (-1,-1) & \mathbb F_{(-4)} & (-1,0) & \mathbb
      F_{(-1)} \\
      (-1,0) & \mathbb F_{(-3)} & (0,-1) & \mathbb
      F_{(-1)} \\
      (0,-1) & \mathbb F_{(-3)} & (-1,1) & \mathbb
      F_{(0)} \\
      (0,0) & \mathbb F_{(-2)}^{\oplus 2} & (0,0) &
      \mathbb F_{(0)}^{\oplus 2} \\
      (0,1) & \mathbb F_{(-1)} & (1,-1) & \mathbb F_{(0)}
      \\
      (1,0) & \mathbb F_{(-1)} & (0,1) & \mathbb F_{(1)}
      \\
      (1,1) & \mathbb F_{(0)} & (1,0) & \mathbb F_{(1)}
  \end{array}~.
    \end{equation}
    \caption{$\HFL$ groups for $T(2,4)$ and $T(2,4)^*$.}
    \label{tab:L4-hfl}
\end{table}
The multivariable Alexander polynomial $\Delta_-$ of $L_-$ can be found in \eqref{eq:T24*-alex}, while $\Delta_+$ can be obtained directly from \cref{prop:Alex-properties}.
The vertices of the Newton polytope are $\{(-1,-1), (1,1)\}$ for $\Delta_+$ and $\{(-1,1),(1,-1)\}$ for $\Delta_-$.

Since the unoriented link is alternating, we can use \cref{cor:Alex-monic-fibered} to detect link fiberedness.
For $L_+$, the functional induced by the canonical class $\mathbb 1=(1,1)$ strictly increases along the two vertices, therefore $L_+$ is fibered, and the unique maximizer of $\inner{\mathbb{1}^*}{\cdot}$ is $v_+=(-\tfrac12,-\tfrac12)\in\supp\Delta_+$.
For $L_-$, on the other hand, $\inner{\mathbb 1}{(-1,1)}=\inner{\mathbb 1}{(1,-1)}=0$, so $\Delta_-$ does not have a unique monomial of maximal total degree, and therefore it is not a fibered link.

\subsubsection{The $F_L$ expansions and the single-variable series}

We compute $F_L^{(v)}$ at the vertex $v_+=(-\tfrac12,-\tfrac12)$ for $L_+$ and at $v_-=(\tfrac12,-\tfrac12)$ for $L_-$, which gives
\begin{multline}\label{eq:L4a1-plus-FL}
 F_+^{(v_+)}(x,y,q)=\sum_{n\geq0}(-1)^{n+1}q^{\binom{n+1}{2}+\frac32}\,x^{n+\frac12}y^{n+\frac12}
 \\= -q^{3/2}x^{1/2}y^{1/2}+q^{5/2}x^{3/2}y^{3/2}-q^{9/2}x^{5/2}y^{5/2}+q^{15/2}x^{7/2}y^{7/2}-\cdots~,
 \end{multline}
and
\begin{multline}\label{eq:L4a1-minus-FL}
 F_-^{(v_-)}(x,y,q)=\sum_{k\geq0}(-1)^{k+1}q^{\binom{k+1}{2}-\frac12}\,x^{-k-\frac12}y^{k+\frac12} \\
 = -q^{-1/2}x^{-1/2}y^{1/2}+q^{1/2}x^{-3/2}y^{3/2}-q^{5/2}x^{-5/2}y^{5/2}+q^{11/2}x^{-7/2}y^{7/2}-\cdots~.
\end{multline}
The linking number of the two components of $L_+$ is 2, and the two expansions are related by
\begin{equation}
    F_{-}^{(v_-)}(x,y,q) = q^{-\lambda_r}F_{+}^{(v_+)}(x^{-1},y,q) =q^{-2}F_{+}^{(v_+)}(x^{-1},y,q)~.
\end{equation}
At $q=1$ these series are, up to sign, the geometric-series expansions of $\Delta_+^{-1}$ around $v_+$ and of $\Delta_-^{-1}$ around $v_-$, respectively.

Now set $\mathbf x=(x,x)$, as in \cref{thm:single-var-FL}. We obtain

\begin{equation}\label{eq:L4a1-plus-single}
 F_+(x,q):=F_+^{(v_+)}(x,x,q)=\sum_{n\geq0}(-1)^{n+1}q^{\binom{n+1}{2}+\frac32}\,x^{2n+1}\in\Z[q,q^{-1}][[x]]~,
\end{equation}
where each coefficient is a finite Laurent $q$-polynomial, consistent with \cref{thm:single-var-FL}.

By contrast, the support of $F_-^{(v_-)}(x,y,q)$ is generated by $\{(-k-\tfrac12,k+\tfrac12)\mid k\in\N\}$, which collapses onto the same monomial $1$ once $y=x$.
The hypothesis of \cref{thm:single-var-FL} fails precisely because $v_-$ is not a unique $\mathbb 1$-minimizer. Concretely,
\begin{equation}\label{eq:L4a1-minus-single}
 F_-(x,q):=F_-^{(v_-)}(x,x,q)=\sum_{k\geq0}(-1)^{k+1}q^{\binom{k+1}{2}-\frac12}
 = -q^{-1/2}+q^{1/2}-q^{5/2}+q^{11/2}-q^{19/2}+q^{29/2}-\cdots~,
\end{equation}

In this case, the $x^0$-coefficient of $F_-(x,q)$ is a genuine one-sided infinite series in $q$.

\section{Partial surgery on 2-component links}\label{sec:surgery}

As a first direct application of our theory, we consider partial surgeries on links.
Given an oriented link $L=\bigcup_{i=1}^{\ell} L_i\subset S^3$, let $M=S^3\setminus \mathcal N(L)$ denote its exterior, with boundary $\partial M=\bigcup_{i=1}^{\ell}T_i$. Dehn filling the torus $T_i$ along a slope $\alpha$ corresponds to performing surgery on the component $L_i$. In particular, suppose that $L_i$ is unknotted and let $(\mu_i,\lambda_i)$ be the meridian--longitude basis of $T_i$. Filling along $\alpha=\mu_i+r\lambda_i$ corresponds to $1/r$-surgery on $L_i$ and produces the exterior of an $(\ell-1)$-component link $L'\subset S^3$. The Gukov--Manolescu series of $L$ and $L'$ are conjecturally related by a partial $\frac{1}{r}$-Laplace transform in the variable associated to the surgered component \cite{Park20}.

\begin{definition}\label{def:Laplace-transform} Let $L=L_1\cup\cdots\cup L_i\cup\cdots\cup L_\ell$ be an oriented link. The partial $\frac{1}{r}$-Laplace transform on the $i$-th component of $L$ is defined by
\begin{equation}\label{eq:rational-Laplace-def}
    \mathcal{L}_{\frac{1}{r};x_i} \left(a(q)\mathbf{x}^\mathbf{k}\right)\mapsto a(q)\,q^{-r k_i^2}\prod_{\substack{1\leq j \leq \ell \\ j\neq i}} x_j^{k_j-r\,\mathrm{lk}(L_i,L_j)\,k_i}\,,
\end{equation}
where $a(q)\in\Z((q))$. We extend this map to formal Laurent series in $q$ and $\mathbf{x}$, whenever the result is well-defined as a formal Laurent series in $q$ and $\hat{\mathbf{x}}_i=(x_1,\dots,x_{i-1},x_{i+1},\dots,x_\ell)$.
\end{definition}

By \cref{thm:main-thm}, if the Gukov--Manolescu invariant $F_L^{(v)}$ exists, then it is supported on the cone $\mathcal{C}_v$ centered around some integral vertex $v$, where
\begin{equation}
    \mathcal{C}_v\defeq \mathrm{Cone}\left\{w-v\mid w\in\supp\Delta_L, \, w\neq v\right\}\subseteq \R^\ell~.
\end{equation}

\begin{conjecture}[{\cite[Conjecture 7]{Park20} (modified)}]\label{conj:Laplace-2comp} Let $L=L_1\cup L_2$ be an oriented link and assume that $L_1$ is an unknot. Then,
\begin{equation}\label{eq:FL-relation-surgery}
    F_{L_2}(x_2,q)\doteq \mathcal{L}_{\frac{1}{r};x_1}\left(\left(x_1^\frac12-x_1^{-\frac12}\right)F_L^{(v)}((x_1,x_2),q)\right)(x_2,q)
\end{equation}
for any $v\in\Hnice$ for which the right-hand side gives an element of $\Z((q))[[x_2]]$.
\end{conjecture}

In this section, we study for which $v\in\Hnice$ we expect the right-hand side of \eqref{eq:FL-relation-surgery} to give a well-defined power series in $q$ and $x$ in terms of the dual cones $\mathcal{C}_v^\vee$.

\begin{remark} We expect that \cref{conj:Laplace-2comp} can be extended to cover the Gukov--Manolescu series of a link $L$ with more than two components. This extension would require studying the relation between the support cones $\mathcal{C}_v$ and $\mathcal{C}_{v'}$, where $v\in\Hnice$ and $v'\in\mathbb{G}(L')$, under the Laplace transform of \eqref{eq:rational-Laplace-def}.
\end{remark}

\subsection{Laplace transform and partial orientation reversal}

By the following Lemma, reversing the orientation of the surgered component does not affect the partial Laplace transform, up to inversion of the $x_i$-variables.

\begin{lemma} Let $L$ be an oriented link, $L_i$ an unknot, and $v$ an extremal vertex exposed by $\phi\in(\Z\setminus\{0\})^\ell$. Let $\nu\in\{\pm1\}^\ell$ with $\nu_i=\sgn{\phi_i}$, and let $v^\nu$ and $L^{(\nu)}$ be as in \cref{prop:alex-positive-convention}.

Assume that $F_L^{(2v)}\in\Z((q))_{\mathcal{C}_v}[[\mathbf{x}]]$ exists. Then, for any $1\leq i\leq \ell$,
\begin{equation}\label{eq:Laplace-orientation-reversal}
    \mathcal{L}_{\frac{1}{r};x_i}\left((x_i^{\frac{1}{2r}}-x_i^{-\frac{1}{2r}})F_L^{(2v)}\right)(\widehat{\mathbf{x}}_i)\doteq \mathcal{L}_{\frac{1}{r};x_i}\left((x_i^{\frac{1}{2r}}-x_i^{-\frac{1}{2r}})F_{L^{(\nu)}}^{(2v^{\nu})}\right)(\widehat{\tau_\nu(\mathbf{x})}_i)~,
\end{equation}
where $\hat{\mathbf{x}}_i=(x_1,\dots,x_{i-1},x_{i+1},\dots,x_\ell)$ and $\tau_\nu(\mathbf{x})=\mathbf{x}^\nu\in\C^\ell$ is defined as
\begin{equation*}
    \tau_\nu(\mathbf{x})_j=\begin{cases}
        x_j &\text{if }\nu_j=1\\
        x_j^{-1} &\text{if }\nu_j=-1
    \end{cases}~.
\end{equation*}
\end{lemma}

\begin{proof} By \cref{rmk:non-positively-exposed-FL}, both series are related by
\begin{equation}
    F_L^{(2v)}(\mathbf{x},q) =\varepsilon q^{\alpha_\nu} F_{L^{(\nu)}}^{(2v^\nu)}(\tau_\nu(\mathbf{x}),q)
\end{equation}
where $\varepsilon\in\{\pm1\}$ and
\begin{equation*}
    \alpha_\nu=\sum_{\substack{j=1\\\nu_j=-1}}^\ell \mathrm{lk}(L_i,L\setminus L_i)~.
\end{equation*}
The left-hand side of \eqref{eq:Laplace-orientation-reversal} is
\begin{equation}
    \sum_{\mathbf{k}\in\mathcal{C}_v} f_\mathbf{k}(q) \left(q^{-r m_{i,+}^2} \prod_{\substack{j=1\\j\neq i}}^\ell x_j^{k_j-r\,\mathrm{lk}(L_i,L_j)m_{i,+}}- q^{-r m_{i,-}^2} \prod_{\substack{j=1\\j\neq i}}^\ell x_j^{k_j-r\,\mathrm{lk}(L_i,L_j)m_{i,-}}\right)
\end{equation}
where $m_{i,\pm}\defeq n_i\pm \frac{1}{2r}$. On the other hand, the right-hand side of \eqref{eq:Laplace-orientation-reversal} is, up to overall $q$-power and sign,
\begin{equation}
\begin{split}
    \sum_{\mathbf{k}\in\mathcal{C}_v} f_\mathbf{k}(q) \left( q^{-r \,m_{i,\nu_i}^2} \prod_{\substack{j=1\\j\neq i}}^\ell x_j^{\nu_jk_j-\nu_jr\,\mathrm{lk}(L_i,L_j)m_{i,\nu_i}} -q^{-r \,m_{i,-\nu_i}^2} \prod_{\substack{j=1\\j\neq i}}^\ell x_j^{\nu_j k_j-\nu_j r\,\mathrm{lk}(L_i,L_j)\,m_{i,-\nu_i}} \right)
\end{split}
\end{equation}
where $m_{i,\pm\nu_i}=n_i\pm\nu_i\frac{1}{2r}$, and we have used that $\mathrm{lk}(L_i^{(\nu_i)},L_j^{(\nu_j)})=\nu_i\nu_j\mathrm{lk}(L_i,L_j)$. Both agree up to an overall sign if $\nu_i=-1$, and on the nose if $\nu_i=1$.
\end{proof}

For two-component links, define the symmetric $F_L$-series as
\begin{equation}
    F_L^{\text{sym},(v)} \defeq \frac{1}{2}\left(F_L^{(v)}(\mathbf{x})+F_L^{(-v)}(\mathbf{x})\right)\,.
\end{equation}

\begin{lemma}\label{prop:2-comp-surgery-convention-positively-exposed} Let $L$ be a two-component, oriented link, and $v$ a vertex exposed by $\phi\in(\Z\setminus \{0\})^2$. Let $\nu\in\{\pm 1\}^2$ with $\nu_i=\sgn{\phi_i}$, and let $v^\nu$ and $L^{(\nu)}$ be as in \cref{prop:alex-positive-convention}.

Assume that $F_L^{(v)}\in\Z((q))_{\mathcal{C}_v}[[\mathbf{x}]]$ exists. Then, for $1\leq i\leq 2$,
\begin{equation}
    \mathcal{L}_{\frac{1}{r};x_i}\left(\left(x_i^{\frac{1}{2r}}-x_i^{-\frac{1}{2r}}\right)F_L^{\text{sym,}(2v)}\right)(x)\doteq \mathcal{L}_{\frac{1}{r};x_i}\left(\left(x_i^{\frac{1}{2r}}-x_i^{-\frac{1}{2r}}\right)F_{L^{(\nu)}}^{\text{sym,}(2v^\nu)}\right)(x)~.
\end{equation}
\end{lemma}

It is thus equivalent to study the effect of the partial Laplace transform on the $F_L^{(v)}$ series associated to the set
\begin{equation*}
    \left\{(L,2v)\mid v \text{ is an exposed vertex of }\newtonalex\right\}
\end{equation*}
and
\begin{equation*}
    \left\{(L^{(\mu)},2v)\mid \mu\in\{\pm 1\}^2\text{ and }v \text{ is a positively exposed vertex of }\newtonalex\right\}~.
\end{equation*}
In the former, we fix an orientation of the link $L$ and allow $v$ to be a generic exposed vertex, while in the latter, we restrict the values of $v$ at the cost of considering all partial orientation reversals of $L$.

\begin{proof} The result follows from the previous Lemma and the relation
\begin{equation}
    \mathcal{L}_{\frac{1}{r};x_i}\left(\left(x_i^{\frac{1}{2r}}-x_i^{-\frac{1}{2r}}\right)F_L^{(-2v)}\right)(x) = \mathcal{L}_{\frac{1}{r};x_i}\left(\left(x_i^{\frac{1}{2r}}-x_i^{-\frac{1}{2r}}\right)F_L^{(2v)}\right)(x^{-1})~.
\end{equation}
\end{proof}

\subsection{Conditions for convergence}

\begin{proposition}\label{prop:2-comp-surgery} Let $L$ be an oriented, two-component link whose component $L_2$ is unknotted, and let $r\in\Z$.

Then, there are at most two vertices $v\in\{v_a,v_b\}\subset\mathrm{Vert}(\newtonalex)$ for which
\begin{equation}
    \mathcal{L}_{\frac{1}{r};x_2}\left(\left(x_2^{\frac{1}{2r}}-x_2^{-\frac{1}{2r}}\right)F_L^{(2v)}\right)(x_1,q)
\end{equation}
is a well-defined power series in $x_1$ and $q$.

In particular, there are two cases:
\begin{enumerate}
    \item
    If, for some $v\in\mathrm{Vert}(\newtonalex)$,
    \begin{equation*}
        (1,-r\,\mathrm{lk}(L_1,L_2))\in\mathrm{int}(\mathcal{C}_v^\vee)
    \end{equation*}
    then the Laplace transform may only give a well-defined power series in $x_1$ for the series $F_L^{(2v)}$. Moreover, if the pair $(L,2v)$ is nice, the resulting Laplace transform is in $\Z[q,q^{-1}][[x_1]]$.
    \item
    If, for some $v_a,v_b\in\mathrm{Vert}(\newtonalex)$,
    \begin{equation*}
        (1,-r\,\mathrm{lk}(L_1,L_2))\in\partial\mathcal{C}_{v_a}^\vee\cap\partial\mathcal{C}_{v_b}^\vee
    \end{equation*}
    then either series $F_L^{(2v_a)}$, $F_L^{(2v_b)}$ may be used for the Laplace transform, but the result may not give a well-defined element of $\Z((q))[[x_1]]$.
\end{enumerate}
Moreover, if $\mathrm{lk}(L_1,L_2)\neq 0$, there will be only finitely many $r$-values for which (2) holds.
\end{proposition}

We restrict our computations to links $L$ and vertices $v\in\Hnice$, for which the pair $(L,v)$ is nice.
However, as in \cref{rem:stratified}, we expect that we can define the Gukov--Manolescu series in greater generality, so we have written the statement of \cref{prop:2-comp-surgery} accordingly.

\begin{remark} Given a knot $K$, the expectation for its $F_K$-series is that, when defined, it satisfies $F_K(x,q)\in\Z[q,q^{-1}][[x]]$ only if $K$ is a fibered knot, and $F_K(x,q)\in\Z((q))[[x]]$ otherwise.

Let $L$ be a two-component, algebraically connected link; i.e. $\mathrm{lk}(L_1,L_2)\neq 0$. Assume, moreover, that the pair $(L,2v)$ is nice for every exposed vertex $v\in\mathrm{Vert}(\newtonalex)$. Then, the expectation above implies that there are at most $\left|\mathrm{Vert}(\newtonalex)\right|$ values of $r$ for which the partial $\frac{1}{r}$-surgery on $L$ yields a non-fibered knot, and infinitely many values of $r$ for which the partial $\frac{1}{r}$-surgery on $L$ yields a fibered knot.

\end{remark}

\begin{example}[Whitehead link] Consider the link $L=L5a1\{0\}$, also known as the Whitehead link. Its multivariable Alexander polynomial is
\begin{equation*}
    \Delta_L(x_1,x_2)= \left(x_1x_2\right)^{-\frac12} - x_1^\frac12x_2^{-\frac12}-x_1^{-\frac12}x_2^\frac12 + \left(x_1x_2\right)^\frac12~.
\end{equation*}
It has four exposed vertices, of which only $\left(-\frac12,-\frac12\right)$ is positively exposed (note that this must be the case since $L$ is a homogeneous braid link). Since $L$ is an alternating link, we conclude that all of the exposed vertices $v$ correspond to fibered faces. In this case, the cones $\mathcal{C}_v$ correspond to each of the four quadrants of $\R^2$.
\begin{table}[ht!]
    \centering
    \begin{tabular}{c|c}
        Vertex & Cone \\
        \hline
        $\left(-\frac12,-\frac12\right)$ & $\mathcal{C}_{\left(-\frac12,-\frac12\right)}=\mathrm{Cone}\{(1,0),(0,1)\} = (\R_{\geq 0})^2$ \\
        $\left(\frac12,-\frac12\right)$ & $\mathcal{C}_{\left(-\frac12,\frac12\right)}=\mathrm{Cone}\{(-1,0),(0,1)\} = \R_{\leq 0}\times \R_{\geq 0}$ \\
        $\left(-\frac12,\frac12\right)$ & $\mathcal{C}_{\left(\frac12,-\frac12\right)}=\mathrm{Cone}\{(1,0),(0,-1)\} = \R_{\geq 0}\times \R_{\leq 0}$ \\
        $\left(\frac12,\frac12\right)$ & $\mathcal{C}_{\left(\frac12,\frac12\right)}=\mathrm{Cone}\{(-1,0),(0,-1)\} = (\R_{\leq 0})^2$
    \end{tabular}
    \caption{The cones $\mathcal{C}_v$ associated to each of the exposed vertices $v$.}
    \label{tab:placeholder}
\end{table}

We can identify the cones with their duals, and we have $\mathcal{C}_v^\vee=\mathcal{C}_v$. Since $\mathrm{lk}(L_1,L_2)=0$, for every $r\in\Z$
\begin{equation*}
    (1,-r\,\mathrm{lk}(L_1,L_2))=(1,0)\in\partial\mathcal{C}_{\left(-\frac12,-\frac12\right)}\cap \partial \mathcal{C}_{\left(-\frac12,\frac12\right)}~.
\end{equation*}
Similarly,
\begin{equation*}
    (-r\,\mathrm{lk}(L_1,L_2),1)=(0,1)\in\partial\mathcal{C}_{\left(-\frac12,-\frac12\right)}\cap \partial \mathcal{C}_{\left(\frac12,-\frac12\right)}~.
\end{equation*}
Therefore, for every $r\in\Z\setminus \{0\}$, the partial $\frac{1}{r}$-Laplace surgery of $F_L^{(v)}$, on either component of $L$, gives a power series in $x_i$ (where $x_i$ is the variable associated to the component that we do not do surgery on), and each coefficient of the result can be expressed as an infinite weighted sum over the coefficients $f_{(n_1,n_2)}(q)$.
\end{example}

\begin{example}[$L11n128\{0\}$] Consider now $L=L11n128\{0\}$. Its multivariable Alexander polynomial is
\begin{equation*}
    \Delta_L(x_1,x_2)= x_1^{-1}x_2^{-2}-3x_1^{-1}x_2^{-1}+2x_1^{-1}-x_2^{-2}+2x_2^{-1}-1+2x_2-x_2^2 + 2x_1 - 3x_1x_2+x_1x_2^2~
\end{equation*}
and $\mathrm{lk}(L_1,L_2)=1$. The integral vertices are
\begin{equation*}
    (-1,-2), (0,-2), (1,2), (0,2)~.
\end{equation*}
These are identified using the link Floer homology of $L$ \cite{lfhcompute} and \cref{thm:HFL-fibered-classes}.
Of these, only $(-1,-2)$ is positively exposed.

\cref{fig:L11n128_0-cones} shows the cones $\mathcal{C}_v$ and their duals $\mathcal{C}_v^\vee$ for the fibered vertices. Since not all exposed vertices are associated to a fibered face, there is a gap in \cref{fig:L11n128_0-cones-dual}: the union of all dual cones does not cover the whole plane. As a consequence, there will be values of $r$ for which the partial $\frac{1}{r}$-Laplace transform does not give a well-defined power series.

The first component $L_1$ is an unknot, so we can apply the partial $\frac{1}{r}$-Laplace transform on the first component. Given $r\in\Z\setminus \{0\}$, we distinguish the following cases:
\begin{itemize}
    \item If $r\leq -1$, then $(-r\,\mathrm{lk}(L_1,L_2),1)=(-r,1)\in\mathrm{int}\left(\mathcal{C}_{(-1,-2)}^\vee\right)$, the interior of the blue cone.
    \item If $r=1$, then $(-1,1)\in\mathrm{int}\left(\mathcal{C}_{(0,-2)}^\vee\right)$, the interior of the yellow cone.
    \item If $r=2$, then $(-2,1)\in\partial\mathcal{C}_{(0,-2)}^\vee$, the boundary of the yellow cone.
    \item If $r>2$, there is no vertex $v$ for which $(-r,1)\in \mathcal{C}_v^\vee$.
\end{itemize}
We conclude that the partial $\frac{1}{r}$-Laplace transform on $x_1$ gives a well-defined element of $\Z[q,q^{-1}][[x_2]]$ for $r\in(-\infty,1)\cap \left(\Z\setminus\{0\}\right)$, as long as we use the corresponding $F_L^{(2v)}$ series.

\begin{figure}
\centering
\begin{subfigure}{.48\textwidth}\centering
\includegraphics[width=\linewidth]{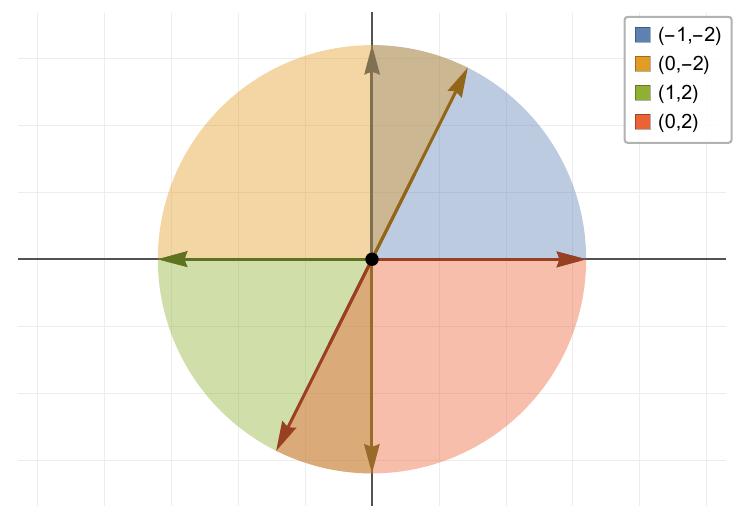}
\caption{The cones $\mathcal{C}_v$.}
\end{subfigure}\hfill
\begin{subfigure}{.48\textwidth}\centering
\includegraphics[width=\linewidth]{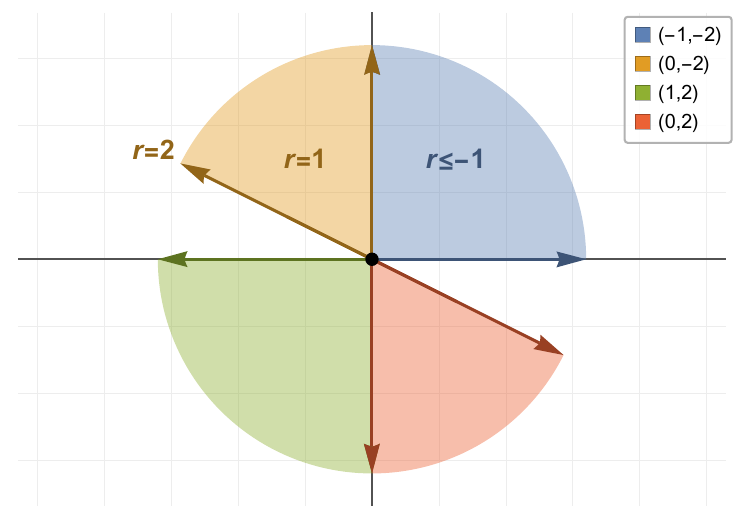}
\caption{The dual cones $\mathcal{C}_v^\vee$ and associated $r$-intervals.}
\label{fig:L11n128_0-cones-dual}
\end{subfigure}
\caption{The cones $\mathcal{C}_v$ and dual cones $\mathcal{C}_v^\vee$ for the fibered vertices of $L11n_{128}\{0\}$. We write the values of $r$ for which $(-r\,\mathrm{lk}(L_1,L_2),1)\in\mathcal{C}_v^\vee$ is satisfied.}
\label{fig:L11n128_0-cones}
\end{figure}

The $F_L$-series are:
\begin{equation}
\begin{split}
    F_L^{(-2,-4)}((x_1,x_2),q) &= q^{\frac32 }x_1^2x_2+3q^\frac32 x_1x_2^3 + q^\frac32 x_1^2x_2^2 + (q^\frac12+3q^\frac32)x_1^2x_2^3+\cdots~,
    \\
    F_L^{(0,-4)}((x_1,x_2),q) &= q^{-\frac12}x_2^2 + 2q^{-\frac12}x_2^3 + q^{-\frac12}x_1^{-1}x_2^2 + \left(2q^{-\frac12}-q^\frac12\right)x_1^{-1}x_2^3+\cdots~.
\end{split}
\end{equation}

Independently, using SnapPy \cite{SnapPy}, we checked that partial $-1$-surgery on $L_1$ is the knot exterior of $12n_{41}$ (in KnotInfo's convention). Indeed, we also get that
\begin{equation}
    \mathcal{L}_{-1;x_1}\left((x_1^{\frac12}-x_1^{-\frac12})F_L^{(-2,-4)}\right)(x_2)\doteq -q x_2^{5/2} - 3 q x_2^{7/2} + \left(-1 - 6 q - q^2\right) x_2^{9/2} +
    \cdots
\end{equation}
which agrees with the $F_K$ series of the mirror of $12n_{41}$ \cite{topologyfyi,FKCompute2026,OSSS25}.
On the other hand, $1$-surgery on $L_1$ is the knot exterior of $7_6$, which is again a fibered knot, and
\begin{equation}
     \mathcal{L}_{1;x_1}\left((x_1^{\frac12}-x_1^{-\frac12})F_L^{(0,-4)}\right)(x_2)\doteq -q x_2^{3/2} - 4 q x_2^{5/2} + \left(-3 - 10 q\right) x_2^{7/2} +
     \cdots\,,
\end{equation}
in agreement with previous computations.

\end{example}

\begin{proof}

Given \cref{prop:2-comp-surgery-convention-positively-exposed}, we fix an orientation of $L$ and consider any exposed vertex $v\in\mathrm{Vert}(\newtonalex)$. Writing
\begin{equation}
    F_L^{(2v)}((x_1,x_2),q)=\sum_{(n_1,n_2)\in\mathcal{C}_v} f_{(n_1,n_2)}(q)\,x_1^{n_1}x_2^{n_2}~,
\end{equation}
the partial $\frac{1}{r}$-Laplace transform on the variable $x_2$ sends $(x_2^{\frac{1}{2r}}-x_2^{-\frac{1}{2r}}) F_L^{(2v)}((x_1,x_2),q)$ to
\begin{equation}
    \sum_{(n_1,n_2)\in\mathcal{C}_v} f_{(n_1,n_2)}(q) \left(q^{-r\left(n_2+\frac{1}{2r}\right)^2}\,x_1^{n_1-r\,\mathrm{lk}(L_1,L_2)n_2-\frac{\mathrm{lk}(L_1,L_2)}{2}}-q^{-r\left(n_2-\frac{1}{2r}\right)^2}\,x_1^{n_1-r\,\mathrm{lk}(L_1,L_2)n_2+\frac{\mathrm{lk}(L_1,L_2)}{2}}\right).
\end{equation}
In order to get a well-defined power series in $x_1$,
\begin{equation}\label{eq:condition-n1n2-r}
    n_1-r\mathrm{lk}(L_1,L_2)n_2\pm \frac{\mathrm{lk}(L_1,L_2)}{2} > 0
\end{equation}
must hold for arbitrary $(n_1,n_2)\in\mathcal{C}_v$. Ignoring the scalar shift, a necessary condition for the convergence of the Laplace transform in $\Z((q))[[x_1]]$ is that $n_1-r\mathrm{lk}(L_1,L_2)n_2\geq 0$ or, equivalently,
\begin{equation}
    (1,-r\,\mathrm{lk}(L_1,L_2))\in\mathcal{C}_v^\vee \,,
\end{equation}
where $\mathcal{C}_v^\vee$ denotes the dual cone of $\mathcal{C}_v$. If, moreover,
\begin{equation}
    (1,-r\,\mathrm{lk}(L_1,L_2))\in \mathrm{int}\left(\mathcal C_v^\vee\right),
\end{equation}
convergence of the Laplace transform is guaranteed as a power series in $x_1$, where each $x_1$-coefficient can be obtained by a \emph{finite} weighted sum of $f_{(n_1,n_2)}(q)$ terms. In particular, for a  nice pair $(L,v)$, the Laplace transform on $F_L^{(v)}$ constructs a series in $\Z[q,q^{-1}][[x]]$.

Note that, for any $\ell\geq 1$,
\begin{equation}
    \bigcup_{v\in \mathrm{Vert}(\newtonalex)} \mathcal{C}_v=
    \bigcup_{v\in \mathrm{Vert}(\newtonalex)} \mathcal{C}_v^\vee=\R^\ell~.
\end{equation}
Therefore, given $r\in\Z\setminus \{0\}$, there is at least one exposed\footnote{All extremal vertices of a convex hull are exposed.} vertex $v$ for which $(1,-r\,\mathrm{lk}(L_1,L_2))\in\mathcal{C}_v^\vee$. Equivalently, to every extremal vertex $v\in \mathrm{Vert}(\newtonalex)$, we can assign two rational numbers $l_v$ and $u_v$ (which may be $\pm \infty$), such that if
\begin{equation}
    l_v\leq r\leq u_v
\end{equation}
then \eqref{eq:condition-n1n2-r} is satisfied.

The dual cone $\mathcal{C}_v^\vee$ consists of the functionals which attain their minima at $v$, and its interior $\mathrm{int}(\mathcal{C}_v^\vee)$ consists of the functionals which are minimized uniquely at $v$. Consequently, the set of all normal cones $\{\mathcal{C}_v^\vee\}_{v\in\mathrm{Vert}(\newtonalex)}$ may intersect one another only at their common boundaries. In particular, the intersection of the dual cones will be a collection of rays
\begin{equation}
    \bigcap_{v\in\mathrm{Vert}(\newtonalex)} \mathcal{C}_v^\vee = \bigcup_{v\in\mathrm{Vert}(\newtonalex)} \partial \mathcal{C}_v^\vee = \bigcup_{v\in\mathrm{Vert}(\newtonalex)} \{c\cdot \phi\mid \phi\in\partial \mathcal{C}_v^\vee,\, c\in \R_{>0}\}~,
\end{equation}
where each of the $\phi\in\partial \mathcal{C}_v^\vee$ is minimized at both $v$ and one of its neighboring extremal vertex, connected to $v$ via an edge.

\end{proof}

\section*{References}
\printbibliography[heading=none]

\appendix

\section{\texorpdfstring{$F_K$}{F\_K} series of non-fibered knots}

In this section, we experimentally compute the $F_K$ series for non-fibered knots that had not appeared in the literature before, by finding a link $L$ and an associated series $F_L^{(v)}$ for which the $\frac{1}{r}$-Laplace transform converges as a power series, in the sense of \cref{prop:2-comp-surgery}.

\subsection{Notation} Given a nice pair $(L,v)$ and its associated $F_L^{(v)}$ series, with coefficients $\{f_\mathbf{k}(q)\}_{\mathbf{k}\in\mathcal{C}_\frac{v}{2}}$, let $F_L^{(v),N}$ be the truncated series
\begin{equation*}
    F_L^{(v),N} = \sum_{\substack{\mathbf{k}\in\mathcal{C}_\frac{v}{2}\\|\mathbf{k}_i|\leq N, 1\leq i\leq \ell}} f_\mathbf{k}(q)\mathbf{x}^\mathbf{k}\,,
\end{equation*}
so that $F_L^{(v),N}\in\Z[q,q^{-1}][\mathbf{x},\mathbf{x}^{-1}]$ for every $N\geq 0$. In order to study the behaviour of the partial Laplace transform
\begin{equation*}
    \mathcal{L}_{\frac{1}{r};x_i}\left(\left(x_i^\frac{1}{2r}-x_i^{-\frac{1}{2r}}\right)F_L^{(v)}\right)(\hat{\mathbf{x}}_i,q)~,
\end{equation*}
for a fixed link $L$ and component $i$, we define
\begin{equation*}
    g_{\mathbf{j}}(r,N;q)\defeq \left[\mathcal{L}_{\frac1r;x_i}\left(\left(x_i^\frac{1}{2r}-x_i^{-\frac{1}{2r}}\right)F_L^{(v),N}\right)(\hat{\mathbf{x}}_i,q)\right]_{\hat{\mathbf{x}}_i^\mathbf{j}}
\end{equation*}
where the square brackets with the subscript $\hat{\mathbf{x}}_i^\mathbf{j}$ denote the selection of the coefficient associated to the monomial $\hat{\mathbf{x}}_i^\mathbf{j}$. Note that
\begin{equation*}
    g_{\mathbf{j}}(r,N;q)\in\Z[q,q^{-1}]
\end{equation*}
for every $N\geq 0$.

\subsection{Study of knots up to 9 crossings}

There are $20$ prime knots up to $9$ crossings that are neither fibered nor strongly quasipositive, and therefore are not covered by the results of \cite{Park21,OSSS25}. Using SnapPy \cite{SnapPy}, we can identify $13$ of them as the result of $\frac{1}{r}$-surgery on an unknot component of a $2$-component link. We gather our results in \cref{tab:non-fibered-knots}.

In practice, the Laplace transform seems to give a well-defined power series in $x$ and $q$ for all the surgeries in the Table. Moreover, if we define
\begin{equation*}
\begin{split}
    \psi_{a,b}(q) &\defeq \sum_{n\in\Z} \epsilon(n) q^{a n^2+bn}  \,,
\end{split}
\end{equation*}
where
\begin{equation*}
    \epsilon(n)=\begin{cases}
        1 &\text{if }n\geq 0\\
        -1 &\text{if }n<0
    \end{cases}\,,
\end{equation*}
we are able to match the first and second nonzero terms of the partial Laplace transform with a linear combination of $\psi_{a,b}(q)$ with coefficients in $\Q(q)$, up to overall $q$-power.

\begin{example} Consider the unoriented link $\mathrm{L}7a_3$. Using SnapPy, we are able to identify $S^3_{\frac{1}{2};L_1}(\mathrm{L}7a_3) = 7_3$ and $S^3_{-\frac{1}{2};L_1}(\mathrm{L}7a_3)=8_4$.

The multivariable Alexander polynomial of $L=\mathrm{L}7a_3$ is
\begin{equation*}
\begin{split}
    \Delta_L(x_1,x_2) &=  -x_1^{-1/2}x_2^{-3/2} + x_1^{-1/2}x_2^{-1/2} - x_1^{-1/2}x_2^{1/2} + x_1^{-1/2}x_2^{3/2} + x_1^{1/2}x_2^{-3/2} \\&- x_1^{1/2}x_2^{-1/2} + x_1^{1/2}x_2^{1/2} - x_1^{1/2}x_2^{3/2}
\end{split}
\end{equation*}
Since $\mathrm{lk}(L_1,L_2)=0$, we have that $(-r\,\mathrm{lk}(L_1,L_2),1)=(0,1) \in \mathcal{C}_{(-1/2,-3/2)}^\vee$ for all $r\in\Z\setminus \{0\}$. For the integral vertex $(-1/2,-3/2)$, we get the series
\begin{equation*}
\begin{aligned}
F_L^{(-1,-3)}((x_1,x_2),q)
={}&\left(-q^{3/2}x_2^{3/2}-q^{3/2}x_2^{5/2}
+\left(q^{5/2}-q^{3/2}\right)x_2^{7/2}+\cdots\right)x_1^{1/2} \\
&+\left(-q^{3/2}x_2^{3/2}
+\left(q^{5/2}-q^{3/2}-q^{1/2}\right)x_2^{5/2}
+\cdots\right)x_1^{3/2} \\
&+\left(-q^{3/2}x_2^{3/2}
+\left(q^{7/2}+q^{5/2}-q^{3/2}-q^{1/2}-q^{-1/2}\right)x_2^{5/2} + \cdots\right)x_1^{5/2}\\& +\cdots
\end{aligned}
\end{equation*}
for which
\begin{equation*}
\begin{split}
    \mathcal{L}_{\frac12;x_1}&\left(\left(x_2^\frac{1}{4}-x_2^{-\frac14}\right)F_{m(L)}^{(-1,-3)}\right)(x_2,q^{-1}) \doteq \left(1-q+q^3-q^6+q^{10}-q^{15}+\cdots\right)x_2^{3/2}
    \\ & +\left(1-q-q^2+q^3+q^4+q^5-q^6-q^7-q^8-q^9+q^{10}+\cdots\right)x_2^{5/2} \\
    & +\left(q^{-1}-2-q-2q^2+q^4+3q^5-2q^8-3q^9-q^{10}+\cdots\right)x_2^{7/2}+\cdots
\end{split}
\end{equation*}

We find
\begin{equation*}
\begin{split}
    &g_{3/2}(2,10;q) =  1-q+q^3-q^6+q^{10}+\cdots = \psi_{2,1}(q) \hspace{1.5em} \text{(mod }q^{200}) \\
    &g_{5/2}(2,10;q) \simeq \frac{-1+2\psi_{2,1}(q)}{1-q}\hspace{1.5em} \text{(mod }q^{199})
\end{split}
\end{equation*}

On the other hand,
\begin{equation*}
\begin{split}
    \mathcal{L}_{-\frac12;x_1}&\left(\left(x_2^\frac{1}{4}-x_2^{-\frac14}\right)F_L^{(-1,-3)}\right)(x_2,q) \doteq \left(1-q+q^3-q^6+q^{10}-q^{15}+\cdots\right)x_2^{3/2}
    \\ & +\left(1 - q + q^2 + q^3 - q^4 - q^5 - q^6 + q^7 + q^8 + q^9 + q^{10}+\cdots\right)x_2^{5/2} \\
    & +\left(1 - q + 2 q^2 - 3 q^4 - q^5 + q^6 + 3 q^7 + 2 q^8 - q^{10}+\cdots\right)x_2^{7/2}+\cdots
\end{split}
\end{equation*}
In this case,
\begin{equation*}
\begin{split}
    &g_{3/2}(-2,10;q) =  1-q+q^3-q^6+q^{10}+\cdots = \psi_{2,1}(q) \hspace{1.5em} \text{(mod }q^{209}) \\
    &g_{5/2}(-2,10;q) = 1 - q + q^2 + q^3 - q^4 +\cdots = \frac{\psi_{2,0}(q)-2q\,\psi_{2,1}(q)}{1-q} \hspace{1.5em} \text{(mod }q^{199})
\end{split}
\end{equation*}

\end{example}

\begin{table}
    \centering
    \makebox[\textwidth][c]{\begin{tabular}{ccccccc}
        Knot & Link $L$ & Surgery pres. & Vertex $v$ & $j$ & Closed form for $g_{j}(r,N;q)$ & Agreement \\
        \hline
        $6_1$ & $\mathrm{L}5a1$ & $S^3_{\frac12;L_1}(L\{0\})$ & $(-1,-1)$ & $1/2$ & $\psi_{2,1}(q)$ & $q^{209}$ \\
        & & & & $3/2$ & $(\psi_{2,0}(q)-2q\,\psi_{2,1}(q))/(1-q)$ & $q^{199}$ \\
        \hline
        $8_1$ & $\mathrm{L}5a1$ & $S^3_{\frac13;L_1}(L\{0\})$ & $(-1,-1)$ & $1/2$ & $\psi_{3,2}(q)$ & $q^{319}$ \\
        & & & & $3/2$ & $(\psi_{3,1}(q)-2q\,\psi_{3,2}(q))/(1-q)$ & $q^{310}$ \\
        \hline
        $8_4$ & $\mathrm{L}7a_{3}$ & $S^3_{-\frac{1}{2};L_1}(L\{0\})$ & $(-1,-3)$ & $3/2$ & $\psi_{2,1}(q)$ & $q^{209}$ \\
        & & & & $5/2$ & $(\psi_{2,0}(q)-2q\,\psi_{2,1}(q))/(1-q)$ & $q^{199}$ \\
        \hline
        $8_6$ & $\mathrm{L}7a_{1}$ & $S^3_{\frac{1}{2};L_1}(m(L\{0\}))$ & $(-1,-3)$ & $3/2$ & $\psi_{2,1}(q)$ & $q^{209}$ \\
        & & & & $5/2$ & $2\psi_{2,1}(q)$ & $q^{209}$ \\
        \hline
        $8_8$ & $\mathrm{L}7a_{1}$ & $S^3_{-\frac{1}{2};L_1}(L\{0\})$ & $(-1,-3)$ & $3/2$ & $\psi_{2,1}(q)$ & $q^{209}$ \\
        & & & & $5/2$ & $2\psi_{2,1}(q)$ & $q^{209}$ \\
        \hline
        $8_{11}$ & $\mathrm{L}8a_{2}$ & $S^3_{\frac{1}{2};L_1}(m(L\{0\}))$ & $(-1,-3)$ & $3/2$ & $\psi_{2,1}(q)$ & $q^{209}$ \\
        & & & & $5/2$ & $((4-2q)\psi_{2,1}(q)-\psi_{2,0}(q))/(1-q)$ & $q^{199}$ \\
        \hline
        $8_{13}$ & $\mathrm{L}8a_{4}$ & $S^3_{\frac{1}{2};L_1}(m(L\{0\}))$ & $(-1,-3)$ & $3/2$ & $\psi_{2,1}(q)$ & $q^{209}$ \\
        & & & & $5/2$ & $((4-2q)\psi_{2,1}(q)-\psi_{2,0}(q))/(1-q)$ & $q^{199}$ \\
        \hline
        $8_{14}$ & $\mathrm{L}10n_{20}$ & $S^3_{-\frac12;L_1}(L\{0\})$ & $(-1,-3)$ & $3/2$ & $\psi_{2,1}(q)$ & $q^{209}$ \\
        & & & & $5/2$ & $((6-2q)\psi_{2,1}(q)-2\psi_{2,0}(q))/(1-q)$ & $q^{199}$ \\
        \hline
        $9_{12}$ & $\mathrm{L}8a_{4}$ & $S^3_{-\frac12;L_1}(L\{0\})$ & $(-1,-3)$ & $3/2$ & $\psi_{2,1}(q)$ & $q^{209}$ \\
        & & & & $5/2$ & $((2-4q)\psi_{2,1}(q)+\psi_{2,0}(q))/(1-q)$ & $q^{199}$ \\
        \hline
        $9_{14}$ & $\mathrm{L}8a_{2}$ & $S^3_{-\frac{1}{2};L_1}(L\{0\})$ & $(-1,-3)$ & $3/2$ & $\psi_{2,1}(q)$ & $q^{209}$ \\
        & & & & $5/2$ & $((2-4q)\psi_{2,1}(q)+\psi_{2,0}(q))/(1-q)$ & $q^{199}$ \\
        \hline
        $9_{15}$ & $\mathrm{L}8a_{1}$ & $S^3_{\frac12;L_1}(m(L\{0\}))$ & $(-1,-3)$ & $3/2$ & $\psi_{2,1}(q)$ & $q^{209}$ \\
        & & & & $5/2$ & $4\psi_{2,1}(q)$ & $q^{209}$ \\
        \hline
        $9_{19}$ & $\mathrm{L}8a_{1}$ & $S^3_{-\frac12;L_1}(L\{0\})$ & $(-1,-3)$ & $3/2$ & $\psi_{2,1}(q)$ & $q^{209}$ \\
        & & & & $5/2$ & $4\psi_{2,1}(q)$ & $q^{209}$ \\
        \hline
        $9_{37}$ & $\mathrm{L}10n_{33}$ & $S^3_{\frac12;L_1}(m(L\{0\}))$ & $(-1,-3)$ & $3/2$ & $\psi_{2,1}(q)$ & $q^{209}$ \\
        & & & & $5/2$ & $((6-4q)\psi_{2,1}(q)-\psi_{2,0}(q))/(1-q)$ & $q^{209}$ \\
    \end{tabular}}
    \caption{Non-fibered knots that can be obtained as partial $\frac{1}{r}$-surgery on the unknot component of a $2$-component link $L$. We compute the associated polynomials $g_{j}(r,N;q)$ from the series $F_L^{(v),N}$ for $N=10$, and we normalize them so their $q$-leading term is $1$. We record a closed formula for the first and second non-zero values of $g_{j}(r,10;q)$, which have been checked up to the $q$-degree recorded in the ``Agreement'' column.
    We orient $L$ and order its components following LinkInfo's convention \cite{knotinfo}, and use $m(L)$ to denote the mirror of $L$.}
    \label{tab:non-fibered-knots}
\end{table}

\end{document}